\documentclass[11pt, reqno]{amsart}

\usepackage[T1]{fontenc}
\usepackage[latin1]{inputenc}
\usepackage[french, english]{babel}
\usepackage{amsthm}
\usepackage{amscd}
\usepackage{amssymb,amsmath,bbm}
\usepackage[all]{xy}
\usepackage{mathrsfs}
\usepackage[left=0.0cm,right=0.0cm,top=2.7cm,bottom=3cm]{geometry}
\usepackage{hyperref}
\usepackage{verbatim}
\usepackage{enumitem}

\theoremstyle{plain}
\newtheorem{theorem}{Theorem}[section]

\newtheorem{lemma}[theorem]{Lemma}
\newtheorem{proposition}[theorem]{Proposition}
\newtheorem{prop}[theorem]{Proposition}
\newtheorem{corollary}[theorem]{Corollary}
\newtheorem{cor}[theorem]{Corollary}

\newtheorem*{theorem*}{Theorem}
\newtheorem{definition}[theorem]{Definition}
\newtheorem{conjecture}[theorem]{Conjecture}
\newtheorem{example}[theorem]{Example}

\theoremstyle{remark}
\newtheorem{remark}[theorem]{Remark}

\usepackage[usenames,dvipsnames]{color}
\newcommand{\anto}[1]{{\color{ForestGreen}   [#1]}}
\newcommand{\joaquin}[1]{{\color{blue}  [#1]}}
\newcommand{\francesco}[1]{{\color{violet}  [#1]}}

\newcommand{\mycomment}[1]{%
}%

\usepackage{tikz}
\usetikzlibrary{calc}

\def\Q{{\bf Q}}

\def\Z{{\bf Z}}
\def\C{{\bf C}}
\def\N{{\bf N}}
\def\R{{\bf R}}
\def\F{{\bf F}}

\def\H{{H}}

\def\Qbar{{\overline{{\bf Q}}}}

\def\SL{{\mathrm{SL}}}

\def\zp{{\Z_p}}

\def\qp{{\Q_p}}

\newcommand{\ind}{\mathrm{ind}}
\def\1{\mathbf{1}}

\def\Hom{\mathrm{Hom}}

\def\Sym{\mathrm{Sym}}

\def\A{\mathbf{A}}
\def\Af{{\bf A}_{f}}

\def\epsilon{\varepsilon}
\def\det{\mathrm{det}}
\def\Sh{{\operatorname{Sh}}}

\def\SU{\mathrm{SU}}
\def\GSp{{\mathrm{GSp}}}

\def\PGSp{{\mathrm{PGSp}}}
\def\Sp{{\mathrm{Sp}}}
\def\GL{\mathrm{GL}}

\def\G2{{\mathrm{G}_2}}
\def\G{\mathbf{G}}
\def\H{\mathbf{H}}

\def\B{\mathbf{B}}
\def\T{\mathbf{T}}

\def\matrix#1#2#3#4{{\big(\begin{smallmatrix}#1&#2\\ #3&#4\end{smallmatrix}\big)}}

\title{Algebraic cycles and functorial lifts from $\mathbf{G}_2$ to $\mathbf{PGSp}_6$}

\author{Antonio Cauchi}
\address{A.C.: Dept. of Mathematics, School of Science, Tokyo Institute of Technology, 2-12-1 Ookayama, Meguro-ku, Tokyo, 152-8551 Japan.}
\email{cauchi.a.aa@m.titech.ac.jp}
\author{Francesco Lemma}
\address{F.L.: Universit\'e Paris Cit\'e, CNRS, IMJ--PRG, b\^atiment Sophie Germain, case 7012, 75205 Paris Cedex 13, France}
\email{francesco.lemma@imj-prg.fr}
\author{Joaqu\'in Rodrigues Jacinto}
\address{J.R.J.: Aix--Marseille Universit\'e, Campus de Luminy, Avenue de Luminy, Case 930, 13288 Marseille Cedex 9, France}
\email{joaquin.rodrigues-jacinto@univ-amu.fr}

\makeatletter
\def\@tocline#1#2#3#4#5#6#7{\relax
  \ifnum #1>\c@tocdepth 
  \else
    \par \addpenalty\@secpenalty\addvspace{#2}%
    \begingroup \hyphenpenalty\@M
    \@ifempty{#4}{%
      \@tempdima\csname r@tocindent\number#1\endcsname\relax
    }{%
      \@tempdima#4\relax
    }%
    \parindent\z@ \leftskip#3\relax \advance\leftskip\@tempdima\relax
    \rightskip\@pnumwidth plus4em \parfillskip-\@pnumwidth
    #5\leavevmode\hskip-\@tempdima
      \ifcase #1
       \or\or \hskip 1em \or \hskip 2em \else \hskip 3em \fi%
      #6\nobreak\relax
    \dotfill\hbox to\@pnumwidth{\@tocpagenum{#7}}\par
    \nobreak
    \endgroup
  \fi}
\makeatother
\usepackage{hyperref}

\thanks{A. Cauchi was supported
by the European Research Council (ERC) under the European Union's
Horizon 2020 research and innovation programme (grant agreement
No. 682152) as well as by the NSERC grant RGPIN-2018-04392 and Concordia Horizon postdoc fellowship n.8009. F. Lemma was supported by the ANR Ferplay and the ANR ClapClap. J. Rodrigues
Jacinto was financially supported by the ERC-2018-COG-818856-HiCoShiVa and by the project ANR-19-CE40-0015
COLOSS. 
}

\begin{document}
\maketitle
\selectlanguage{english}

\setcounter{tocdepth}{1}

\begin{abstract} We study instances of Beilinson--Tate conjectures for automorphic representations of $\PGSp_6$ whose Spin $L$-function has a pole at $s=1$. We construct algebraic cycles of codimension three in the Siegel--Shimura variety of dimension six and we relate its regulator to the residue at $s=1$ of the $L$-function of certain cuspidal forms of $\PGSp_6$. Using the exceptional theta correspondence between the  split group of type $G_2$ and $\PGSp_6$ and assuming the non-vanishing of a certain archimedean integral, this allows us to confirm a conjecture of Gross and Savin on rank $7$ motives of type $G_2$.
\end{abstract}

\tableofcontents

\section{Introduction}

In this paper we establish a connection between algebraic cycles in Siegel sixfolds and the residue at $s=1$ of Spin $L$-functions of automorphic representations of $\mathrm{GSp}_6$, as predicted by conjectures of Beilinson and Tate. Moreover, we exploit an exceptional theta correspondence between the split group of type $G_2$ and $\mathrm{PGSp}_6$ to answer a question of Gross and Savin.

\subsection{Motivation}

Let $\pi = \pi_\infty \otimes \pi_f$ be a cohomological cuspidal automorphic representation of $\mathrm{PGSp}_6(\A)$, let $M(\pi_f)$ denote the Spin Chow motive with coefficients in a number field $L$ conjecturally attached to $\pi$ and let $L(s, M(\pi_f)(3))$ be its Hasse-Weil $L$-function. Let $$
r_\mathcal{H} : H^1_\mathcal{M}(M(\pi_f)(4)) \oplus N(M(\pi_f)(3)) \rightarrow H^1_\mathcal{H}(M(\pi_f)(4))
$$
denote Beilinson-Deligne regulator. Here $H^1_\mathcal{M}(M(\pi_f)(4))$ denotes the first motivic cohomology group of $M(\pi_f)(4)$, the group $N(M(\pi_f)(3))$ denotes algebraic cycles in $M(\pi_f)(3)$ up to homological equivalence and $H^1_\mathcal{H}(M(\pi_f)(4))$ denotes the first absolute Hodge cohomology group of $M(\pi_f)(4)$. 

\begin{conjecture}(Beilinson--Tate) \label{ConjectureBT}
\begin{enumerate}
\item  The map $r_\mathcal{H}$ induces an isomorphism $$
(H^1_\mathcal{M}(M(\pi_f)(4)) \oplus N(M(\pi_f)(3))) \otimes_\Q \R \rightarrow H^1_\mathcal{H}(M(\pi_f)(4)),
$$
\item $\mathrm{ord}_{s=0} L(s, M(\pi_f)(3)) = \dim_L H^1_\mathcal{M}(M(\pi_f)(4)),$\\
\item $- \mathrm{ord}_{s = 1} L(s, M(\pi_f)(3)) = \dim_L N(M(\pi_f)(3))$\\
\item $\det(\mathrm{Im} \, r_\mathcal{H}) = L^*(1, M(\pi_f)(3)) \mathcal{D}(M(\pi_f)(4))$, where $\mathcal{D}(M(\pi_f)(4))$ denotes the Deligne $L$-structure of $\det(H^1_\mathcal{H}(M(\pi_f)(4))$.
\end{enumerate}
\end{conjecture}

In \cite{CLR}, we studied the contribution of the motivic cohomology to this conjecture. This corresponds to the case where $L(s, M(\pi_f)(3))$ does not have a pole at $s=1$. In this article, we focus on the contribution of algebraic cycles, which corresponds to the case where $L(s, M(\pi_f)(3))$ has a simple pole at $s=1$. \\

The $\ell$-adic \'etale realization $M_\ell(\pi_f)$ of $M(\pi_f)$ is expected to be a $\mathrm{GL}_8(\overline{\Q}_\ell)$-valued Galois representation factoring through the Spin representation $\mathrm{Spin}: \mathrm{Spin}_7(\overline{\Q}_\ell) \rightarrow \mathrm{GL}_8(\overline{\Q}_\ell)$. If $L(s, M(\pi_f)(3))$ has a pole at $s = 1$, Conjecture \ref{ConjectureBT} (3) implies the existence of an invariant vector in this $8$-dimensional Galois representation. As the stabilizer in $\mathrm{Spin}_7(\overline{\Q}_\ell)$ of a generic vector in the Spin representation is the exceptional group $G_2(\overline{\Q}_\ell)$, by   Langlands reciprocity principle,  $\pi$ should be a functorial lift from a group $G$ of type $G_2$. In fact, we have $\mathrm{Spin}_{|_{G_2}}=\mathrm{Std} \oplus \mathbf{1}$, where Std denotes the standard representation of $G_2$ and $\mathbf{1}$ denotes the trivial representation. Then, if $\sigma$ is a cuspidal automorphic representation of $G(\A)$ lifting to $\pi$, Gross and Savin \cite{GrossSavin} conjectured that the motive $M(\pi_f)$ decomposes as the direct sum of the rank $7$ motive $M(\sigma_f)$ attached to $\sigma$ and the rank $1$ trivial motive generated by the class given in Conjecture \ref{ConjectureBT}. Moreover, inspired by local calculations, they conjectured that this class should arise from a Hilbert modular threefold.  \\

\subsection{Main results}

Let $F$ denote a real \'etale quadratic $\Q$-algebra, i.e., $F$ is either a quadratic extension of $\Q$ or $\Q \times \Q$. Associated to the totally real \'etale cubic algebra $E = \Q \times F$ of $\Q$ there is a Hilbert modular threefold $\Sh_{\H}/\Q$, with underlying reductive group $\H= \{ g \in \mathrm{Res}_{E / \Q} \GL_{2, E} \; | \; \det(g)\in \G_m \}$. The group $\H$ embeds naturally into $\G = \GSp_6$ and one has a closed embedding $\iota: \Sh_\H \hookrightarrow \Sh_\G$ of codimension $3$ in the Shimura variety attached to $\G$, which is the Siegel variety of dimension $6$. Let $V^\lambda$ be the irreducible algebraic representation of $\G$ of highest weight $\lambda = (\lambda_1, \lambda_2, \lambda_3, c)$ (cf. \S \ref{sectPreliminaries} for notations on algebraic representations). The representation $V^\lambda$ contains the trivial $\H$-representation if and only if $c=0$ and $\lambda_1 = \lambda_2 + \lambda_3$. When this holds $\iota^* V^\lambda$ contains  $\lambda_2 - \lambda_3 +1$ copies of the trivial representation of $\H$, which we index by the values $\lambda_2 \geq \mu \geq \lambda_3$. Then, for any such $\mu$, the cycle $\Sh_\H$ of $\Sh_\G$ induces a class
\[ \mathcal{Z}_{\H, \mathcal{M}}^{[\lambda, \mu]} \in H^6_{\mathcal{M}}(\Sh_\G, \mathscr{V}^\lambda_\mathcal{M}(3)), \]
where $\mathscr{V}^\lambda_\mathcal{M}$ is the Chow local system associated to $V^\lambda$ and $H^6_{\mathcal{M}}(\Sh_\G, \mathscr{V}^\lambda_\mathcal{M}(3))$ is the motivic cohomology group of $\Sh_\G$ with coefficients in $\mathscr{V}_\mathcal{M}^\lambda(3)$. 
We denote by $\mathcal{Z}^{[\lambda, \mu]}_{\H, \mathcal{H}} \in H^7_\mathcal{H}(\Sh_\G, \mathscr{V}^\lambda_\mathcal{H}(4))$, resp. $\mathcal{Z}_{\H, B}^{[\lambda, \mu]} \in H^6_B(\Sh_\G, \mathscr{V}^\lambda_B(3))$ the image of $\mathcal{Z}^{[\lambda, \mu]}_{\H, \mathcal{M}}$ in absolute Hodge cohomology, resp. Betti cohomology (see Definition \ref{HodgeCycleClass} for the precise definition of $\mathcal{Z}^{[\lambda, \mu]}_{\H, \mathcal{H}}$). Let $\pi$ be a cuspidal automorphic representation of $\PGSp_6(\A)$  whose archimedean component belongs to the discrete series $L$-packet of $V^\lambda$ and has Hodge type $(3,3)$. For a cusp form $\Psi = \Psi_\infty \otimes \Psi_f$ in the space of $\pi$, whose archimedean component $\Psi_\infty$ is a highest weigth vector in the minimal $K$-type of $\pi_\infty$, we have a vector valued harmonic differential form $\omega_\Psi$ whose cohomology class $[\omega_\Psi]$ is an element of $H^6_{dR,c}(\Sh_\G, \mathscr{V}_{dR}^\lambda)$. Poincar\'e duality induces maps
\[\langle \cdot , [\omega_\Psi] \rangle_B : H^6_B(\Sh_\G, \mathscr{V}_B^\lambda(3))  \to \C, \] 
\[\langle \cdot , [\omega_\Psi] \rangle_\mathcal{H} : H^7_\mathcal{H}(\Sh_\G, \mathscr{V}_\mathcal{H}^\lambda(4))  \to \C. \]
The pairings $\langle \mathcal{Z}_{\H, B}^{[\lambda, \mu]}, [\omega_\Psi] \rangle_B$ and $\langle \mathcal{Z}_{\H, \mathcal{H}}^{[\lambda, \mu]}, [\omega_\Psi] \rangle_\mathcal{H}$ are computed in terms of the residue of a certain adelic integral of Rankin-Selberg type considered in \cite{Pollack-Shah}. In \textit{loc. cit.} it is shown that, if $\pi$ supports certain Fourier coefficients associated to $F$, then the local factors at unramified places $v$ of this integral represent the degree $8$ local spin $L$-function $L(s, \pi_v, \mathrm{Spin})$ of $\pi_v$. The following result gives evidence for Conjecture \ref{ConjectureBT} for the motive associated to $\pi$.

\begin{theorem}[Theorem \ref{theoremcyclebetti1}] \label{IntroTheo1}
Let $\pi = \pi_\infty \otimes \pi_f$ be a cuspidal automorphic representation of $\mathrm{PGSp}_6(\A)$ such that $\pi_\infty$ is a discrete series of Hodge type $(3,3)$ in the discrete series $L$-packet of $V^\lambda$. Then 
\begin{align*}
         \langle  \mathcal{Z}_{\H, B}^{[\lambda, \mu]} , [\omega_\Psi] \rangle_B = \langle  \mathcal{Z}_{\H, \mathcal{H}}^{[\lambda, \mu]} , [\omega_\Psi] \rangle_\mathcal{H} = C \cdot 
       \mathrm{Res}_{s=1} \left (  \mathcal{I}_S(\Phi, \Psi^{[\lambda, \mu]} , s) L^S(s,\pi , {\rm Spin}) \right),
       \end{align*}
where $C$ is an explicit non-zero constant independent of $\pi$, $S$ is a sufficiently large set of places containing the ramified and archimedean places, $\Psi^{[\lambda, \mu]} = A^{[\lambda, \mu]} \cdot \Psi$ for some weight lowering operator $A^{[\lambda, \mu]}$ defined in Proposition \ref{period}, $\Phi$ is a Schwartz-Bruhat function and $\mathcal{I}_S(\Phi, \Psi^{[\lambda, \mu]} , s)$ is the integral defined in Theorem \ref{Pollackshahongsp6}. 
\end{theorem}

\begin{remark}
We point out that, according to \cite[Proposition 12.1]{Gan-Gurevich} there exist a Schwartz-Bruhat function $\Phi$ and a vector $\Psi \in \pi$ such that $\mathcal{I}_S(\Phi, \Psi , 1)$ is non-zero. However we do not know if this holds for $\Psi_\infty$ in the minimal $K$-type of $\pi_\infty$. 
 Moreover, one can show that there exists a cusp form $\tilde{\Psi} \in \pi$, which coincides with $\Psi$ at the archimedean place and away from $S$, such that \[ \mathcal{I}_S(\Phi, \tilde{\Psi}^{[\lambda, \mu]}, s) = \mathcal{I}_\infty(\Phi_\infty, \Psi_\infty^{[\lambda, \mu]}, s).\]
Although we have not been able to calculate it, we expect that for a natural choice of $\Phi_\infty$ the archimedean integral $\mathcal{I}_\infty(\Phi_\infty, \Psi^{[\lambda, \mu]}_\infty , s)$ is the Gamma factor of the Spin motive attached to $\pi$ by the rule of Serre, and hence holomorphic and non-zero at $s=1$.
\end{remark}

As a corollary of this theorem, one can deduce, under the additional assumption that $\pi$ is the Steinberg representation at a finite place, a weak version of Conjecture \ref{ConjectureBT}(1) (Corollary \ref{theoremcyclebetti2}) and Conjecture \ref{ConjectureBT}(3) (Corollary \ref{cor:tateconjecture}). \\

When $\mathrm{Res}_{s=1}  L^S(s,\pi , {\rm Spin})$ is non-zero then (cf. \cite[Theorem 1.1]{Gan-SavinExceptionalSW}) $\pi$ is a weak functorial lift of a cuspidal automorphic representation $\sigma$ of an exceptional group of type $G_2$. Moreover (cf. Proposition \ref{equivalences}), we have \[\mathrm{Res}_{s=1}  L^S(s,\pi , {\rm Spin}) = L^S(1, \sigma, {\rm Std})\mathrm{Res}_{s=1}  \zeta^S(s).\]
Hence, up to controlling the value of the archimedean integral at $s=1$, Theorem \ref{IntroTheo1} above gives a cohomological formula for the critical value $L^S(1, \sigma, {\rm Std})$. \\

Our second main result concerns the program of Gross and Savin on rank seven motives of Galois type $G_2$. The first step towards the conjecture of Gross and Savin was done by Kret and Shin in \cite{KretShin}, where they more generally constructed $\mathrm{GSpin}$-valued Galois representations associated to cohomological cuspidal automorphic forms of symplectic groups. Moreover, based on the calculations of \cite{GrossSavin}, they verified (\cite[Theorem 11.1]{KretShin}) that, for suitable automorphic representations of $\PGSp_6(\A)$ in the image of the exceptional theta correspondence from the compact form $G_2^c$ of type $G_2$, the image of their Galois representation lies actually in $G_2(\overline{\Q}_\ell)$. More precisely, let $\rho_\pi$ be the ${\rm Spin}_7(\overline{\Q}_\ell)$-valued Galois representation attached to $\pi$. Assuming that $\pi$ is a non-trivial small theta lift of $\sigma$, we have 
\begin{equation} \label{eqintrodesc}
\mathrm{Spin} \circ \rho_\pi = \mathrm{Std} \circ \rho_\sigma \oplus \mathbf{1}
\end{equation}
where $\mathrm{Std} \circ \rho_\sigma$ is the standard Galois representation attached to $\sigma$ and $\mathbf{1}$ denotes the one dimensional trivial representation.

\begin{remark}
Technically speaking, only the dual pair $(G_2^c, \PGSp_6)$ is considered in \cite{GrossSavin}, but their conjecture also applies to the dual pair $(G_2, \PGSp_6)$. Using the results of \cite{KretShin} and the study of the exceptional theta correspondence for $(G_2, \PGSp_6)$ (see Theorem \ref{IntroTheo3} below), we construct (Theorem \ref{thmonsplitGalois}), under some assumptions, Galois representations associated to cohomological cuspidal automorphic representations $\sigma$ of $G_2(\A)$, which sit in a decomposition as that of Equation \eqref{eqintrodesc}.
\end{remark}


\begin{theorem} [Theorem \ref{TheoGS}] \label{TheoIntro2} Let $\sigma$ be an irreducible cuspidal automorphic representation of $G_2^c(\A)$ or $G_2(\A)$ such that the big theta lift $\Theta(\sigma)$ to $\PGSp_6(\A)$ has an irreducible subquotient $\pi=\bigotimes'_v \pi_v$, which is a cuspidal automorphic representation such that $\pi_\infty$ is cohomological for $V$ as above and $\pi_p$ is the Steinberg representation for some prime number $p$. Assume that the integral $\mathcal{I}_S(\Phi, \Psi^{[\lambda, \mu]} , 1)$ is non-zero for some $\Phi$ and $\Psi^{[\lambda, \mu]}$ as above. Then, the trivial representation $\mathbf{1}$ in \eqref{eqintrodesc} is generated by the \'etale realization of $\mathcal{Z}_{\H, \mathcal{M}}^{[\lambda, \mu]}$.
\end{theorem}

\begin{remark}
    Note that the archimedean part $\pi_\infty$ of $\pi$ is not necessarily of Hodge type $(3,3)$. However, it is one of the main results of \cite{KretShin} that the $L$-packet of $\pi$ is stable at infinity. In particular, there exists a cuspidal automorphic representation $\pi^{3,3}=\pi_\infty^{3,3} \otimes \pi_f$ whose archimedean part is cohomological and of Hodge type $(3,3)$ and whose non-archimedean part is equivalent to $\pi_f$. In the integral appearing in the statement of Theorem \ref{IntroTheo1}, the archimedean part of the cusp form $\Psi^{[\lambda, \mu]}$ is a suitable vector in the minimal $K$-type of $\pi^{3,3}_\infty$. 
\end{remark}

\begin{remark}
    In Proposition \ref{prop:assumptions} we give a list of cases where $\sigma$ is known to have a small theta lift $\pi=\bigotimes'_v \pi_v$ of $\sigma$ to $\PGSp_6(\A)$ which is a cuspidal automorphic representation such that $\pi_\infty$ is cohomological for $V$ as above and $\pi_p$ is the Steinberg representation for some prime number $p$, as in the previous theorem.
\end{remark}

We conclude this introduction explaining a result which provides cases where Theorem \ref{TheoIntro2} can be applied and which has its own interest. Indeed, note that a necessary condition for the integral $\mathcal{I}_S(\Phi, \Psi^{[\lambda, \mu]} , 1)$ to be non-zero, is that $\pi$ supports a rank $2$ Fourier coefficient associated to $F$. By a result of Gan \cite[Theorem 3.1]{ganmultcubic}, every cuspidal automorphic representation $\sigma$ of $G_2(\A)$ supports a Fourier coefficient associated to an \'etale cubic algebra $E$.

\begin{theorem} [{Theorem  \ref{cuspidality}, Proposition \ref{prop:comparisonbetweenFC}}] \label{IntroTheo3}
    Let $\sigma$ be a cuspidal automorphic representation of $G_2(\A)$. Assume that \begin{itemize}
        \item $\sigma$ is not globally generic;
        \item $\sigma_p$ is generic at some finite place $p$.
    \end{itemize}  Then the big theta lift $\Theta(\sigma)$ is cuspidal. Moreover $\Theta(\sigma)$ supports a rank $2$ Fourier coefficient associated to $F$ (and is in particular non-zero) if and only if $\sigma$ supports a Fourier coefficient associated to $\Q \times F$.
\end{theorem}

\subsection{Overview of the proofs} The main difficulty for calculating the pairing of Theorem \ref{IntroTheo1} between the motivic class and the cohomology class $[\omega_\Psi]$ resides on the fact that the first class is constructed from the decomposition into irreducible components of the restriction of $V$ to the subgroup $\H$, while the test vector is constructed from its restriction to the maximal compact subgroup $\mathrm{U}(3)$ of $\G(\R)$. One needs to carefully study the relationship between these two different decompositions (Theorem \ref{theo:non-vanishing}). As a consequence we get a formula for the pairing in terms of a period integral (Proposition \ref{period} and Proposition \ref{period2}). These adelic integrals are in turn related to the residue of the partial Spin $L$-function of $\pi$ by means of the work of Pollack and Shah (Proposition \ref{periodvsresidue}), which allows to conclude the proof. Theorem \ref{TheoIntro2} follows basically from Theorem \ref{IntroTheo1} and \ref{IntroTheo3}. The proof of Theorem \ref{IntroTheo3} goes as follows. We first prove (Theorem \ref{cuspidality} and Corollary \ref{corocusptheta}) that $\sigma$ lifts to a cuspidal representation using the tower of exceptional correspondences for $G_2$ studied in \cite{Ginzburg-Rallis-Soudry2}, which reduces the problem to the vanishing of certain automorphic period integrals. Finally, we establish (Proposition \ref{prop:comparisonbetweenFC}) a correspondence between Fourier coefficients of $\sigma$ and its theta lift, which in particular implies the non-vanishing of the latter.

\subsection{Structure of the manuscript} In section \ref{sectPreliminaries} we fix notation, conventions, and basic results that will be useful in the body of the article. In particular, we prove that, under some mild assumptions, the localization at a maximal ideal of the Hecke algebra of the cohomology of the Siegel sixfold is cuspidal and concentrated in the middle degree. We also introduce Absolute Hodge cohomology and compute the dimension of its $\pi_f$-isotypical component. In section \ref{section:motivicclass} we explain the construction of the motivic class $\mathcal{Z}_{\mathcal{M}}^{[\lambda, \mu]}$ and its realizations. In Section \ref{Section4} we construct the harmonic differential form $\omega_\Psi$ associated to a suitable cuspidal form $\Psi$ in the space of $\pi$ and we prove our first main result concerning the calculation of the pairing between the motivic class and the cohomology class $[\omega_\Psi]$.  In Section \ref{Section5}, we use the results of Pollack and Shah to relate the pairing to the residue of the Spin $L$-function. Sections \ref{Section6} and \ref{sectiononcuspidalityandfouriercoefficients} are devoted to the study of the exceptional theta correspondence between $G_2$ and $\mathrm{PGSp}_6$ and contain the proof of Theorem \ref{IntroTheo3}. Finally, in Section \ref{sec:cycleclassformulaandstandardlvalues} we relate the pairing to a critical value of the standard $L$-function of $G_2$. We also deduce from the work of Kret and Shin the existence of Galois representations attached to certain cuspidal representations of $G_2$ and we conclude with a proof of Theorem \ref{TheoIntro2}.

\subsection{Acknowledgements}

We would like to heartily thank David Ginzburg for sharing with us the proof of Lemma \ref{lem:periodsl2n}. We also thank Nadir Matringe,  Aaron Pollack and Armando Gutierrez for fruitful exchanges. We thank Marc-Hubert Nicole and Vincent Pilloni for comments on an earlier draft of the article. Finally, we thank the anonymous referee for his/her careful reading of the manuscript, comments and corrections which have significantly improved the content of this article.

\section{Preliminaries} \label{sectPreliminaries}

\subsection{Algebraic groups and algebraic representations} Let $\psi$ denote an antisymmetric non-degenerate bilinear form on a finite dimensional $\Q$-vector space $V$. The symplectic group $\GSp(V, \psi)$ is the $\Q$-group scheme defined by
$$
\GSp(V, \psi)=\{g \in \mathrm{GL}(V)\,|\, \forall v,w \in V, \psi(gv,gw)=\nu(g)\psi(v,w), \nu(g) \in \mathbf{G}_m\}
$$
Then $\nu: \GSp(V, \psi) \rightarrow \mathbf{G}_m$ is a character. Let $I_n$ denote the identity matrix of size $n$. When $V$ is the $\Q$-vector space $\Q^{2n}$ endowed with the bilinear form whose matrix is $J={ \matrix 0 {I_n} {-I_n} 0}$, we let $\GSp_{2n}$ denote $\GSp(\Q^{2n}, J)$ and we let $\Sp_{2n}$ denote $\ker \nu$. In this paper, we are mainly interested in the case $n=3$. Hence we will denote by $\G$ the group $\GSp_6$ and by $\G_0$ the group $\Sp_6$. Let $\T \subset \G$ denote the maximal diagonal torus and $\B \subset \G$ denote the standard Borel. We have
$$
\T=\left\{\mathrm{diag}(\alpha_1, \alpha_2, \alpha_3, \alpha_1^{-1}\nu, \alpha_2^{-1}\nu, \alpha_3^{-1}\nu), \alpha_1, \alpha_2, \alpha_3, \nu \in \G_m\right\}. 
$$
We associate to any $4$-uple $(\lambda_1, \lambda_2, \lambda_3, c) \in \Z^4$ such that $c \equiv \lambda_1+\lambda_2+\lambda_3 \pmod{2}$ the algebraic character $\lambda(\lambda_1, \lambda_2, \lambda_3, c)$ of $\T$ defined by
$$
\lambda(\lambda_1, \lambda_2, \lambda_3, c): \mathrm{diag}(\alpha_1, \alpha_2, \alpha_3, \alpha_1^{-1}\nu, \alpha_2^{-1}\nu, \alpha_3^{-1}\nu) \mapsto \alpha_1^{\lambda_1} \alpha_2^{\lambda_2} \alpha_3^{\lambda_3} \nu^{\frac{c-\lambda_1-\lambda_2-\lambda_3}{2}}.
$$
This defines an isomorphism between the group of $4$-uples $(\lambda_1, \lambda_2, \lambda_3, c) \in \Z^4$ such that $c \equiv \lambda_1+\lambda_2+\lambda_3 \pmod{2}$ and the group of algebraic characters of $\T$. Let $\rho_1=\lambda(1,-1, 0, 0)$ and $\rho_2=\lambda(0, 1, -1, 0)$ denote the short simple roots and let $\rho_3=\lambda(0,0,2,0)$ denote the long simple root. The set of roots of $\T$ in $\G$ is $R=R^+ \cup R^-$ where 
$$
R^+=\{ \rho_1, \rho_2, \rho_1+\rho_2, \rho_2+\rho_3, \rho_1+\rho_2+\rho_3, \rho_1+2\rho_2+\rho_3, 2\rho_1+2\rho_2+\rho_3, 2\rho_2+\rho_3, \rho_3\}
$$
is the set of positive roots with respect to $\B$ and $R^-=-R^+$. A weight $\lambda=\lambda(\lambda_1, \lambda_2, \lambda_3, c)$ is dominant for $\B$ if $\lambda_1 \geq \lambda_2 \geq \lambda_3 \geq 0$. For any such $\lambda$, there exists a unique (up to isomorphism) irreducible algebraic representation $V^\lambda$ of $\G$ of highest weight $\lambda$ and every irreducible algebraic representation of $\G$ is obtained in this way (up to isomorphism). Similarly, irreducible algebraic representations of $\GSp_4$ are classified by their highest weight which is a character of the shape $\lambda(\lambda_1, \lambda_2, c)$ with $\lambda_1 \geq \lambda_2 \geq 0$ and $\lambda_1+\lambda_2 \equiv c \pmod{2}$ (see for example \cite[\S 2.3]{lemmaCM} for more details). We will also use the classification of irreducible algebraic representations of the groups $\G_0=\Sp_6$ and $\Sp_4$. Hence let us recall that the diagonal maximal torus $\T_0=\T \cap \G_0$ of $\G_0$ is
$$
\T_0=\left\{ \mathrm{diag}(\alpha_1, \alpha_2, \alpha_3, \alpha_1^{-1}, \alpha_2^{-1}, \alpha_3^{-1}), \alpha_1, \alpha_2, \alpha_3 \in \G_m\right\}.
$$
and that its group of algebraic characters is isomorphic to $\Z^3$ via $(\lambda_1, \lambda_2, \lambda_3) \mapsto \lambda(\lambda_1, \lambda_2, \lambda_3)$ where
\begin{equation} \label{algweights}
\lambda(\lambda_1, \lambda_2, \lambda_3): \mathrm{diag}(\alpha_1, \alpha_2, \alpha_3, \alpha_1^{-1}, \alpha_2^{-1}, \alpha_3^{-1}) \mapsto \alpha_1^{\lambda_1} \alpha_2^{\lambda_2} \alpha_3^{\lambda_3}.
\end{equation}
A weight $\lambda=\lambda(\lambda_1, \lambda_2, \lambda_3)$ is dominant with respect to the standard Borel $\B_0=\B \cap \G_0$ if $\lambda_1 \geq \lambda_2 \geq \lambda_3 \geq 0$ and for any such $\lambda$ there exists a unique (up to isomorphism) irreducible algebraic representation $V^\lambda$ of $\G_0$ of highest weight $\lambda$ and every irreducible algebraic representation of $\G_0$ is obtained in this way (up to isomorphism). Similarly, irreducible algebraic representations of $\Sp_4$ are classified by characters $\lambda(\lambda_1, \lambda_2)$ with $\lambda_1 \geq \lambda_2$, with obvious notation.

\subsection{Compact Lie groups and representations} \label{SectCompactLie} Let $\mathrm{U}(n)=\{g \in \mathrm{GL}_n(\C) \,|\, ^t\overline{g}g=I_n\}$ denote the unitary group and let $K_\infty \subset \G_0(\R)$ be the subgroup defined as
$$K_\infty = \left\{ {\matrix A B {-B} A} \; | \; A A^t + B B^t = 1, A B^t = B A^t \right\}.$$
We have an isomorphism $\kappa: \mathrm{U}(3) \simeq K_\infty$ defined by $A+iB \mapsto {\matrix A B {-B} A}$. In fact $K_\infty$ is a maximal compact subgroup of $\G_0(\R)$. Let $T_\infty \subset K_\infty$ denote $\{\kappa(\mathrm{diag}(z_1, z_2, z_3)), z_1, z_2, z_3 \in \mathrm{U}(1)\}$. Then $T_\infty$ is Cartan subgroup of $K_\infty$. Its group of algebraic characters is isomorphic to $\Z^3$ via $(\lambda_1, \lambda_2, \lambda_3) \mapsto \lambda'(\lambda_1, \lambda_2, \lambda_3)$, where
\[ \lambda'(\lambda_1, \lambda_2, \lambda_3) : \kappa(\mathrm{diag}(z_1, z_2, z_3)) \mapsto z_1^{\lambda_1} z_2^{\lambda_2} z_3^{\lambda_3}. \]
An algebraic character is dominant if $\lambda_1 \geq \lambda_2 \geq \lambda_3$. For any dominant integral weight $\lambda'$, there exists a unique (up to isomorphism) irreducible representation $\tau_{\lambda'}$ of $K_\infty$ in a finite dimensional $\C$-vector space and every irreducible representation of $K_\infty$ is obtained in this way (up to isomorphism). In what follows, we will simply denote the irreducible representation of highest weight $\lambda'(\lambda_1, \lambda_2, \lambda_3)$ by $\tau_{(\lambda_1, \lambda_2, \lambda_3)}$. Let us explain the connection between the weights $\lambda$ of $\T_0$ defined by equation \eqref{algweights} in the previous section and the weights $\lambda'$ defined above. Let $J \in \G_0(\C)$ denote the matrix $J=\frac{1}{\sqrt{2}}{\matrix {I_3} { i I_3} {i I_3} {I_3}}$. Then we have 
$$
J^{-1}\kappa(\mathrm{diag}(z_1, z_2, z_3))J=\mathrm{diag}(z_1, z_2, z_3)
$$
and so, for any $(\lambda_1, \lambda_2, \lambda_3) \in \Z^3$, we have
$$
\lambda(\lambda_1, \lambda_2, \lambda_3)(J^{-1}\kappa(z_1, z_2, z_3)J)=\lambda'(\lambda_1, \lambda_2, \lambda_3)(\mathrm{diag}(z_1, z_2, z_3)).
$$
In brief, the character $\lambda'(\lambda_1,\lambda_2, \lambda_3)$ of $T_\infty$ is conjugated to the restriction of $\lambda(\lambda_1, \lambda_2, \lambda_3)$ to $\mathrm{U}(1)^3 \subset \C^\times \times \C^\times \times \C^\times=\T_0(\C)$.

\subsection{Lie algebras} Let $\mathfrak{g}_0$, resp. $\mathfrak{k}$, denote the Lie algebra of $\G_0(\R)$, resp. $K_\infty$, and let $\mathfrak{g}_{0,\C}$, resp. $\mathfrak{k}_\C$, denote its complexification. Then
\begin{eqnarray*} 
\mathfrak{g}_0 &=& \left\{ {\matrix A B C D} \in M_6(\R) \; | \; B = B^t, C = C^t, A = - D^t \right\}, \\
\mathfrak{k} &=& \left\{ {\matrix A B {-B} A} \in M_6(\R) \; | \; A = - A^t, B = B^t \right\}.
\end{eqnarray*}
The Lie algebra $\mathfrak{k}$ is the $1$-eigenspace for the Cartan involution $\theta (X)=-X^t$. The $(-1)$-eigenspace is $\mathfrak{p}=\left\{ {\matrix A B {B} {-A}} \in M_6(\R) \; | \; A = A^t, B = B^t \right\}.$ Letting 
$$
\mathfrak{p}_{\C}^\pm=\left\{ {\matrix A {\pm i A} {\pm i A} {-A} } \in M_6(\C) \;| \; A = A^t \right\},
$$
we have $\mathfrak{p}_{\C}=\mathfrak{p}_{\C}^+ \oplus \mathfrak{p}_{\C}^-$ and one has the Cartan decomposition \[ \mathfrak{g}_{0,\C}=\mathfrak{k}_{\C} \oplus \mathfrak{p}^+_{\C} \oplus \mathfrak{p}^-_{\C}. \]
For $1 \leq j \leq 3$, let $D_j \in M_3(\C)$ be the matrix with entry $1$ at position $(j, j)$ and $0$ elsewhere. Define
$T_j = {\matrix 0 {D_j} {- D_j} 0}$. Then the Lie algebra $\mathfrak{h}$ of $T_\infty$ is $\mathfrak{h}=\R \cdot T_1 \oplus \R \cdot T_2 \oplus \R \cdot T_3$. This is a compact Cartan subalgebra of $\mathfrak{g}_0$. Let $(e_1, e_2, e_3)$ denote the basis of $\mathfrak{h}^*_\C$ dual to $(-iT_1, -iT_2, -iT_3)$. A system of positive roots for $(\mathfrak{g}_{0,\C}, \mathfrak{h}_\C)$ is then given by
\[\{e_1 \pm e_2, e_1 \pm e_3, e_2 \pm e_3, 2e_1 , 2 e_2,2 e_3  \}. \]
The simple roots are $e_1 - e_2, e_2 - e_3$ and $2 e_3$. We note that $\mathfrak{p}^+_{\C}$ is spanned by the root spaces corresponding to the positive roots of type $2 e_j$ and $e_j+e_k$. We denote $\Delta = \{ \pm 2 e_j, \pm (e_j \pm e_k) \}$ the set of all roots, $\Delta_\mathrm{c} = \{ \pm(e_j - e_k)\}$ the set of compact roots and $\Delta_{\rm nc} = \Delta - \Delta_{\rm c}$ the non-compact roots. Finally, we note $\Delta^+, \Delta_{\rm c}^+$ and $\Delta_{\rm nc}^+$ the set of positive, positive compact and positive non-compact roots, respectively.

\subsection{Weyl groups} \label{rmk:weylcompact} \label{weylgroups} Recall that the Weyl group of $\G_0$ is given by $\mathfrak{W}_{\G_0} = \{ \pm 1\}^3 \rtimes \mathfrak{S}_3$. The reflection $\sigma_j$ in the hyperplane orthogonal to $2 e_j$ simply reverses the sign of $e_j$ while leaving the other $e_k$ fixed. The reflection $\sigma_{jk}$ in the hyperplane orthogonal to $e_j - e_k$ exchanges $e_j$ and $e_k$ and leaves the remaining $e_\ell$ fixed. The Weyl group $\mathfrak{W}_{K_\infty}$ of $K_\infty \cong U(3)$ is isomorphic to $\mathfrak{S}_3$ and, via the embedding into $\G$, identifies with the subgroup of $\mathfrak{W}_{\G_0}$ generated by the $\sigma_{jk}$. With the identification $\mathfrak{W}_{\G_0} = N(\mathbf{T}_0) / Z(\mathbf{T}_0)$, an explicit description of $\mathfrak{W}_{\G_0}$ and $\mathfrak{W}_{K_\infty}$ is given as follows. The matrices corresponding to the reflections $\sigma_{jk}$ are ${\matrix {S_{jk}} 0 0 {-S_{jk}} }$,
where $S_{jk}$ is the matrix with entry $1$ at places $(\ell, \ell)$, $\ell \neq j, k$, $(k, j)$ and $(j, k)$ and zeroes elsewhere. The matrices corresponding to the reflection $\sigma_j$ in the hyperplane orthogonal to $2 e_j$ are of the form
$ {\matrix 0 {U_j} {-U_j} 0}$,
where $U_j$ denotes the diagonal matrix with $-1$ at the place $(j, j)$ and ones at the other entries of the diagonal. This gives an explicit description of the elements of $\mathfrak{W}_{K_\infty}$ and their length:
\[ \mathfrak{W}_{K_\infty} = \{ 1, \sigma_{12}, \sigma_{13}, \sigma_{23}, \sigma_{12}  \sigma_{13}, \sigma_{12} \sigma_{23} \} \xrightarrow[]{\ell(\bullet)} \{0,1,1,1,2,2 \}. \]

\subsection{Discrete series}\label{ss:discreteseries}
We recall standard facts on discrete series for $\G_0(\R)=\Sp_6(\R)$ and for $\PGSp_6(\R)$. For any non-singular weight $\Lambda$ define
\[ \Delta^+(\Lambda) := \{ \alpha \in \Delta \; | \; \langle \alpha, \Lambda \rangle > 0 \}, \;\;\; \Delta^+_c(\Lambda) = \Delta^+(\Lambda) \cap \Delta_c, \]
where $\langle \;,\;\rangle$ is the standard scalar product on $\R^3$. Let $\lambda=(\lambda_1,\lambda_2,\lambda_3)$ be a weight of $T_\infty$ such that $\lambda_1 \geq \lambda_2 \geq \lambda_3 \geq 0$ and let $\rho = \frac{1}{2} \sum_{\alpha \in \Delta^+} \alpha = (3,2, 1)$. As $| \mathfrak{W}_{\G_0} / \mathfrak{W}_{K_\infty}| = 8$, by \cite[Theorem 9.20]{knapp}, the set of equivalence classes of irreducible discrete series representations of $\G_0(\R)$ with Harish-Chandra parameter $\lambda + \rho$ contains $8$ elements. More precisely, choose representatives $\{w_1, \ldots, w_{8}\}$ of $\mathfrak{W}_{\G_0} / \mathfrak{W}_{K_\infty}$ of increasing length and such that for any $1 \leq i \leq 8$, then the weight $w_i (\lambda + \rho)$ is dominant for $K_\infty$. The representatives, defined by their action on $\rho$, are $w_1(3,2,1)=(3,2,1)$, $w_2(3,2,1)= (3,2,-1) $, $w_3(3,2,1)=(3,1,-2)$, $w_4(3,2,1)= (2,1,-3)$, $w_5(3,2,1)=(3,-1,-2)$, $w_6(3,2,1)=(2,-1,-3)$, $w_7(3,2,1)= (1,-2,-3)$, $w_8(3,2,1)=(-1,-2,-3)$. Then, for any $1 \leq i \leq 8$, there exists an irreducible discrete series $\pi_\infty^{\Lambda}$, where $\Lambda = w_i(\lambda + \rho)$, of Harish-Chandra parameter $\Lambda$ and containing with multiplicity $1$ the minimal $K_\infty$-type with highest weight $ \Lambda + \delta_{\G_0} - 2\delta_{K_\infty}$ where $\delta_{\G_0}$, resp. $\delta_{K_\infty}$, is the half-sum of roots, resp. of compact roots, which are positive with respect to the Weyl chamber in which $\Lambda$ lies, i.e., $2 \delta_{\G_0} := \sum_{\alpha \in \Delta^+(\Lambda)} \alpha$, $2 \delta_{K_\infty} := \sum_{\alpha \in \Delta^+_c(\Lambda)} \alpha$.
Moreover, for $i \neq j$, $\Lambda = w_i (\lambda + \rho)$, $\Lambda = w_j(\lambda + \rho)$, the representations $\pi_\infty^{\Lambda}$ and $\pi_\infty^{\Lambda}$ are not equivalent and any discrete series of $\G_0$ is obtained in this way. Let $V^\lambda$ be the irreducible algebraic representation of $\G_0$ of highest weight $\lambda=(\lambda_1,\lambda_2,\lambda_3)$ (for $\T_0$). 

\begin{definition}
The discrete series $L$-packet $P(V^\lambda)$ associated to $\lambda$ is the set of isomorphism classes of discrete series of $\G_0(\R)$ whose Harish-Chandra parameter is of the form $\Lambda = w_i(\lambda + \rho)$ as $i$ varies.
\end{definition}

By \cite[Theorem II.5.3]{Borel-Wallach}, for each $\pi_\infty^{\Lambda} \in P(V^\lambda)$, the space \[ {\rm Hom}_{K_\infty}\left(\bigwedge^6 \mathfrak{g}_{0, \C}/ \mathfrak{k}_{\C} \otimes V^\lambda, \pi_\infty^{\Lambda}\right)\]
has dimension 1.
This is a consequence of the fact (cf. the proof of \cite[Theorem II.5.3]{Borel-Wallach}) that the minimal $K_\infty$-type of $\pi_\infty^{\Lambda}$ appears uniquely in $\bigwedge^6 \mathfrak{g}_{0, \C}/ \mathfrak{k}_{\C} \otimes V^\lambda$. Using the Cartan decomposition, we get \[\bigwedge^6 \mathfrak{g}_{0, \C}/ \mathfrak{k}_{\C} = \bigoplus_{p+q=6} \bigwedge^p \mathfrak{p}^+_{\C} \otimes_\C  \bigwedge^q \mathfrak{p}^-_{\C}.\] 
Hence, there exists a unique pair $(p,q)$ such that
$\Hom_{K_\infty}\left( \bigwedge^p \mathfrak{p}_{\C}^+ \otimes \bigwedge^q \mathfrak{p}_{\C}^- \otimes V^\lambda , \pi_\infty^\Lambda \right)$ is non-zero
and hence of dimension one. We call such a pair $(p,q)$  the Hodge type of $\pi_\infty^\Lambda$.

\begin{lemma} \label{discrete-series}
There exist two elements $\pi_{\infty,1}^{3,3}$ and $\overline{\pi}_{\infty,1}^{3,3}$ in $P(V^\lambda)$ of Hodge type $(3,3)$. They are characterized by having Harish-Chandra parameters $(\lambda_2 + 2, \lambda_3 + 1, - \lambda_1-3)$ and $(\lambda_1  + 3, -\lambda_3 -1, -\lambda_2 -2)$  and  minimal $K_\infty$-types  $\tau_{(\lambda_2 +2, \lambda_3 + 2, - \lambda_1 -4)}$ and $\tau_{(\lambda_1 + 4,-\lambda_3-2, - \lambda_2-2)}$ respectively.
\end{lemma}

\begin{proof}

The discrete series $\pi_{\infty,1}^{3,3}$ and $\overline{\pi}_{\infty,1}^{3,3}$ correspond to the Weyl representatives $w_4$ and $w_5$. Since $w_4 \lambda = (\lambda_2, \lambda_3, -\lambda_1)$ and $w_5 \lambda = (\lambda_1, -\lambda_3, -\lambda_2)$, the Harish-Chandra parameters of $\pi_{\infty,1}^{3,3}$ and $\overline{\pi}_{\infty,1}^{3,3}$ are as desired. When $\Lambda = w_4(\lambda + \rho)$ (resp. $\Lambda = w_5(\lambda + \rho)$), observe that $\delta_{\G_0}$ equals to $(2,1,-3)$ (resp. $(3,-1,-2)$), while  $ \delta_{K_\infty} = (1,0,-1)$ in both cases. Hence, using the formula above, the minimal $K_\infty$-types of $\pi_{\infty,1}^{3,3}$ and $\overline{\pi}_{\infty,1}^{3,3}$  are $\tau_{(\lambda_2 +2, \lambda_3 + 2, - \lambda_1 -4)}$ and $\tau_{(\lambda_1 + 4,-\lambda_3-2, - \lambda_2-2)}$ respectively. 

Recall that, after \cite[Proposition 6.19]{VoganZuckerman}, the Hodge type of a discrete series representation of Harish-Chandra parameter $\Lambda$ is $(p, q)$, where $p$ (resp. $q$) is the number of positive non-compact roots in $\Delta^+(\Lambda)$ (resp. $\Delta^-(\Lambda)$). Using this, one easily checks that the Hodge type of $\pi_{\infty, 1}^{3,3}$ and $\overline{\pi}_{\infty, 1}^{3,3}$ is $(3,3)$.

\end{proof}

The picture for $\PGSp_6(\R)$ is similar, but the set of its Harish-Chandra parameters changes slightly. This is due to the fact that, since its maximal compact subgroup has two connected components, the set of parameters has to be considered up to the action of $\mathfrak{W}_{K_\infty}$ and of $w_8$, as the latter, which is the anti-diagonal matrix with all entries $-1$, now belongs to the connected component away from the identity of the maximal compact subgroup. Concretely, any parameter  $\mu = ( \mu_1, \mu_2, \mu_3)$ has to be identified with $ w_8 \mu = (-\mu_3, -\mu_2, -\mu_1)$. 
  If $\lambda = (\lambda_1, \lambda_2, \lambda_3)$ is such that $\lambda_1 \geq \lambda_2 \geq \lambda_3 \geq 0$ and $\sum_i \lambda_i \equiv 0 \pmod{2}$, then the  irreducible algebraic $\G$-representation $V^{(\lambda,0)}$ of highest weight $\lambda(\lambda_1, \lambda_2, \lambda_3, 0)$ defines a representation of $\PGSp_6$. The corresponding discrete series $L$-packet $P(V^{(\lambda,0)})$ for $\PGSp_6(\R)$ has thus four elements. Any element $\pi_\infty \in P(V^{(\lambda,0)})$ of Harish-Chandra parameter $\mu$, viewed as a $\G(\R)$-representation, decomposes when restricted to $\G_0(\R)$ as the direct sum of two discrete series in $P(V^\lambda)$ of Harish-Chandra parameters $\mu$ and $w_8\mu$. As a consequence, for any such $\pi_\infty$, the space
  $$
  H^6\left(\mathfrak{g}, K_\G; \pi_\infty \otimes V^{(\lambda,0)} \right)=\mathrm{Hom}_{K_\G}\left(\bigwedge^6 \mathfrak{g}_{\C}/\mathrm{Lie}(K_\G)_{\C}, \pi_\infty \otimes V^{(\lambda, 0)} \right),
  $$
  where $\mathfrak{g}=\mathrm{Lie}( \G)$, $\mathfrak{g}_\C$ is its complexification and $K_\G=\R_+^\times K_\infty$, is $2$-dimensional. The discussion above implies the following.

\begin{lemma}\label{discrete-seriesforPGS}
Let $\lambda = (\lambda_1, \lambda_2, \lambda_3)$ be a dominant weight for $\G_0$ such that $ \sum_i \lambda_i \equiv 0 \pmod{2}$. Then there exists a unique discrete series $\pi_\infty^{3,3} \in P(V^{(\lambda,0)})$ of $\PGSp_6(\R)$, with Harish-Chandra parameter $(\lambda_2 + 2, \lambda_3 + 1, - \lambda_1-3)$, such that \[{\pi_\infty^{3,3}}_{|_{\G_0(\R)}}= \pi_{\infty,1}^{3,3}\oplus \overline{\pi}_{\infty,1}^{3,3}.\]
\end{lemma}

We will refer to $\pi_\infty^{3,3}$ as the discrete series of $\PGSp_6(\R)$ in $P(V^{(\lambda,0)})$ of Hodge type $(3,3)$.

\subsection{Shimura varieties}\label{sec:modularembedding}

Let $F$ denote a real \'etale quadratic $\Q$-algebra, i.e. $F$ is either a totally real quadratic extension of $\Q$ or $\Q \times \Q$.
Denote by $\GL_{2,F}^* /\Q$ the subgroup scheme of ${\rm Res}_{F/\Q} \GL_{2,F}$ sitting in the Cartesian diagram 
\[ \xymatrix{ 
\GL_{2,F}^* \ar@{^{(}->}[r] \ar[d] & {\rm Res}_{F/\Q} \GL_{2,F} \ar[d]^{{\rm det}} \\ 
\mathbf{G}_m  \ar@{^{(}->}[r] & {\rm Res}_{F/\Q} \mathbf{G}_{m,F}.
}
\]
For instance, when $F=\Q \times \Q$, we have
\[\GL_{2,F}^*=\{ (g_1, g_2) \in \GL_{2} \times \GL_2 \; | \; {\rm det}(g_1) = {\rm det}(g_2) \}.\]
Let $\H$ denote the group 
\begin{equation} \label{thegroupH}
\H = \GL_2 \boxtimes  \GL_{2,F}^* = \{(g_1, g_2) \in \GL_2 \times  \GL_{2,F}^* \; | \; {\rm det}(g_1) = {\rm det}(g_2) \}.
\end{equation}
We embed $\H$ into $\G$ as follows. Let us consider the $\Q \times F$-module \[ V := \Q e_1 \oplus F e_2 \oplus \Q f_1 \oplus F f_2,  \]
where $V_1:=\Q e_1 \oplus \Q f_1$ and $V_2:=F e_2 \oplus F f_2$ are respectively the standard representations of $\GL_2$ and $\GL_{2,F}^*.$ We equip $V$ with the $\Q \times F$-valued alternating form $\psi': V \times V \to  \Q \times F,$ such that $\psi'(e_1,f_1)=(1,0)$, $\psi'(e_2,f_2)=(0,\tfrac{1}{2})$ and $V_1$ is orthogonal to $V_2$. The group $\H$ acts naturally on $V$ and preserves $\psi'$ up to a scalar.
We can regard $V$ as a $6$-dimensional $\Q$-vector space with $\Q$-valued symplectic form $\psi:= {\rm tr}_{(\Q \times F)/\Q} \circ \psi' $. 
Explicitly, we have \[ \psi(a e_1 + \alpha e_2 , b f_1 + \beta f_2) = ab+ \tfrac{1}{2} {\rm tr}_{F/\Q}(\alpha \beta).\]
This identification defines an embedding $\H \hookrightarrow \GSp(V,\psi)$.
We now identify $\GSp(V,\psi)$ with $\G$ by choosing a suitable $\Q$-basis of $V$. Recall that the set of real quadratic $\Q$-algebras is parametrized by $D \in \Q^\times_{> 0} / ( \Q^\times _{> 0})^2,$ via $D \mapsto F = \Q \oplus \Q \sqrt{D}$. Using the decomposition $F = \Q \oplus \Q \sqrt{D}$, we consider the $\Q$-basis of $V$ given by \[\{ e_1, e_2, e_3 , f_1, f_2 , f_3\}:=\left\{ e_1, e_2, \sqrt{D} e_2 , f_1, f_2 , \tfrac{1}{\sqrt{D}} f_2 \right\}.  \]
In this basis, $\psi$ is represented by the matrix $J = { \matrix 0 {I_n} {-I_n} 0}$, thus we obtain and isomorphism $\GSp(V,\psi) \simeq \G$ and the embedding
$$
\iota: \H \hookrightarrow \G.
$$
Note that the group \[ \H':=\GL_2 \boxtimes \GSp_4 :=  \{(g_1, g_2) \in \GL_2 \times  \GSp_4 \; | \; {\rm det}(g_1) = {\nu}(g_2) \},  \]
is also naturally embedded in $\G$ and $\iota$ factors through $\H'$.

Recall from \cite[\S 2.2]{CLR} that there is a 3-dimensional Shimura variety $\Sh_\H$ associated to the $\H(\R)$-conjugacy class of \[h:\mathbf{S} \longrightarrow \H_{/\R}, \quad x+iy \mapsto \left( { \matrix {x} {y } {-y} {x} }, { \matrix {x} {y } {-y} {x} }, { \matrix {x} {y } {-y} {x} }  \right), \] where $\mathbf{S}={\rm Res}_{\C/\R}{\mathbf{G}_m}_{/\C}$ is the Deligne torus. The associated Shimura datum has reflex field is $\Q$ and the Shimura variety $\Sh_\H$ can be described as follows. If $V \subseteq \H(\Af)$ is a fibre product (over the similitude characters) $V_1 \times_{\A^\times_f} V_2 $ of sufficiently small subgroups, we have
\[ \Sh_{\H}(V) = \Sh_{\GL_2}(V_1) \times_{\mathbf{G}_m} \Sh_{\GL_{2,F}^*}(V_2), \]
where $\times_{\mathbf{G}_m}$ denotes the fibre product over the zero dimensional Shimura variety of level $W = {\rm det}(V_1)  = \det(V_2)$. The connected components are given by 
\[\pi_0(\Sh_{\H}(V)(\C)) = \hat{\Z}^\times/W.\]
Hence, $\Sh_\H$ can be thought as the fibre product of a modular curve and a Hilbert-Blumenthal modular surface.
We also recall that the complex points of $\Sh_{\H}(V)$ are given by
\[ \Sh_{\H}(V)(\C) = \H(\Q) \backslash \H(\A) / \Z_{\H}(\R) K_{\H, \infty} V, \]
where $\Z_{\H}$ denotes the center of $\H$ and $K_{\H, \infty} \subseteq \H(\R)$ is the maximal compact defined as the product $\mathrm{U}(1) \times \mathrm{U}(1) \times \mathrm{U}(1)$.

The embedding $\iota : \H \hookrightarrow \G$ induces a Shimura datum for $\G$ whose reflex field is $\Q$. 
For any sufficiently small compact open subgroup $U$ of $\G(\Af)$, denote by $\Sh_{\G}(U)$ the associated Shimura variety of dimension $6$.
We also write $\iota: \Sh_{\H}(U \cap \H) \hookrightarrow \Sh_{\G}(U)$ the closed embedding of codimension $3$ induced by the group homomorphism $\iota: \H \hookrightarrow \G$.

\begin{remark}
If $E / \Q$ is a totally real cubic field extension of $\Q$ then one can analogously define $\H = \{ g \in {\rm Res}_{E/\Q}\GL_{2, E} \; | \; \det(g) \in \mathbf{G}_m  \}$ and there is a natural embedding $\iota : \H \hookrightarrow \G$ (cf. \cite[\S 1]{PiatetskiShapiro-Rallis} for details) inducing closed embeddings $\iota : \Sh_\H(U \cap \H) \hookrightarrow \Sh_\G(U)$ for sufficiently small open compact $U$. All our results up to Section \S 5 will hold for any real \'etale cubic algebra $E$ over $\Q$. Our main interest in the case $E = \Q \times F$ for $F$ a real \'etale quadratic algebra over $\Q$ is motivated by the integral representation of the Spin $L$-function of $\G$ of \cite{Pollack-Shah}.
\end{remark}

\subsection{Cohomology of Siegel sixfolds}\label{ss:cohomologyloc} Let $\pi$ be a  cuspidal automorphic representation of $\G(\A)$ having non-zero fixed vectors by a neat compact open group $U \subseteq \G(\Af)$.  We assume that $\pi$ has trivial central character and hence we regard it as a cuspidal automorphic representation of $\PGSp_6(\A)$. Our purpose is to establish that, under mild assumptions, suitable localizations at $\pi$ of cuspidal, $L^2$, inner Betti and Betti cohomologies coincide and are concentrated in the middle degree. The assumptions are the following.
\begin{itemize}
    \item[\textbf{(DS)}] the archimedean component $\pi_\infty$ is a discrete series representation of $\PGSp_6(\R)$,
\item[\textbf{(St)}] at a finite place $p$ the component $\pi_p$ is the Steinberg representation of $\PGSp_6(\Q_p)$.
\end{itemize}
Let us fix for the rest of this section $\lambda=(\lambda_1, \lambda_2, \lambda_3) \in \Z^3$ satisfying $\lambda_1 \geq \lambda_2 \geq \lambda_3 \geq 0$ and $\sum \lambda_i \equiv 0 \pmod{2}$. We will denote by $V$, without mentioning $\lambda$ anymore, the irreducible algebraic representation of $\G$ of highest weight $(\lambda,0)$. As $V$ has trivial central character, it will be considered as an irreducible representation of $\PGSp_6$. Then $\pi_\infty$ belongs to the discrete series $L$-packet $P(V)$. As a consequence
$$
H^6(\mathfrak{g}, K_{\G}; \pi_\infty \otimes V)=\mathrm{Hom}_{K_{\G}}\left(\bigwedge^6 \mathfrak{g}_{\C} / \mathrm{Lie}(K_{\G})_\C; \pi_\infty \otimes V\right) \neq 0,
$$
where $K_\G = \R_{+}^\times K_\infty$.\\

There are natural inclusions of spaces of $\C$-valued functions
\[
\mathcal{C}_{\rm cusp}^\infty(\G(\Q) \backslash \G(\A)) \subseteq \mathcal{C}^\infty_{\rm rd} (\G(\Q) \backslash \G(\A)) \subseteq \mathcal{C}_{(2)}^\infty(\G(\Q) \backslash \G(\A)) \subseteq \mathcal{C}^\infty(\G(\Q) \backslash \G(\A)),
\]
where these spaces denote, respectively, the space of cuspidal square-integrable functions, rapidly decreasing functions, square-integrable functions and smooth functions, and
\[
\mathcal{C}^\infty_{\rm c / center}(\G(\Q) \backslash \G(\A)) \subseteq \mathcal{C}^\infty_{\rm rd}(\G(\Q) \backslash \G(\A)),
\]
where the first space is the space of compactly supported modulo the center functions (for the precise definition of these spaces, we refer to \cite{borel2}). Tensoring by $V$ the inclusions above and applying the $(\mathfrak{g}, K_\G)$-cohomology functor, we obtain the natural maps 
\[
\xymatrix@C=1em{ 
H^\bullet_{\rm cusp}(\Sh_\G(U), \mathcal{V}_\C) \ar[r] & H^\bullet_{\rm rd}(\Sh_\G(U) , \mathcal{V}_\C) \ar[r] & H^\bullet_{(2)}(\Sh_\G(U) , \mathcal{V}_\C) \ar[r] &  H^\bullet(\Sh_\G(U) , \mathcal{V}_\C), \\
& H^\bullet_{\rm c}(\Sh_\G(U) , \mathcal{V}_\C) \ar[u] & &
}
\]
where $\mathcal{V}_\C$ is the $\C$-local system associated to $V$. Let $H^\bullet_{!}(\Sh_\G(U) ,  \mathcal{V}_\C)$ denote the image of $H^\bullet_{c}(\Sh_\G(U) ,  \mathcal{V}_\C)$ in $H^\bullet(\Sh_\G(U) ,  \mathcal{V}_\C)$. Let $N$ denote the positive integer defined as the product of prime numbers $\ell$ such that $\pi_\ell$ is ramified.
The fact that $\pi_\infty$ is cohomological implies that there exists a number field $L$ whose ring of integers $\mathcal{O}_L$ contains the eigenvalues of the spherical Hecke algebra $\mathcal{H}^{sph, N}$ away from $N$ and with coefficients in $\Z$ acting on $\bigotimes'_{\ell \nmid N} \pi_\ell^{\G(\Z_\ell)}$. Let $\mathcal{H}^{sph, N}_L$ denote the spherical Hecke algebra away from $N$ with coefficients in $L$, let $\theta_\pi: \mathcal{H}^{sph, N}_L \to L$ denote the Hecke character of $\pi$ and let $\mathfrak{m}_\pi := {\rm ker}(\theta_\pi)$. Considering the localization at $\mathfrak{m}_\pi$ of the above cohomology groups, we have the following result.

\begin{proposition} \label{comparisoncohomologies}
Let $\pi$ satisfy the hypothesis \textbf{(DS)} and \textbf{(St)} above. Then 
 $$H^\bullet_{\rm cusp}(\Sh_\G(U) , \mathcal{V}_\C)_{\mathfrak{m}_\pi} =H^\bullet_{(2)}(\Sh_\G(U) ,  \mathcal{V}_\C)_{\mathfrak{m}_\pi} = H^\bullet_{!}(\Sh_\G(U) ,  \mathcal{V}_\C)_{\mathfrak{m}_\pi}=H^\bullet(\Sh_\G(U),  \mathcal{V}_\C)_{\mathfrak{m}_\pi}.$$
\end{proposition}

\begin{proof}
By \cite[Theorem 5.3 \& Corollary 5.5]{borel2}, the compositions of the horizontal maps \[ H^\bullet_{\rm cusp}(\Sh_\G(U) ,  \mathcal{V}_\C) \hookrightarrow H^\bullet_{*}(\Sh_\G(U) ,  \mathcal{V}_\C), \]
for $* \in \{ {\rm rd}, (2), \emptyset \}$, are injections. By \cite[Theorem 5.2]{borel2}, one has an isomorphism
\[ H^\bullet_{\rm c}(\Sh_\G(U) ,  \mathcal{V}_\C) \cong H^\bullet_{\rm rd}(\Sh_\G(U) ,  \mathcal{V}_\C). \]
Hence, if the equality $H^\bullet_{\rm cusp}(\Sh_\G(U) ,  \mathcal{V}_\C)_{\mathfrak{m}_\pi} =H^\bullet_{(2)}(\Sh_\G(U) ,  \mathcal{V}_\C)_{\mathfrak{m}_\pi}$ holds, we have
\[ H^\bullet_{\rm cusp}(\Sh_\G(U) ,  \mathcal{V}_\C)_{\mathfrak{m}_\pi} = H^\bullet_{(2)}(\Sh_\G(U) ,  \mathcal{V}_\C)_{\mathfrak{m}_\pi} = H^\bullet_{!}(\Sh_\G(U) ,  \mathcal{V}_\C)_{\mathfrak{m}_\pi}. \]
We show the former equality as follows. By \cite[\S 4]{borelstablel2}, \begin{equation}\label{eq:L2coho}
    H^\bullet_{(2)}(\Sh_\G(U) , \mathcal{V}_\C) = \bigoplus_{\sigma \subset L^2_d} \sigma_f^U  \otimes H^{\bullet}(\mathfrak{g}, K_{\G}; \sigma_\infty \otimes V)^{m(\sigma)}, \end{equation}
where $\sigma$ runs over the set of isomorphism classes of automorphic representations appearing in the discrete spectrum $L^2_d$ of $L^2(Z(\A)\G(\Q) \backslash \G(\A))$. Similarly,
\[H^\bullet_{\rm cusp}(\Sh_\G(U) , \mathcal{V}_\C) = \bigoplus_{\sigma \subset L^2_0} \sigma_f^U  \otimes H^{\bullet}(\mathfrak{g}, K_{\G}; \sigma_\infty \otimes V)^{m_0(\sigma)}, \]
where $\sigma$ runs over the set of isomorphism classes of automorphic representations in the cuspidal spectrum $L^2_0 \subset L^2_d$. From \eqref{eq:L2coho}, we can write 
\[H^\bullet_{(2)}(\Sh_\G(U), \mathcal{V}_\C)_{\mathfrak{m}_\pi} = \bigoplus_{\sigma = \sigma_\infty \otimes \sigma_f} \sigma_f^U  \otimes H^{\bullet}(\mathfrak{g}, K_{\G} ; \sigma_\infty \otimes V)^{m(\sigma)},\]
where $\sigma \in L^2_d$ is such that $\sigma_\ell^{\G(\Z_\ell)} \simeq \pi_\ell^{\G(\Z_\ell)} \ne 0$ at all $\ell \nmid N$. Notice that the latter implies that $\sigma_f^{N} \simeq \pi^{N}_f$, where for any automorphic representation $\tau$ we have denoted $\tau_f^{N} = \otimes_{\ell \nmid N} \tau_\ell$.
By \cite[Lemma 8.1(2)]{KretShin}, the Steinberg condition implies that the representation $\pi_\ell$ is tempered and unitary at each $\ell \nmid N$ (as $\pi$ has trivial central character). Thus, if $\sigma$ contributes non-trivially to the above sum, its local component at a finite place $\ell \nmid N$ is tempered. This implies that $\sigma$ is necessarily cuspidal and thus appears in $H^\bullet_{\rm cusp}(\Sh_\G , \mathcal{V}_\C)_{\mathfrak{m}_\pi}$ with multiplicity $m_0(\sigma)=m(\sigma)$. This last statement follows from the fact that any non-cuspidal automorphic representation appearing in $L^2_d$ is obtained as a residue of an Eisenstein series and in particular it is non-tempered at every place (cf. \cite[Proposition 4.5.4]{Labesse}). We are left to show that \[H^\bullet_{(2)}(\Sh_\G(U) , \mathcal{V}_\C)_{\mathfrak{m}_\pi} = H^\bullet(\Sh_\G(U) , \mathcal{V}_\C)_{\mathfrak{m}_\pi}. \] 
Recall that Franke's decreasing filtration on the space of automorphic forms for $\G(\A )$ (cf. \cite[\S 4.7]{Waldspurger}) yields a spectral sequence $E_1^{p,q} \Rightarrow H^{p+q}(\Sh_\G(U), \mathcal{V}_\C)$, where
\[ E_1^{p,q}  =  \bigoplus_{\substack{(w,P) \in B(p) \\ \ell(w) \leq p+q}} \bigoplus_{\sigma = \sigma_\infty \otimes \sigma_f} ({\rm Ind}_{P(\A_f)}^{\G(\A_f)} \sigma_f )^U \otimes  H^{p+q - \ell(w)}( \mathfrak{m}, K_M; \sigma_\infty \otimes W^{w(\lambda + \rho) - \rho}),  \]
where, for all $p \in \Z_{\geq 0}$, $B(p)$ denotes a certain  subset depending on $p$ of elements $(w,P)$  (cf. \cite[\S 4.8]{Waldspurger}), with $w \in \mathfrak{W}_\G$ and $P=M \cdot U_P$ a standard parabolic subgroup of $\G$,  $W^{w(\lambda + \rho) - \rho}$ denotes the irreducible algebraic representation of $M$ of highest weight ${w(\lambda + \rho) - \rho}$, and $\sigma$ runs over the set of isomorphism classes of automorphic representations appearing in the discrete spectrum of $L^2(Z_M(\A) M(\Q) \backslash M(\A))$.  By the proof of \cite[Lemma 8.1(1)]{KretShin}, we have that $E^{p,q}_{1,\mathfrak{m}_\pi}$ are zero unless when $(w,P)=(1,\G)$, in which case there exists a unique $ p_0 \in \Z_{\geq 0}$, for which \[E^{p,q}_{1,\mathfrak{m}_\pi} = \begin{cases} H^{p+q}_{(2)}(\Sh_\G(U) , \mathcal{V}_\C)_{\mathfrak{m}_\pi} & \text{ if } p=p_0, \\ 
0 & \text{otherwise.}
\end{cases}  \]
Thus, the spectral sequence for the localization degenerates at the first page and gives \[ H^{p_0 + \bullet}_{(2)}(\Sh_\G(U) , \mathcal{V}_\C)_{\mathfrak{m}_\pi} =E^{p_0,\bullet}_{1,\mathfrak{m}_\pi} = H^{p_0 + \bullet}(\Sh_\G(U) , \mathcal{V}_\C)_{\mathfrak{m}_\pi}. \]
\end{proof}

\begin{proposition} \label{vanishingoutsidemiddledegree}
Let $\pi$ satisfy the hypothesis \textbf{(DS)} and \textbf{(St)} above. Then, we have  
\[H^\bullet(\Sh_\G(U) , \mathcal{V}_\C)_{\mathfrak{m}_\pi} = H^6(\Sh_\G(U) , \mathcal{V}_\C)_{\mathfrak{m}_\pi} \ne 0.\]

\end{proposition}

\begin{proof} 
Suppose that $\tau_f$ contributes to $H^i(\Sh_\G(U) , \mathcal{V}_\C)_{\mathfrak{m}_\pi}$. As we noted in the proof of Proposition \ref{comparisoncohomologies}, this implies that, for every $\ell \nmid N$, $\tau_\ell \simeq \pi_\ell$ is tempered and unitary (cf. \cite[Lemma 8.1(2)]{KretShin}).  
Let us fix $\ell \nmid N$; the action of the Frobenius correspondence on intersection cohomology ${\rm Frob}_\ell$ on $IH^i(\Sh_\G(U), \mathcal{V}_\C)[\tau_f]$ and thus on $H^i(\Sh_\G(U) , \mathcal{V}_\C)[\tau_f]$ is pure of weight $i$, i.e.  its eigenvalues all have absolute value $\ell^{i/2}$ (cf. \cite[Remark 7.2.5]{Morel}). On the other hand, by the congruence relation conjectured in \cite[\S 6]{BlasiusRogawski} and verified in \cite{Wedhorn}, ${\rm Frob}_\ell$  is a root of the Hecke polynomial \[ H_\ell(T):= {\rm det}(T - \ell^3 {\rm spin}({\rm Fr}_\ell \ltimes \hat{g})),\]
which is a polynomial in $T$ whose coefficients are elements in the coordinate ring of the set of ${\rm Fr}_\ell$-conjugacy classes of semisimple elements of $\widehat{\G}(\C) = \mathbf{GSpin}_{7}(\C)$, for ${\rm Fr}_\ell$ a Frobenius element in the Weil group of $\Q_\ell$. By the untwisted Satake isomorphism, we can see $ H_\ell(T)$ as a polynomial with coefficients in the spherical Hecke algebra $\mathcal{H}(\G(\Q_\ell) // \G(\Z_\ell), \Q)$ (cf. \cite[(2.2.1) \& Corollary (2.8)]{Wedhorn}) and thus we can denote by $H_\ell(T ; \tau_\ell)$ the specialization of $H_\ell(T)$ to $\tau_\ell$, i.e. \[H_\ell(T ; \tau_\ell) = {\rm det}(T - \ell^3{\rm spin}( \phi_{\tau_\ell}({\rm Fr}_\ell))),\] where $\phi_{\tau_\ell}$ is the unramified Langlands parameter of $\tau_\ell$. The congruence relation gives that $H_\ell({\rm Frob}_\ell ; \tau_\ell) =0$ on $IH^\bullet(\Sh_\G(U), \mathcal{V}_\C)[\tau_f]$, which implies that the eigenvalues of ${\rm Frob}_\ell$ on $IH^\bullet(\Sh_\G(U), \mathcal{V}_\C)[\tau_f]$ are a subset of the ones of $\ell^{3} {\rm spin}(\phi_{\tau_\ell}({\rm Frob}_\ell))$. As $\tau_\ell$ is tempered, all the eigenvalues of ${\rm spin}(\phi_{\tau_\ell}({\rm Fr}_\ell))$ have absolute value equal to 1 (cf. \cite[\S 6]{Gross}). Hence the eigenvalues of  $\ell^3 {\rm spin}(\phi_{\tau_\ell}({\rm Fr}_\ell))$, and thus of ${\rm Frob}_\ell$, have all absolute value equal to $\ell^3$. In particular, $H^i(\Sh_\G(U), \mathcal{V}_\C)[\tau_f]$ is zero unless $i = 6$. Finally, notice that $H^6(\Sh_\G(U), \mathcal{V}_\C)_{\mathfrak{m}_\pi} \ne 0$  as the assumption \textbf{(DS)} implies $H^6(\Sh_\G(U), \mathcal{V}_\C)[\pi_f] \ne 0$. 
\end{proof}

\begin{remark}
 The proof of Proposition \ref{vanishingoutsidemiddledegree} is similar to the one of \cite[Proposition 8.2]{KretShin}, where the proof is carried on with a trace formula argument.
\end{remark}

\subsection{Hodge theory} \label{section:Hodgetheory}

We keep the same notation as \S \ref{ss:cohomologyloc}. In particular, $\pi = \pi_\infty \otimes \pi_f$ is a cuspidal automorphic representation of $\G$ with trivial central character which satisfies \textbf{(DS)} and  \textbf{(St)}, with $\pi_\infty \in P(V)$ for some irreducible algebraic representation $V$ of $\G$ as above.

Let $\mathcal{V}$ denote the $\Q$-local system on $\Sh_\G(U)$ attached to $V$. We can take the $\pi_f$-isotypic component 
$H^6_{B,*}[\pi_f]$ of $H^6_*(\Sh_\G(U), \mathcal{V}_\C)$, where $* \in \{ \emptyset, ! \}$ and where $\mathcal{V}_\C$ denotes $\mathcal{V} \otimes_{\Q} \C$. Propositions \ref{comparisoncohomologies} and \ref{vanishingoutsidemiddledegree} imply \begin{align}\label{proponcohomologies}
H^\bullet_{B}[\pi_f] =H^\bullet_{B,!}[\pi_f]= H^6_{B,!}[\pi_f] \ne 0. \end{align}
By \cite[(2.3.1)]{BlasiusRogawski} (see also \cite[Proposition 2.15]{shintemplier}), if $L$ is a sufficiently large number field, $H^6_{B}[\pi_f]$ appears as a sub-quotient of $H^6_{!}(\Sh_\G(U), \mathcal{V}_{L})$, where $\mathcal{V}_L$ denotes $\mathcal{V} \otimes_\Q L$. In particular, we have a projection \[ {\rm pr}_\pi : H^6(\Sh_\G(U), \mathcal{V}_{L})_{\mathfrak{m}_\pi}(n) \twoheadrightarrow H^6_{B}[\pi_f](n).\]
Since $ H^6_{!}(\Sh_\G(U), \mathcal{V}_{L})$ is a pure $L$-Hodge structure of weight $6$, we have \[ H^6_{B}[\pi_f] = \pi_f^U(L) \otimes M_B(\pi_f),\]
with $\pi_f^U(L)$ a realization of $\pi_f^U$ over $L$  and 
$M_B(\pi_f)$ a pure $L$-Hodge structure of weight $6$. Thus  we have a decomposition \[ M_B(\pi_f) \otimes \C =\bigoplus_{p+q=6} H^{p,q}(\pi_f) .\]

\begin{lemma}\label{lemmaonBettiisotypicpart}
Under the hypothesis \textbf{(DS)} and \textbf{(St)}\[{\rm dim}_\C \, H^{p,q}(\pi_f) = \begin{cases} 1 & \text{ if } p\ne 3,\\  2 & \text{ if } p= 3.\end{cases}\] 
In particular,  we have $ {\rm dim}_L M_B(\pi_f) = 8$. 
\end{lemma}
\begin{proof}
Thanks to \eqref{proponcohomologies}, we have that \[ H^6_{B}[\pi_f] \otimes \C = H^6_{B,!}[\pi_f] \otimes \C =H^6_{B,{\rm cusp}}[\pi_f] \otimes \C,  \]
hence \[  H^6_{B}[\pi_f] \otimes \C =\pi_f^U \otimes \bigoplus_{\sigma_\infty}  H^6(\mathfrak{g}, K_{\G}; \sigma_\infty \otimes V)^{m(\sigma)},\]
where $\sigma_\infty $ runs over the elements of the discrete series $L$-packet $P(V)$ of $\PGSp_6(\R)$ and $m(\sigma)$ denotes the multiplicity of $\sigma=\sigma_\infty \otimes \pi_f$. 
Notice that $H^6(\mathfrak{g}, K_\G; \sigma_\infty \otimes V )$ equals \begin{align}\label{eq:gKcohomologyDS}
     \mathrm{Hom}_{K_\infty}\left( \bigwedge^6 \mathfrak{g}_0/\mathfrak{k}, \sigma_\infty^1 \otimes V \right) \oplus \mathrm{Hom}_{K_\infty}\left( \bigwedge^6 \mathfrak{g}_0/\mathfrak{k}, \bar\sigma_\infty^1 \otimes V \right)
\end{align}
where we have denoted ${\sigma_\infty}_{|_{\G_0(\R)}}  = \sigma_\infty^1 \oplus \bar\sigma_\infty^1$. According to \cite[Theorem II.5.3 b)]{Borel-Wallach}, each space in the decomposition above is 1-dimensional. Moreover there exists a unique pair of integers $(r_{\sigma_\infty},s_{\sigma_\infty})$ satisfying $r_{\sigma_\infty}+s_{\sigma_\infty}=6$ such that \eqref{eq:gKcohomologyDS} equals

$$
\mathrm{Hom}_{K_\infty}\left( \bigwedge^{r_{\sigma_\infty}} \mathfrak{p}^+_\C \otimes \bigwedge^{s_{\sigma_\infty}} \mathfrak{p}^-_\C, \sigma^1_\infty  \otimes V \right) \oplus \mathrm{Hom}_{K_\infty}\left( \bigwedge^{s_{\sigma_\infty}} \mathfrak{p}^+_\C \otimes \bigwedge^{r_{\sigma_\infty}} \mathfrak{p}^-_\C, \bar\sigma^1_\infty  \otimes V \right).
$$       
As we remarked in \S \ref{ss:discreteseries}, the set $P(V)$ has four elements and is in bijection with the set of Hodge types up to conjugation. Since the Hodge structure in  $H^6_{B,{\rm cusp}}[\pi_f]$ is induced by this splitting, we deduce that 
\[ {\rm dim}_\C \, H^{r_{\sigma_\infty},s_{\sigma_\infty}}(\pi_f) = \begin{cases} m(\sigma) & \text{ if } r_{\sigma_\infty} \ne 3,\\  2 m(\sigma) & \text{ if } r_{\sigma_\infty} = 3.\end{cases} \]
By \cite[Theorem 12.1]{KretShin}, the multiplicity of $\sigma$ is either $0$ or $1$, while thanks to \cite[Corollary 8.4 \& Corollary 12.4]{KretShin} the dimension of $M_B(\pi_f)$ equals 8. Hence $m(\sigma) = 1$ for all $\sigma_\infty \in P(V)$, which concludes the proof.    
  \end{proof}

\subsection{Absolute Hodge cohomology} \label{sectAHC}

Let us first recall some definitions from \cite{BeilinsonAHC}. A mixed $\R$-Hodge structure consists of a finite dimensional $\R$-vector space $M_\R$ equipped with an increasing finite filtration $W_*$ called the weight filtration and a decreasing finite filtration $F^*$ on $M_\C = M_\R \otimes_\R \C$ called the Hodge filtration, such that each pair $(\mathrm{Gr}^W_n M_\R, (\mathrm{Gr}^W_n M_\C, F^*))$ is a  pure $\R$-Hodge structure of weight $n$ (\cite[D\'efinition 2.1.10]{DeligneTH2}). The category of mixed $\R$-Hodge structures is an abelian category \cite[T\'eor\`eme (2.3.5)]{DeligneTH2} and we denote it by $\mathrm{MHS}_\R$.

\begin{definition}
A real mixed $\R$-Hodge structure is given by a mixed $\R$-Hodge structure such that $M_\R$ is equipped with an involution $F_\infty^*$ stabilizing the weight filtration and whose $\C$-antilinear complexification $\overline{F_\infty^*} = F_\infty^* \otimes c$, where $c$ denotes the complex conjugation, defines an involution on $M_\C$ stabilizing the Hodge filtration. 
\end{definition}

We will refer to $F_\infty^*$ as the real Frobenius and to $\overline{F_\infty^*}$ as the de Rham involution. We denote by $\mathrm{MHS}_\R^+$ the abelian category of real mixed Hodge $\R$-structures. For any pair of objects $M, N \in D(\mathrm{MHS}_\R^+)$, one has $R \Hom_{\mathrm{MHS}_\R^+}(M, N) = R \Hom_{\mathrm{MHS}_\R}(M, N)^{\overline{F_\infty^*}}$, since taking invariants by $\overline{F_\infty^*}$ is an exact functor.

\begin{definition}
If $M = (M_\R, F_\infty^*) \in C(\mathrm{MHS}_\R^+)$ is a complex of real mixed $\R$-Hodge structure, its absolute Hodge cohomology is defined as
\[ R \Gamma_{\mathcal{H}}(M) = R \Hom_{\mathrm{MHS}_\R}(\R(0), M_\R). \]
Its real absolute Hodge cohomology is defined as
\[ R \Gamma_{\mathcal{H / \R}}(M) := R \Hom_{\mathrm{MHS}_\R^+}(\R(0), M) ) = R \Gamma_{\mathcal{H}}(M_\R)^{\overline{F_\infty^*}}. \]
\end{definition}

The cohomology groups $H^i_B(\Sh_{\G}(U), \mathcal{V}_\R)$, where $\mathcal{V}_\R = \mathcal{V} \otimes_\Q \R$, are equipped with a real Frobenius $F_\infty^*$ acting as the complex conjugation on (the complex points) $\Sh_{\G}(U)$ and on $\mathcal{V}_\R$, define real mixed $\R$-Hodge structures. This can be deduced directly from \cite{DeligneTH2} since the cohomology with coefficients is a direct factor of the cohomology of a fiber product of the universal abelian variety of $\Sh_\G(U)$, or from the theory of mixed Hodge modules of \cite{MorihikoSaito}. We let $M \in C(\mathrm{MHS}_\R^+)$ be the complex of real mixed $\R$-Hodge structures given by $(\bigoplus_{i \in \N} H^i_B(\Sh_{\G}(U), \mathcal{V}_\R)[-i], F_\infty^*)$ and we define the absolute real Hodge cohomology $H^7_\mathcal{H}(\Sh_\G(U) / \R, \mathcal{V}_\R(4))$ of $\Sh_\G(U)$ and coefficients in $\mathcal{V}_\R(4)$ to be $H^1(R \Gamma_{\mathcal{H}/\R}(M(4)))$.
Then we have the short exact sequence
$$
0 \rightarrow \mathrm{Ext}^1_{\mathrm{MHS}_{\R}^+}(\R(0), H^6_B(\Sh_{\G}(U),\mathcal{V}_\R(4))) \rightarrow
$$
$$
H^{7}_{\mathcal{H}}(\Sh_{\G}(U)/\R, \mathcal{V}_\R(4)) \rightarrow \mathrm{Hom}_{\mathrm{MHS}_{\R}^+}(\R(0), H^7_B(\Sh_{\G}(U),\mathcal{V}_\R(4))) \rightarrow 0
$$
If $\pi = \pi_\infty \otimes \pi_f$ is as above, we denote by \[H^1_\mathcal{H}(M(\pi_f)_\R(4)):=  \left(H^{7}_{\mathcal{H}}(\Sh_{\G}(U)/\R, \mathcal{V}_\R(4)) \otimes L \right ) [\pi_f]\]
the $\pi_f$-isotypical component.

\begin{lemma} \label{exloc}
Under the hypothesis \textbf{(DS)} and \textbf{(St)}, we have a canonical short exact sequence of finite rank free $\R \otimes_{\Q} L$-modules
\[
 0 \rightarrow F^{4}H^6_{\rm dR}[\pi_f] \rightarrow H^6_{\rm B}[\pi_f]^{F_\infty^\star = -1 }(3)  \rightarrow  H^1_\mathcal{H}(M(\pi_f)_\R(4)) \rightarrow 0.
\]
Moreover, we have
\[{\rm dim}_{\R \otimes_\Q L} \, H^1_\mathcal{H}(M(\pi_f)_\R(4)) = {\rm dim}_\C \, \pi_f^U.\]
\end{lemma}

\begin{proof} It follows from the existence of the short exact sequence above and from Proposition \ref{vanishingoutsidemiddledegree} that we have a canonical isomorphism 
$$
H^1_\mathcal{H}(M(\pi_f)_\R(4)) \simeq \mathrm{Ext}^1_{\mathrm{MHS}_{\R}^+}(\R(0), H^6_B(\Sh_{\G}(U),\mathcal{V}_\R(4))[\pi_f]).
$$
Hence, the first statement of the Lemma follows as in \cite[Lemma 4.11]{lemmaCM}. In particular,  the map $F^{4}H^6_{\rm dR}[\pi_f] \to H^6_{\rm B}[\pi_f]^{F_\infty^\star = -1 }(3)$ is defined by the composition of 
\[F^{4}H^6_{\rm dR}[\pi_f] \rightarrow H^6_{\rm dR}[\pi_f] \otimes \C \simeq  H^6_{B}(\Sh_\G(U), \mathcal{V}_{\C})[\pi_f],  \]
of the projection to
$ H^{6}_{B}(\Sh_{\G}(U), \mathcal{V}_{\C})[\pi_f]^{\overline{F_\infty^*}=1}$, where $\overline{F_\infty^*}$ is the complexification $F_\infty^* \otimes c$, with $c$ denoting the complex conjugation, and of the projection $$ H^{6}_{B}(\Sh_{\G}(U), \mathcal{V}_{\C})[\pi_f]^{\overline{F_\infty^*}=1}= H^6_{\rm B}[\pi_f]^{F_\infty^\star = -1 }(3) \oplus H^6_{\rm B}[\pi_f]^{F_\infty^\star = 1 }(4) \rightarrow  H^6_{\rm B}[\pi_f]^{F_\infty^\star = -1 }(3).$$ Finally, by Lemma \ref{lemmaonBettiisotypicpart}, we have that \[{\rm dim}_{\R \otimes_\Q L} \, F^{4}H^6_{\rm dR}[\pi_f] = 3{\rm dim}_L \, \pi_f^U(L)=3{\rm dim}_\C \, \pi_f^U.\] On the other hand, \[{\rm dim}_{\R \otimes_\Q L} \, H^6_{\rm B}[\pi_f]^{F_\infty^\star = -1 }(3) = (3 + h^{3,+} ){\rm dim}_\C \, \pi_f^U,\]
 where $h^{3,+}$ is the dimension of the $\C$-vector space $\{ x \in H^{3,3}(\pi_f)\,:\, F_\infty^*(x) = - x \}$  (cf. \cite[\S 3.4.2]{CLR}). By the proof of Lemma \ref{lemmaonBettiisotypicpart}, we have $h^{3,+} = 1$, which implies the result.
\end{proof}

\section{Construction of the motivic class} \label{section:motivicclass}

\subsection{Cartan product}

Before starting, we briefly recall some properties of the Cartan product of irreducible representations that will be needed (cf. \cite[\S 2.5]{sun} for more details). Let $A$ denote either a connected compact Lie group or a semisimple algebraic group. Fix a Cartan subgroup of $A$ and an orientation of the roots. Irreducible algebraic representations of $A$ are parametrized by dominant weights. If $\lambda$ and $\sigma$ are two dominant weights with corresponding representations $V^\lambda$ and $V^\sigma$, then the representation $V ^{\lambda + \sigma}$ appears in $V^\lambda \otimes V^\sigma$ with multiplicity one. We denote it by $V^\lambda \cdot V^\sigma$ and we call it the Cartan component of $V^\lambda \otimes V^\sigma$. Clearly, the tensor product of two highest weight vectors maps to a corresponding highest weight vector. We denote by $v \otimes w \mapsto v \cdot w$ the projection from $V^\lambda \otimes V^\sigma$ to its Cartan component $V^\lambda \cdot V^\sigma$.
 
\begin{lemma} \cite[Lemma 2.12]{sun} \label{lemma:sun}
Every non-zero pure tensor in $V^\lambda \otimes V^\sigma$ projects non trivially to the Cartan component.
\end{lemma}

\subsection{Branching laws} In what follows, we fix a totally real field $F$ over which $\H$ splits. Since $\H$ is split over $F$, its finite dimensional irreducible representations are determined by the highest weight theory and we can thus use the branching laws for algebraic representations from $\G$ to $\H$ established in \cite{CRDocumenta}.

\begin{lemma} \label{lemma:branchinglaw}
The $\G$-representation $V^\lambda$ over $F$ of highest weight $\lambda=(\lambda_1, \lambda_2, \lambda_3, c)$ contains the trivial $\H$-representation if and only if $c=0$ and $\lambda_1 = \lambda_2 + \lambda_3$. When this holds the trivial representation of $\H$ appears in $(V^\lambda)_{|_\H}$ with multiplicity $\lambda_2 - \lambda_3 +1$.
\end{lemma}

\begin{proof}
From \cite[Lemma 2.10]{CRDocumenta},  the sum of all irreducible sub-$\H$-representations of $V^\lambda$ isomorphic (up to a twist) to $\operatorname{Sym}^{(k, 0, 0)}$ for some $k \geq 0$  is given by

\[ \bigoplus_{\substack{k = | \lambda_1 - \lambda_2 - \lambda_3 | \\ k \equiv |\lambda| \, (\text{mod } 2)}}^{\lambda_1 - \lambda_2 + \lambda_3} r \cdot \operatorname{Sym}^{(k, 0, 0)} \otimes \det^{\frac{|\lambda| - k}{2}}, \]
for  $r= \lambda_2 - \lambda_3 +1$. From this we deduce that $V^\lambda$
contains the trivial $\H$-representation with multiplicity $r= \lambda_2 - \lambda_3 +1$ if and only if $\lambda_1 - \lambda_2 -\lambda_3 = 0 $.
\end{proof}

It will be useful to construct explicitly generators of the trivial $\H$-representations inside $V^\lambda$ given by the branching laws. We achieve this by constructing some vectors in the representations $V^{(1, 1 , 0, 0)}$ and $V^{(2,1,1,0)}$ and then by taking their Cartan product. From now on, all the representations are defined over $F$. Moreover, since the branching laws are determined by the restriction to the derived subgroups, in the following we work with the groups \[ \H_0 := \SL_2 \times \SL_{2} \times \SL_2 \hookrightarrow \H_0' := \SL_2 \times \Sp_4 \hookrightarrow \G_0 = \Sp_6. \] Recall that  we associate to any $\lambda=(\lambda_1, \lambda_2) \in \Z^2$ such that $\lambda_1 \geq \lambda_2 \geq 0$, the irreducible $\Sp_4$-representation with highest weight $\lambda$. Applying the branching laws \cite[Proposition 2.8]{CRDocumenta},  we get the following decompositions of representations of  $\H_ 0'$:

\begin{align*}
 V^{(1 , 1, 0)} &= (\operatorname{Sym}^0 \boxtimes V^{(0,0)}) \oplus  (\operatorname{Sym}^0 \boxtimes V^{(1, 1)}) \oplus (\operatorname{Sym}^1 \boxtimes V^{(1,0)}), \\ 
 V^{(2,1,1)} &= (\operatorname{Sym}^0 \boxtimes V^{(1, 1)}) \oplus  (\operatorname{Sym}^0 \boxtimes V^{(2,0)}) \oplus (\operatorname{Sym}^1 \boxtimes V^{(1, 0)}) \oplus  (\operatorname{Sym}^1 \boxtimes V^{(2,1)})\\
 &\,\,\,\,\,\,\,\,\,\,\oplus  (\operatorname{Sym}^2 \boxtimes V^{(1,1)}).
\end{align*}
By Lemma \ref{lemma:branchinglaw}, $V^{(1 , 1, 0)}$ contains two copies of the trivial $\H_0$-representation, each of which lies  resp. in $\operatorname{Sym}^0 \boxtimes V^{(0,0)}$ and $\operatorname{Sym}^0 \boxtimes V^{(1, 1)}$, while $V^{(2,1,1)}$ contains a unique trivial $\H_0$-representation appearing in $\operatorname{Sym}^0 \boxtimes V^{(1, 1)}$. Using these facts, we can explicitly define generators of these three trivial representations of $\H_0$. 

Let $V$ be the standard representation of $\G_0$ with its symplectic basis $\langle e_1, e_2, e_3, f_1, f_2, f_3 \rangle$ given in \S \ref{sec:modularembedding}. According to our choice of embedding $\H'_0 \hookrightarrow \G_0$, $\langle e_1,f_1 \rangle$, resp. $\langle  e_2, e_3,  f_2, f_3 \rangle$, defines a basis of the standard representation of $\SL_2$, resp. $\Sp_4$. We first recall how one can realize the representations $V^{(1,1,0)}$ and $V^{(1,1,1)}$. As explained in \cite[\S 17.1]{FultonHarris}, $V^{(1,1,0)}$ is realized inside $\bigwedge^2 V$ as the complement of the $\G_0$-invariant subspace generated by the vector $e_1 \wedge f_1 + e_2 \wedge f_2 + e_3\wedge f_3$ corresponding to the symplectic form or, in other words, as the kernel of the map $\bigwedge^2 V \to V$ sending $v_1 \wedge v_2$ to $\psi(v_1, v_2)$. By \cite[Theorem 17.5]{FultonHarris}, the irreducible representation $V^{(1 , 1 , 1)}$ is identified with the kernel of the map $\varphi : \bigwedge^3 V \to V,\; \; v_1 \wedge v_2 \wedge v_3 \mapsto \sum_{i<j, k \ne i,j} \psi(v_i,v_j)(-1)^{i-j+1}v_k$.

\begin{lemma} \label{vwz}Let $F(0)$ denote the trivial $\H_0$-representation. We have the following
\begin{align*}
    v &:= e_2 \wedge f_2 - e_3 \wedge f_3 \in F(0) \subseteq \operatorname{Sym}^0 \boxtimes V^{(1, 1)} \subseteq V^{(1 , 1 , 0)}, \\
    w &:= e_2 \wedge f_2 + e_3 \wedge f_3 - 2 e_1 \wedge f_1 \in F(0) \subseteq \operatorname{Sym}^0 \boxtimes V^{(0, 0)} \subseteq V^{(1 , 1 , 0)},\\
    z &:= z_1 - z_2 \in F(0) \subseteq \operatorname{Sym}^0 \boxtimes V^{(1, 1)} \subseteq V^{(2 , 1 , 1)},
\end{align*}
where \[ z_1 :=e_1 \cdot( f_1 \wedge e_2 \wedge f_2 - f_1  \wedge e_3 \wedge f_3),\] \[ z_2:= f_1 \cdot( e_1 \wedge e_2 \wedge f_2 - e_1  \wedge e_3 \wedge f_3),\] and $\cdot$ denotes the Cartan product.

\end{lemma}

\begin{proof}
 The vector $v$ is obtained from the highest weight vector $e_1 \wedge e_2$ in $V^{(1,1,0)}$ by applying the composition $X_{(0,1,-1)} \circ X_{(0,-2,0)} \circ X_{(-1,0,1)}$, where $X_{(-1, 0 , 1)}, X_{(0,-2,0)}, X_{(0,1,-1)}  \in \mathfrak{sp}_6$ denote the weight vectors for $\lambda(-1,0,1)$, $\lambda(0,-2,0)$, and $\lambda(0,1,-1)$ respectively (cf. \cite[\S 16.1]{FultonHarris} for the precise description). Moreover, the vector $X_{(-1,0,1)}(e_1 \wedge e_2) = - e_2 \wedge e_3$ is of weight $(0,1,1)$, which appears only in the component $\Sym^0 \boxtimes V^{(1,1)}$, and $X_{(0,1,- 1)}, X_{(0,-2,0)} \in \mathfrak{sp}_4 \subseteq \mathfrak{sp}_6$ so $v$ still lies inside $\Sym^0 \boxtimes V^{(1,1)}$. The vector $w$ is $\H_0'$-invariant and therefore it generates the only trivial $\H_0'$-representation in $V^{(1,1,0)}$. We now explain the definition of $z$.
Note that $e_1 \in V^{(1,0,0)}$ and $ f_1 \wedge e_2 \wedge f_2 - f_1  \wedge e_3 \wedge f_3 \in V^{(1 , 1 , 1)}$. Thus, by the properties of the Cartan product
 \[ V^{(1 , 0 , 0)} \otimes V^{(1 , 1 , 1)} = V^{(1,1,0)} \oplus V^{(2,1,1)} \to V^{(2,1,1)}, \; \; v_1  \otimes v_2 \mapsto v_1 \cdot v_2,  \] $z_1$ is a non-zero vector in $V^{(2,1,1)}$ by Lemma \ref{lemma:sun}. The vector $z_1$ is fixed by $\{ I_2 \} \times \SL_{2}^2$, but not by $\SL_2 \times \{ I_2 \} \times \{ I_2 \}$, however, as it is easy to verify, we have that 
\[ z = z_1 + h \cdot z_1 = z_1 -z_2 \in F(0) \subset V^{(2,1,1)}, \; \; \text{ with } h=  \left({\matrix {}{1}{ -1}{} }, I_2, I_2\right), \]
generates the unique trivial $\H_0$-representation of $V^{(2,1,1)}$.
\end{proof}
 
\begin{lemma} \label{lemma:choiceofsomehighestwv}
 Let $\lambda = (\lambda_2 + \lambda_3, \lambda_2, \lambda_3,0)$ with $\lambda_2 \geq \lambda_3 \geq 0$. For each $ \lambda_2 \geq \mu \geq \lambda_3$, the vector
 \begin{eqnarray*} v^{[\lambda, \mu]}:= v^{\lambda_2 - \mu} \cdot w^{ \mu - \lambda_3} \cdot z^{\lambda_3} \in F(0) \subseteq (V^\lambda)_{|_\H} \end{eqnarray*}
realizes a distinct copy of the trivial representation $F(0)$ of $\H$ inside $(V^\lambda)_{|_\H}$.
\end{lemma}

\begin{proof}
For $p, q, r\in \N$, we have
\[ v^p \cdot w^q \cdot z^r \in F(0) \subseteq \operatorname{Sym}^0 \boxtimes V^{(p + r, p + r)} \subseteq V^{( p+q+2r , p+ q + r , r)}. \]
The vectors $v,w,z$ are $\H$-highest weight vectors, thus $v^{[\lambda, \mu]}$ is such one too. We are left to show that each of the vectors is different. This follows from the fact that each $v^{[\lambda, \mu]}$ lies in $\operatorname{Sym}^0 \boxtimes V^{(\lambda_2 + \lambda_3 - \mu , \lambda_2 + \lambda_3 - \mu )} \otimes \nu^{\mu - \lambda_2 - \lambda_3} $ and these representations are all different as $\mu$ varies.

 \end{proof}
 
\subsection{The motivic class}\label{subsec:tateclass}

As in the section above, we fix a totally real field $F$ such that $\H$ splits over $F$. For a smooth quasi-projective scheme $S$ over a field of characteristic zero, let $\operatorname{CHM}_{L}(S)$ denote the tensor  category of relative Chow motives over $S$ with coefficients in a number field $L$ and denote by $M : \mathrm{Var} / S \to \operatorname{CHM}_{L}(S)$ the contravariant functor from the category of smooth projective schemes over $S$ to the category of relative Chow motives over $S$ (cf. \cite[\S 2.1]{ancona}). By \cite[Corollary 3.2]{DeningerMurre}, if $A/S$ is an abelian scheme of relative dimension $g$, there is a decomposition $M(A) = \bigoplus_{i = 1}^{2g} h^i(A)$ in $\operatorname{CHM}_{L}(S)$. 
Let $G$ temporarily denote one of the groups $\H$ or $\G$, and  denote by $\operatorname{Rep}_{F}(G)$ the category of finite dimensional algebraic representations of $G $ defined over $F$. In \cite{ancona}, Ancona constructs an additive functor  

\[\mu^{G}_{U}: \operatorname{Rep}_{F}(G) \longrightarrow \operatorname{CHM}_{F}(\Sh_G(U)),\] 
where $U$ is a sufficiently small open compact subgroup of $G(\A_f)$. We recall some of its properties.
 
\begin{proposition}[{\cite[Th\'eor\`eme 8.6]{ancona}}]   The functor $\mu^{G}_{U}$ respects duals, tensor products and satisfies the following properties.

\begin{enumerate}
\item If $V$ is the standard representation of $G$, then $ \mu^{G}_{U}(V)=h^1(\mathscr{A}_G)$, where $\mathscr{A}_G$ is the universal abelian scheme over $\Sh_G(U)$.
\item If $\nu: G \to \mathbf{G}_m$ is the multiplier, then $ \mu^{G }_{U }(\nu)$ is the Lefschetz motive $F(-1)$.
\item For a $G$-representation $V$ defined over $F$, the Betti realization of $\mu^{G}_{U}(V)$ is the local system $\mathcal{V}_F$ associated to the vector bundle  
\[ G(\Q) \backslash ( X_G \times V \times (G(\A_f)/U)) \to \Sh_G(U)(\C).\]
\item For any prime $v$ of $F$ above $\ell$ and $G$-representation $V$, the $v$-adic \'etale realization $\mathcal{V}_v$ of $\mu^{G}_{U}(V)$ is the \'etale sheaf associated to $V \otimes_F F_v$, with $U$ acting on the left via $U \hookrightarrow G(\A_f) \to G(\Q_\ell)$.
\end{enumerate}
\end{proposition}

\begin{definition}
Let $V^\lambda$ be the finite dimensional irreducible algebraic representation over $\Q$ of $\G$ of highest weight $\lambda$. We denote by $\mathscr{V}^\lambda_F$ the relative Chow motive associated to $V^\lambda \otimes F$.
\end{definition}

Let $U \subset \G(\A_f)$ be a sufficiently small  compact open subgroup and let $U'=U \cap \H(\A_f)$. Recall that we have a closed embedding $\iota: \Sh_{\H}(U') \hookrightarrow \Sh_{\G}(U)$ which is of codimension $3$. Let $V^\lambda$ the algebraic representation of $\G$ (over $F$) of highest weight $\lambda=(\lambda_1, \lambda_2, \lambda_3, c)$ such that $\lambda_1=\lambda_2 + \lambda_3$ and $c=0$. Using the branching laws of Lemma \ref{lemma:branchinglaw} and \cite[Theorem 1.2]{torzewski}, we get the following (cf. \cite[Proposition 2.17]{CRDocumenta}).

\begin{proposition}
For any $\lambda_2 \geq \mu \geq \lambda_3$, we have a Gysin morphism  
\[\iota_*^{[\lambda,\mu]}:  H^0_{\mathcal{M}}(\Sh_\H(U'), F(0)) \to H^6_{\mathcal{M}}(\Sh_\G(U), \mathscr{V}^\lambda_F(3)), \] 
corresponding to the embedding of $F(0) \subset \iota^* V^\lambda$ given by the $\H$-trivial vector $v^{[\lambda,\mu]}$ of Lemma \ref{lemma:choiceofsomehighestwv}.
\end{proposition} 

\begin{definition}
We let $\mathcal{Z}_{\H, \mathcal{M}}^{[\lambda,\mu]} \in H^6_{\mathcal{M}}(\Sh_\G(U), \mathscr{V}^\lambda_F(3))$ be the image by  $\iota_*^{[\lambda,\mu]}$ of  \[ \mathbf{1}_{\Sh_\H(U')} \in {\rm CH}^0(\Sh_\H(U'))_F = H^0_{\mathcal{M}}(\Sh_\H(U'), F(0)).\] 
\end{definition}

\subsection{Realizations} \label{SectRealizations}

\subsubsection{\'Etale realization}

Let $\mathfrak{l}$ be a prime of $F$ above $\ell$.
We have an \'etale cycle class map \[ \rm{cl}_{\text{\'et}}:H^6_{\mathcal{M}}(\Sh_\G(U), \mathscr{V}^\lambda_F(3)) \to H^6_{\text{\'et}}(\Sh_\G(U), \mathcal{V}^\lambda_\mathfrak{l}(3)) \to H^6_{\text{\'et}}(\Sh_\G(U)_{\overline{\Q}}, \mathcal{V}^\lambda_\mathfrak{l}(3))^{G_\Q}, \]
where the last arrow is the natural map obtained from the Hochschild-Serre spectral sequence. We define the following.

\begin{definition} \label{etaleclass}
We let  $\mathcal{Z}_{\H, \text{\'et}}^{[\lambda,\mu]} := \rm{cl}_{\text{\'et}}(\mathcal{Z}_{\H, \mathcal{M}}^{[\lambda,\mu]}) \in H^6_{\text{\'et}}(\Sh_\G(U)_{\overline{\Q}}, \mathcal{V}^\lambda_\mathfrak{l}(3))^{G_\Q}$.
\end{definition}

\begin{remark}\leavevmode
\begin{itemize}
    \item  Notice that $\mathcal{Z}_{\H, \text{\'et}}^{[\lambda,\mu]}$ equals to the image of $\mathbf{1} \in H^0_{\text{\'et}}(\Sh_\H(U')_{\overline{\Q}}, F_\mathfrak{l}(0))$ via the \'etale Gysin map 
\[\iota_{\text{\'et},*}^{[\lambda,\mu]} : H^0_{\text{\'et}}(\Sh_\H(U')_{\overline{\Q}}, F_\mathfrak{l}(0))  \to H^0_{\text{\'et}}(\Sh_\H(U')_{\overline{\Q}}, \iota^* \mathcal{V}^\lambda_\mathfrak{l})\to H^6_{\text{\'et}}(\Sh_\G(U)_{\overline{\Q}}, \mathcal{V}^\lambda_\mathfrak{l}(3)).\]  
\item 
As the representation $V^\lambda$ is self dual, we have a Galois equivariant perfect pairing

\[ H^6_{\text{\'et},c}(\Sh_{\G}(U)_{\Qbar{}},\mathcal{V}_\mathfrak{l}^\lambda(3)) \times H^6_{\text{\'et}}(\Sh_{\G}(U)_{\Qbar{}},\mathcal{V}_\mathfrak{l}^\lambda(3)) \rightarrow F_\mathfrak{l}(0).\]
Hence, by duality $\mathcal{Z}_{\H, \text{\'et}}^{[\lambda,\mu]}$ determines a map  \[ H^6_{\text{\'et},c}(\Sh_{\G}(U)_{\Qbar{}},\mathcal{V}_\mathfrak{l}^\lambda(3)) \to F_\mathfrak{l}(0).\]
\end{itemize}
\end{remark}
 
\subsubsection{Betti realizations}

As in the previous subsection, we define the class \[\mathcal{Z}_{\H, B}^{[\lambda,\mu]} \in H^6_{B}(\Sh_\G(U)(\C), \mathcal{V}^\lambda_{F}(3))\] as the image of $\mathcal{Z}_{\H, \mathcal{M}}^{[\lambda,\mu]}$ via the Betti cycle class map
\[ \rm{cl}_{B}:H^6_{\mathcal{M}}(\Sh_\G(U), \mathscr{V}^\lambda_F(3)) \to H^6_{B}(\Sh_\G(U)(\C), \mathcal{V}^\lambda_{F}(3)).  \]
Note that, as $F$ is totally real, the image \[{\rm Im}(\rm{cl}_{B}) \subset  H^{6}_{B}(\Sh_{\G}(U)(\C), \mathcal{V}^\lambda_{\R}(3))^{{F_\infty^\star} = 1 }, \]
where $F_\infty^*$ denotes the composition of the map induced by complex conjugation on the $\C$-points of $\Sh_\G(U)$ with complex conjugation on the coefficients.

\subsubsection{Absolute Hodge realizations} \label{section:ahreal} Let $H^6_{\mathcal{M}}(\Sh_\G(U), \mathscr{V}^\lambda_F(3))^0=\ker(\mathrm{cl}_B)$ denote the subgroup of homologically trivial classes and let $H^6_{\mathcal{M}}(\Sh_\G(U), \mathscr{V}^\lambda_F(3))_{hom}$ denote the quotient $H^6_{\mathcal{M}}(\Sh_\G(U), \mathscr{V}^\lambda_F(3))/H^6_{\mathcal{M}}(\Sh_\G(U), \mathscr{V}^\lambda_F(3))^0$. Note that when $\lambda_2=\lambda_3=0$, i.e. the representation $V^\lambda$ is the trivial representation, then $H^6_{\mathcal{M}}(\Sh_\G(U), \mathscr{V}^\lambda_F(3))=H^6_{\mathcal{M}}(\Sh_\G(U), F(3))=\mathrm{CH}^3(\Sh_\G(U))_F$ is the usual Chow group of $3$-codimensional cycles modulo rational equivalence and the space $H^6_{\mathcal{M}}(\Sh_\G(U), \mathscr{V}^\lambda_F(3))_{hom}=\mathrm{N}^3(\Sh_\G(U))_F$ is the space of $3$-codimensional cycles modulo homological equivalence, with coefficients in $F$. In this section, we define a natural injective map
\begin{equation} \label{cycleHodge}
H^6_{\mathcal{M}}(\Sh_\G(U), \mathscr{V}^\lambda_F(3))_{hom} \rightarrow H^7_\mathcal{H}(\Sh_\G(U), \mathcal{V}^\lambda_{\R}(4)).
\end{equation}
The definition is similar to the one for smooth projective varieties (see \cite[\S 5]{Schneider}) and we recall it for the convenience of the reader. The cycle class map is an injection 
$$
\mathrm{cl}_B: H^6_{\mathcal{M}}(\Sh_\G(U), \mathscr{V}^\lambda_F(3))_{hom} \rightarrow H^6_B(\Sh_\G(U), \mathcal{V}_F(3))^{F_\infty^* = 1} \cap H^6_{B}(\Sh_\G(U), \mathcal{V}^\lambda_{\C}(3))^{0,0}
$$
where $H^6_{B}(\Sh_\G(U), \mathcal{V}^\lambda_{\C}(3))^{0,0}$ denotes the subspace of $$W_0 H^6_{B}(\Sh_\G(U), \mathcal{V}^\lambda_{\C}(3))=\mathrm{Gr}_0^WH^6_{B}(\Sh_\G(U), \mathcal{V}^\lambda_{\C}(3))$$
of vectors which have Hodge type $(0,0)$. The composite of the inclusions
\begin{eqnarray*}
H^6_B(\Sh_\G(U), \mathcal{V}_F(3))^{F_\infty^* = 1} \cap H^6_{B}(\Sh_\G(U), \mathcal{V}^\lambda_\C(3))^{0,0} & \hookrightarrow& W_0 H^6_{dR}(\Sh_\G(U),\mathcal{V}_{\R}(3))\\
&\hookrightarrow& W_2 H^6_{dR}(\Sh_\G(U),\mathcal{V}_{\R}(3))\\
&=& W_0 H^6_{dR}(\Sh_\G(U),\mathcal{V}_{\R}(4))
\end{eqnarray*}
and of the projection
$$
W_0 H^6_{dR}(\Sh_\G(U),\mathcal{V}_{\R}(4)) \rightarrow
$$
$$
W_0 H^6_{B}(\Sh_\G(U),\mathcal{V}_{\R}(4))^+ \backslash W_0 H^6_{dR}(\Sh_\G(U),\mathcal{V}_{\R}(4))/F^0W_0 H^6_{dR}(\Sh_\G(U),\mathcal{V}_{\R}(4))
$$
is injective. As the last space above is canonically isomorphic to $$\mathrm{Ext}^1_{\mathrm{MHS}_{\R}^+}(\R(0), H^6_{B}(\Sh_\G(U),\mathcal{V}_{\R}(4))),$$
we obtain a natural injective map
$$
H^6_{\mathcal{M}}(\Sh_\G(U), \mathscr{V}^\lambda_F(3))_{hom} \rightarrow \mathrm{Ext}^1_{\mathrm{MHS}_{\R}^+}(\R(0), H^6_{B}(\Sh_\G(U),\mathcal{V}_{\R}(4))).
$$
Composing this map with the canonical injection 
$$
\mathrm{Ext}^1_{\mathrm{MHS}_{\R}^+}(\R(0), H^6_{B}(\Sh_\G(U),\mathcal{V}_{\R}(4))) \rightarrow H^7_\mathcal{H}(\Sh_\G(U), \mathcal{V}^\lambda_{\R}(4))
$$
we obtain the map \eqref{cycleHodge}. We denote by $\overline{ \rm cl}_{\mathcal{H}} : H^6_{\mathcal{M}}(\Sh_\G(U), \mathscr{V}^\lambda_F(3)) \rightarrow H^7_\mathcal{H}(\Sh_\G(U), \mathcal{V}^\lambda_{\R}(4))$ the composition of the map \eqref{cycleHodge} with the projection $H^6_{\mathcal{M}}(\Sh_\G(U), \mathscr{V}^\lambda_F(3))  \rightarrow H^6_{\mathcal{M}}(\Sh_\G(U), \mathscr{V}^\lambda_F(3))_{hom} $.
 
\begin{definition} \label{HodgeCycleClass}
We define \[\mathcal{Z}_{\H, \mathcal{H}}^{[\lambda,\mu]} := \overline{ \rm cl}_{\mathcal{H}} (\mathcal{Z}_{\H, \mathcal{M}}^{[\lambda,\mu]}) \in H^7_\mathcal{H}(\Sh_\G(U), \mathcal{V}^\lambda_{\R}(4)).\] 
\end{definition} 
 
\begin{remark}
Let $\pi$ be a cuspidal automorphic representation of $\PGSp_6(\A)$ which satisfies the hypotheses of Lemma \ref{exloc} and let $S$ be a finite set of places containing the ramified places of $\pi_f$ and $\infty$. By the conjectures of Beilinson and Tate and the local calculations of Gross and Savin in \cite{GrossSavin}, there should exist a cubic \'etale algebra $E/\Q$ such that the cycle $\mathcal{Z}_{\H, \mathcal{H}}^{[\lambda,\mu]}$, with $\H$ defined by $E/\Q$, and their Hecke translates are expected to generate $H^1_\mathcal{H}(M(\pi_f)_\R(4))$ when ${\rm ord}_{s=1} L^S(s, \pi, {\rm Spin})= -1$. Assuming the non-vanishing of the archimedean integral, Corollary \ref{theoremcyclebetti2} confirms this expectation. 
\end{remark}

 \section{Construction of the differential form and pairing with the motivic class} \label{Section4}
 
The purpose of this chapter is to study the Betti and Hodge realizations of the cycle constructed in \S \ref{subsec:tateclass} by relating their pairing with a suitable cuspidal harmonic differential form to an automorphic period. 

\subsection{Test vectors}
 Recall that the discrete series $L$-packets for $\PGSp_6(\R)$ have four elements, each indexed by a Hodge type (and its conjugate).
 Let $\pi$ denote a cuspidal automorphic representation of 
 $\PGSp_6(\A)$ for which $\pi_\infty$ is the discrete series of Hodge type (3,3) in the $L$-packet of $V^\lambda$ where $\lambda=(\lambda_1, \lambda_2, \lambda_3,0)$ and $\lambda_1=\lambda_2+\lambda_3$. This translates into saying that $\pi$ is a cuspidal automorphic representation of $\G(\A)$ with trivial central character for which
\[ H^6(\mathfrak{g},K_\G; \pi_\infty \otimes V^\lambda) \neq 0,\] and such that    $\pi_\infty|_{\G_0(\R)} = \pi_{\infty, 1}^{3,3} \oplus \overline{\pi}_{\infty,1}^{3,3}$ is the direct sum of the discrete series representations of respective Harish Chandra parameters $(\lambda_2+2, \lambda_3+1, -\lambda_1-3)$ and $(\lambda_1+3, -\lambda_3-1, -\lambda_2-2)$. Recall that these discrete series contain with multiplicity one their minimal $K_\infty$-types $\tau_{(\lambda_2+2, \lambda_3+2,-\lambda_1-4)}$ and $\tau_{(\lambda_1+4, -\lambda_3-2, -\lambda_2-2)}$ respectively. On the other hand, as $K_\infty$-representations we have 
\[ \bigwedge^6 \mathfrak{p}_{\C} \supseteq \bigwedge^3 \mathfrak{p}_\C^+ \otimes \bigwedge^3 \mathfrak{p}_{\C}^- = \bigoplus_i \tau_i \supseteq \tau_{(2,2,-4)} \oplus \tau_{(4, -2, -2)}, \]
where the equality expresses the decomposition of the tensor product into irreducible $K_\infty$-representations.
This fact will be useful to construct an element in
$$
H^6(\mathfrak{g},K_\G; \pi_\infty \otimes V^\lambda)=\Hom_{K_\infty}\left(\bigwedge^6 \mathfrak{p}_{\C}, \pi_\infty \otimes V^\lambda \right) \simeq \Hom_{K_\infty}\left(\bigwedge^6 \mathfrak{p}_{\C} \otimes V^\lambda, \pi_\infty \right),
$$
where the last equality follows from the fact that $V^\lambda$ is self-dual. Before stating the next result, let us fix the following data.
\begin{itemize}
    \item A highest weight vector $\Psi_\infty$, resp. $\overline{\Psi}_\infty$, of the minimal $K_\infty$-type $\tau_{(\lambda_2+2, \lambda_3+2,-\lambda_1-4)}$, resp. $\tau_{(\lambda_1+4, -\lambda_3-2, -\lambda_2-2)}$ of $\pi_{\infty, 1}^{3,3}$, resp. $\overline{\pi}_{\infty, 1}^{3,3}$.
     \item  A highest weight vector $X_{(2,2,-4)}$, resp. $X_{(4,-2,-2)}$, of $\tau_{(2,2,-4)}$, resp. $\tau_{(4,-2,-2)}$.
     \item  A highest weight vector $v^{\lambda'}$, resp. $v^{\overline{\lambda}'}$, of $\tau_{\lambda'} \subseteq V^\lambda$, resp. $\tau_{\overline{\lambda}'} \subseteq V^\lambda$, where $\tau_{\lambda'}$ and $\tau_{\overline{\lambda}'}$ denote the irreducible algebraic $K_\infty$-representations of highest weight $\lambda' = (\lambda_2, \lambda_3, -\lambda_1)$ and $\overline{\lambda}' = (\lambda_1, -\lambda_3, -\lambda_2)$ respectively.
 \end{itemize}
 
 \begin{lemma} \label{test-vector} The spaces $\Hom_{K_\infty}(\bigwedge^6 \mathfrak{p}_{\C} \otimes V^\lambda, \pi_{\infty, 1}^{3,3})$ and $\Hom_{K_\infty}(\bigwedge^6 \mathfrak{p}_{\C} \otimes V^\lambda, \overline{\pi}_{\infty, 1}^{3,3})$
 are of dimension one and the elements $$\omega_{\Psi_\infty} \in \Hom_{K_\infty}(\wedge^6 \mathfrak{p}_{\C} \otimes V^\lambda, \pi_{\infty,1}^{3,3}), \omega_{\overline{\Psi}_\infty} \in \Hom_{K_\infty}(\wedge^6 \mathfrak{p}_{\C} \otimes V^\lambda, \overline{\pi}_{\infty,1 }^{3,3})$$ defined by \[ \omega_{\Psi_\infty}(X_{(2,2,-4)} \otimes v^{\lambda'})=\Psi_\infty \]
 \[ \omega_{\overline{\Psi}_\infty}(X_{(4, -2,-2)} \otimes v^{\overline{\lambda}'}) = \overline{\Psi}_\infty \]
 are generators of these spaces.
 \end{lemma}
 
 \begin{proof}
 This is a consequence of \cite[Theorem II.5.3 b)]{Borel-Wallach} and its proof.
 \end{proof}

\subsection{Restriction to $\H$} \label{sectionrestriction} Let $\lambda=(\lambda_1,\lambda_2,\lambda_3,0)$, with $\lambda_1 = \lambda_2+\lambda_3$ and let $V^\lambda$ be as above. Let $\mathfrak{h}$, resp. $\mathfrak{k}_{\H}$ denote the Lie algebra of $\H(\R)$, resp. the maximal compact modulo the center $K_\H$. Observe that via the embedding $\iota: \H(\R) \hookrightarrow \G(\R)$, the group $K_\H$ is isomorphic to $T_\infty$. One has a Cartan decomposition $\mathfrak{h}_\C = \mathfrak{k}_{\H,\C} \oplus  \mathfrak{p}_{\H, \C}$, where $\mathfrak{p}_{\H, \C}$ is 6-dimensional and is spanned by the non-compact root spaces.  We fix once and for all a generator $X_0$ of the $1$-dimensional $\C$-vector space $\bigwedge^6 \mathfrak{p}_{\H, \C} \subseteq \bigwedge^6 \mathfrak{p}_{\C}$ as in \cite[\S 5.2]{CLR}. The main result of this section is the following.

\begin{theorem} \label{theo:non-vanishing} 
Let $ \omega_{ \Psi_\infty}$ and $ \omega_{\overline{\Psi}_\infty}$ be the elements of  $\Hom_{K_\infty}\left(\bigwedge^6 \mathfrak{p}_{\C} \otimes V^\lambda, \pi_\infty\right)$ defined in Lemma \ref{test-vector}. Let $X_0$ be as above and let $v$ be any $\H$-invariant vector in $V^\lambda$. Then
\[ \omega_{\Psi_\infty}\left(X_0 \otimes  v \right) \ne 0, \] \[ \omega_{\overline{\Psi}_\infty}\left(X_0 \otimes  v \right) \ne 0. \]
\end{theorem}

The proof of Theorem \ref{theo:non-vanishing} will be constructive and occupies the rest of this section. We start by recalling the following result.

\begin{lemma}[{\cite[Lemma 5.4]{CLR}}] \label{Lemma:projX0}
Let $X_0$ be as above. Then the image of $X_0$ by
$$
\bigwedge^6 \mathfrak{p}_{\H, \C} \rightarrow \bigwedge^6 \mathfrak{p}_{\C} \rightarrow \bigwedge^3 \mathfrak{p}_\C^+ \otimes \bigwedge^3 \mathfrak{p}_{\C}^- \rightarrow \tau_{(2,2,-4)},
$$
where the first map is induced by the embedding $\H \rightarrow \G$ and the second and the third maps are the natural projections, is non zero.
\end{lemma}

We next study the interaction between the branching laws of $V^\lambda$ to the subgroup $\H$ of $\G$ and to its maximal compact subgroup. More precisely, we show that the $\H$-invariant vectors constructed in Lemma \ref{lemma:choiceofsomehighestwv} project non-trivially to $\tau_{\lambda'}$ and $\tau_{\overline{\lambda}'}$ and moreover that their projections form  a basis of the corresponding $(0,0,0)$-weight spaces for the action of $T_\infty$.
 
\begin{lemma}\label{lemma:multiplicityK000}
Let $\tau_{\lambda'}$ and $\tau_{\overline{\lambda}'}$ be the irreducible algebraic sub-$K_\infty$-representations of $V^\lambda$ of highest weight $\lambda'=(\lambda_2, \lambda_3, -\lambda_1)$ and $\overline{\lambda}'=(\lambda_1, -\lambda_3, -\lambda_2)$. Then the weight $(0,0,0)$ appears in both $\tau_{\lambda'}$ and $\tau_{\overline{\lambda}'}$ with multiplicity $\lambda_2-\lambda_3 + 1$. 
\end{lemma}
 
\begin{proof}
Let $n_0(\lambda')$ denote the multiplicity of the weight $(0,0,0)$ in $\tau_{\lambda'}$. Kostant multiplicity formula reads as \[ n_0(\lambda') = \sum_{w \in \mathfrak{W}_{K_\infty}} (-1)^{\ell(w)} P( w  (\lambda' + \rho_{K_\infty}) - \rho_{K_\infty}),  \] 
 where $\rho_{K_\infty}= \tfrac{1}{2}\sum_{\alpha \in \Delta^+_c} \alpha = (1,0,-1)$ and the function $\mu \mapsto P(\mu)$ calculates the number of ways for which the weight $\mu$ can be expressed as a linear combination 
 \[\alpha (e_1-e_2) + \beta(e_1-e_3) + \gamma (e_2-e_3),  \]
 with $\alpha,\beta,\gamma \in \Z_{\geq 0}$ (cf. \cite{FultonHarris}). Using this formula, it is a tedious but straightforward calculation to verify that  $n_0(\lambda') = \lambda_2 - \lambda_3 + 1$ and the same for $\overline{\lambda}' = w_8 \lambda'$.

\end{proof}

According to Lemma \ref{lemma:multiplicityK000}, there are $\lambda_2-\lambda_3+1$ linearly independent vectors of weight $(0,0,0)$ in $\tau_{\lambda'}$. We now show that these weight vectors correspond one to one to the $\H$-invariant vectors of Lemma \ref{lemma:branchinglaw}.
 
\begin{lemma}\label{lemma:KweightsforHvectors}
Let $v, w$ be the vectors of $V^{(1,1,0)}$ and let $z$ be the vector of $V^{(2,1,1)}$ defined in Lemma \ref{vwz}. The irreducible algebraic representation $\tau_{(1,0,-1)}$, resp. $\tau_{(1,1,-2)}$ and $\tau_{(2,-1,-1)}$, appear in the restriction of $V^{(1,1,0)}$, resp. of $V^{(1,1,1)}$, to $K_\infty$ with multiplicity $1$. Moreover, we have $v,w \in \tau_{(1,0,-1)} \subseteq V^{(1,1,0)}$, and $z \in \tau_{(1,1,-2)} \oplus \tau_{(2,-1,-1)} \subseteq V^{(1,1,1)}$, with $z$ projecting non-trivially to each factor of this decomposition. 
\end{lemma} 
 
\begin{proof}
First observe that  $v,w \in V^{(1,1,0)}$ and $z \in V^{(2,1,1)}$ are vectors of weight $(0,0,0)$ both for the split and the compact tori of $\G_0(\R)$. Indeed these vectors are fixed (up to a constant) by the matrix $J$ sending the non-compact torus $\mathbf{T}_0$ to the compact torus $T_\infty$ defined in \S \ref{SectCompactLie}.  Using branching laws from $\G_0(\R)$ to $K_\infty$, we have a decomposition of $K_\infty$-representations  \[V^{(1,1,0)} = \tau_{(1,1,0)} \oplus \tau_{(1,0,-1)} \oplus \tau_{(0,-1,-1)}.  \]
The weight $(0,0,0)$ appears only in $\tau_{(1,0,-1)}$
and with multiplicity 2. Since it has also multiplicity $2$ in $V^{(1,1,0)}$, we deduce that $\{ v,w \}$ forms a basis for the $(0,0,0)$-eigenspace of $\tau_{(1,0,-1)}$. On the other hand, we have
\begin{align*} V^{(2,1,1)} &= \tau_{(-1,-1,-2)} \oplus \tau_{(1,-1,-2)} \oplus \tau_{(1,1,0)} \oplus \tau_{(1,1,-2)} \\ &\,\,\,\,\,\,\,\,\, \oplus   \tau_{(1,0,-1)}\oplus \tau_{(2,-1,-1)} \oplus \tau_{(2,1,-1)} \oplus \tau_{(2,1,1)} \oplus \tau_{(0,-1,-1)}. \end{align*} 
The weight $(0,0,0)$ only appears in $\tau_{(1,1,-2)}  \oplus   \tau_{(1,0,-1)}\oplus \tau_{(2,-1,-1)}$, which implies that \[z \in \tau_{(1,1,-2)}  \oplus   \tau_{(1,0,-1)}\oplus \tau_{(2,-1,-1)}.\]
  Notice that the decomposition of the standard representation of $\G_0$ \[ V  = \tau_{(1,0,0)} \oplus \tau_{(0,0,-1)} \] of $K_\infty$-representations can be realized by picking the basis $\{v_1, v_2, v_3, w_1, w_2, w_3\}$, where $v_r:= e_r+i f_r$ and $w_r:= i e_r + f_r$. The set $\{v_r \}_{1 \leq r \leq 3}$, resp. $\{w_r \}_{1 \leq r \leq 3}$, defines a basis for $\tau_{(1,0,0)}$, resp. $\tau_{(0,0,-1)}$. We now write $z$ in terms of this basis. By using the relations \begin{align*}
    e_r   &= \tfrac{1}{2} v_r - \tfrac{i}{2} w_r, \\
    f_r  &= \tfrac{1}{2} w_r - \tfrac{i}{2} v_r, 
\end{align*}
we have that \[ e_1 \otimes f_1 \wedge (e_2 \wedge f_2  - e_3 \wedge f_3  ) - f_1  \otimes e_1  \wedge (e_2  \wedge f_2  - e_3  \wedge f_3  )  \]
equals to 
\[ \frac{1}{4} \left(  v_1 \otimes w_1 \wedge (v_2 \wedge w_2 - v_3 \wedge w_3 ) - w_1 \otimes v_1 \wedge (v_2 \wedge w_2 - v_3 \wedge w_3 ) \right). \]
Thus, 
\[ z= z_1 - z_2 =  \frac{1}{4} \left(  v_1 \cdot w_1 \wedge (v_2 \wedge w_2 - v_3 \wedge w_3 ) - w_1 \cdot v_1 \wedge (v_2 \wedge w_2 - v_3 \wedge w_3 ) \right). \]
Notice that the vector $w_1 \wedge (v_2 \wedge w_2 - v_3 \wedge w_3 ) \in V^{(1,1,1)}$ is of weight $(-1,0,0)$ for $T_\infty$, while  $v_1 \wedge (v_2 \wedge w_2 - v_3 \wedge w_3 ) \in V^{(1,1,1)}$ is of weight $(1,0,0)$ for $T_\infty$. As \[ V^{(1,1,1)} = \tau_{(1,1,1)} \oplus \tau_{(1,-1,-1)} \oplus \tau_{(1,1,-1)} \oplus \tau_{(-1,-1,-1)},  \]
 and the fact that the weight $(-1,0,0)$ appears only in  $\tau_{(1,-1,-1)}$ and  $(1,0,0)$ only in  $\tau_{(1,1,-1)}$, we have that 
 \[w_1 \wedge (v_2 \wedge w_2 - v_3 \wedge w_3 ) \in \tau_{(1,-1,-1)}\] \[v_1 \wedge (v_2 \wedge w_2 - v_3 \wedge w_3 ) \in \tau_{(1,1,-1)} \] 
 By the properties of the Cartan product, the vector $ s_1 : = v_1 \cdot w_1 \wedge (v_2 \wedge w_2 - v_3 \wedge w_3 )$ is non-zero in $\tau_{(2,-1,-1)}$, while $ s_2 : = w_1 \cdot v_1 \wedge (v_2 \wedge w_2 - v_3 \wedge w_3 )$ is non-zero in $\tau_{(1,1,-2)}$. This shows that the vector $z \in V^{(2,1,1)}$ lives in $\tau_{(2,-1,-1)} \oplus \tau_{(1,1,-2)}$, thus finishing the proof.
\end{proof}

\begin{proposition} \label{prop:doublebranch} 
The set $\{ {\rm pr}_{\tau_{\lambda'}}( v^{[\lambda, \mu]})\}_{\mu}$, resp. $ \{ {\rm pr}_{\tau_{\overline{\lambda}'}}( v^{[\lambda, \mu]})\}_{\mu}$, forms a basis of the weight $(0,0,0)$-eigenspace of $\tau_{{\lambda}'} \subset V^\lambda$, resp. $\tau_{\overline{\lambda}'} \subset V^\lambda$.
\end{proposition} 
 
\begin{proof}


Recall that we have defined  \[ v^{[\lambda, \mu]}:= v^{\lambda_2 - \mu} \cdot w^{ \mu - \lambda_3} \cdot z^{\lambda_3} \in F(0) \subseteq (V^\lambda)_{|_\H}. \]
By Lemma \ref{lemma:KweightsforHvectors}, we have that $v,w \in \tau_{(1,0,-1)} \subseteq V^{(1,1,0)}$ so that, for any $\lambda_3 \leq \mu \leq \lambda_2$, we have $v^{\lambda_2 - \mu} \otimes w^{ \mu - \lambda_3} \in \tau_{(1,0,-1)}^{\otimes \lambda_2 - \lambda_3}$ and we deduce that the projection of $v^{\lambda_2 - \mu} \cdot w^{ \mu - \lambda_3} \in V^{(\lambda_2 - \lambda_3, \lambda_2 - \lambda_3, 0)}$ to $\tau_{(\lambda_2-\lambda_3 , 0 , \lambda_3 - \lambda_2)}$ coincides with their Cartan product with respect to $K_\infty$. Moreover, each of these projections is non-zero because of Lemma \ref{lemma:sun}. Since the vectors
\[v^{\lambda_2 - \mu} \cdot w^{ \mu - \lambda_3} \in \tau_{(\lambda_2-\lambda_3 , 0 , \lambda_3 - \lambda_2)}\]
are all different as they live in different $\H'_0$ sub-representations (cf. the proof of Lemma \ref{lemma:choiceofsomehighestwv}), we conclude that they span the $ \lambda_2 - \lambda_3 + 1 $-dimensional weight $(0,0,0)$-eigenspace of $\tau_{(\lambda_2-\lambda_3 , 0 , \lambda_3 - \lambda_2)}$. We now show that  $z^{\lambda_3}$ projects non-trivially to both $\tau_{(2 \lambda_3, - \lambda_3, - \lambda_3)}$ and $\tau_{(\lambda_3 , \lambda_3, -2 \lambda_3)}$. 
Notice that, as the weights $(2\lambda_3,-\lambda_3,-\lambda_3)$ and $(\lambda_3,\lambda_3,-2\lambda_3)$ are extremal in $V^{(2\lambda_3, \lambda_3 , \lambda_3)}$ and appear uniquely, we have a commutative diagram 
\[\xymatrix{
(V^{(2,1,1)})^{ \otimes \lambda_3} \ar@{->>}[r]^{\cdot} \ar@{->>}[d]_{({\rm pr}_1,{\rm pr}_1')} &   V^{(2\lambda_3, \lambda_3 , \lambda_3)}\ar@{->>}[d]^{{\rm pr}_2}\\ 
(\tau_{(2,-1,-1)})^{ \otimes \lambda_3}  \oplus (\tau_{(1,1,-2)})^{ \otimes \lambda_3}  \ar@{->>}[r]^{\cdot} & 
 \tau_{(2\lambda_3,-\lambda_3,-\lambda_3)} \oplus \tau_{(\lambda_3,\lambda_3,-2\lambda_3)}, }\]
where the horizontal arrows are the Cartan projections and the vertical arrows are the natural projections given by the decomposition of $V^{(2r,r,r)}$ as $K_\infty$-representations. Thanks to the commutativity of the diagram, we know that the vector $z^{\otimes \lambda_3 } \in (V^{(2,1,1)})^{ \otimes \lambda_3}$ maps to \[ {{\rm pr}_2}(z^{\lambda_3}) = {{\rm pr}_1}(z)^{\lambda_3} + {{\rm pr}_1'}(z)^{\lambda_3} = s_1^{\lambda_3} + s_2^{\lambda_3}, \]
where $s_1,s_2$ are as in Lemma \ref{lemma:KweightsforHvectors}. This shows, again by Lemma \ref{lemma:sun}, that each $v^{[\lambda, \mu]} $ projects non-trivially to both $\tau_{\lambda'}$ and $\tau_{\overline{\lambda}'}$ and that each of these projections are different by Lemma \ref{lemma:choiceofsomehighestwv}. Indeed,   
\[ {\rm pr}_{\tau_{\lambda'}}( v^{[\lambda, \mu]}) = v^{\lambda_2 - \mu} \cdot w^{ \mu - \lambda_3} \cdot s_1^{\lambda_3}, \] \[ {\rm pr}_{\tau_{\overline{\lambda}'}}( v^{[\lambda, \mu]}) = v^{\lambda_2 - \mu} \cdot w^{ \mu - \lambda_3} \cdot s_2^{\lambda_3}. \]
By Lemma \ref{lemma:multiplicityK000}, this means that $ \{ {\rm pr}_{\tau_{\lambda'}}( v^{[\lambda, \mu]})\}_{\mu}$ (resp. $ \{ {\rm pr}_{\tau_{\overline{\lambda}'}}( v^{[\lambda, \mu]})\}_{\mu}$) defines a basis of the weight $(0,0,0)$-eigenspace of $\tau_{{\lambda}'}$, resp. $\tau_{\overline{\lambda}'}$. This finishes the proof.
\end{proof}
 
We are now in condition of finishing the proof of Theorem \ref{theo:non-vanishing}

\begin{proof}[Proof of Theorem \ref{theo:non-vanishing}]
By construction, the map $\omega_{\Psi_\infty}$ factors through $\tau_{\lambda' + (2,2,-4)} \subseteq \tau_{(2,2,-4)} \otimes \tau_{\lambda'}$. Lemma \ref{Lemma:projX0} shows that the projection of $X_0$ to $\tau_{(2,2,-4)}$ is non-zero, while Proposition \ref{prop:doublebranch} shows that $\mathrm{pr}_{\tau_{\lambda'}}(v^{[\lambda, \mu]})$ is non-zero. Since $\tau_{\lambda' + (2,2,-4)}$ is the Cartan product of $\tau_{(2,2,-4)}$ and $\tau_{\lambda'}$, we deduce from Lemma \ref{lemma:sun} that the image of the pure tensor $\mathrm{pr}_{\tau_{(2,2,-4)}}(X_0) \otimes \mathrm{pr}_{\tau_{\lambda'}}(v^{[\lambda, \mu]})$ is non-zero.

\end{proof}
 
\subsection{The pairing}\label{ss:thepairing} 
Let $\pi$ denote a cuspidal automorphic representation of 
 $\PGSp_6(\A)$ for which $\pi_\infty$ is the discrete series of Hodge type (3,3) in the $L$-packet of $V^\lambda$ with $\lambda=(\lambda_2+\lambda_3, \lambda_2, \lambda_3,0)$.    
Let $\Psi=\Psi_\infty \otimes \Psi_f$ denote a cusp form in $\pi=\pi_\infty \otimes \pi_f$. We assume that $\Psi_\infty$ is a highest weight vector of the minimal $K_\infty$-type $\tau_{(\lambda_2+2, \lambda_3+2, -\lambda_1-4)}$ of $\pi_\infty |_{\G_0(\R)}$. We let $[\omega_{\Psi_\infty}] \in H^6(\mathfrak{g},K_\G; \pi_\infty \otimes V^\lambda)$ be the cohomology class of the harmonic differential form $\omega_{\Psi_\infty}$ defined in Lemma \ref{test-vector}. We also assume that $\Psi_f \in V_{\pi_f} $ is $U$-invariant. Then we have
$
[\omega_\Psi] := [\omega_{\Psi_\infty} \otimes \Psi_f] \in H^6(\mathfrak{g},K_\G; \pi^U \otimes V^\lambda).
 $
 
 \begin{lemma}\label{lemmaonthetestvectorinsidecohomology}
There is a Hecke-equivariant inclusion
 $$
 H^6(\mathfrak{g},K_\G; \pi^U \otimes V^\lambda) \subset H^6_{dR, c}(\Sh_{\G}(U), \mathcal{V}^\lambda_\C).
 $$
Moreover, if $\pi_w$ is the Steinberg representation for some finite palce $w$, such inclusion is unique.
 \end{lemma}
 
 \begin{proof}
 Let $\mathcal{C}^\infty_{rd}(\G(\Q) \backslash \G(\A)/U, V^\lambda)$ denote the space of $V^\lambda$-valued $\mathcal{C}^\infty$-functions on the double quotient $\G(\Q) \backslash \G(\A)/U$ which, together with all their right $\mathfrak{U}(\mathfrak{g}_{\C})$-derivatives, are rapidly decreasing in the sense of \cite{harris}. As $\pi$ is cuspidal and cusp forms are rapidly decreasing, we have 
 $
 H^6(\mathfrak{g},K_\G; \pi_\infty \otimes V^\lambda_\C)^{m(\pi)} \otimes \pi_f^U\subset H^6(\mathfrak{g},K_\G;  \mathcal{C}^\infty_{rd}(\G(\Q) \backslash \G(\A)/U, V^\lambda)).$ 
 Thus the result follows from the fact that, according to \cite[Theorem 5.2]{borel2} (see also \cite[Theorem 1.4.1]{harris}), there exists a canonical Hecke equivariant isomorphism
 $
 H^6_{dR, c}(\Sh_{\G}(U), \mathcal{V}^\lambda) \simeq H^6(\mathfrak{g},K_\G; \mathcal{C}^\infty_{rd}(\G(\Q) \backslash \G(\A)/U, V^\lambda)).
 $ Finally, if $\pi_w$ is Steinberg at a finite place $w$, we have, as in Lemma \ref{lemmaonBettiisotypicpart}, that $m(\pi)=1$.
 \end{proof}

\subsubsection{The pairing in Betti cohomology} Poincar\'e duality is a perfect pairing
$$
\langle \, , \, \rangle : H^{6}_{B}(\Sh_{\G}(U), \mathcal{V}_F^\lambda(3)) \times H^{6}_{B,c}(\Sh_{\G}(U), \mathcal{V}_F^\lambda) \to F(-3),
$$
which is a morphism of mixed $F$-Hodge structures.
Fix the choice of a measure $dh$ on $\H(\A)$ as follows. For each finite place $p$, we take the Haar measure $d h_p$ on $\H(\qp)$ that assigns volume $1$ to $\H(\zp)$. For the archimedean place, we let $X_0 \in \bigwedge^6 \mathfrak{p}_{\H, \C}$ be the generator fixed at the beginning of section \ref{sectionrestriction}. The choice of $X_0$ induces an equivalence between top differential forms on $X_\H = \H(\R) / K_{\H, \infty}$ and invariant measures $d h_\infty$ on $\H(\R)$ assigning measure one to $K_{\H, \infty}$ (cf. \cite[p. 83]{Harris97} for details). We let $dh_\infty$ denote the measure associated in this way to the pullback of $\iota^{[\lambda, \mu]*}\omega_\Psi$ to $X_\H$ and we then define $d h = dh_\infty \prod_{p} dh_p$.

\begin{prop} \label{period} We have
\[\langle  \mathcal{Z}_{\H, B}^{[\lambda,\mu]} , [\omega_\Psi] \rangle = \frac{h_{{U'}}}{(2\pi i)^3\cdot \mathrm{vol}({U'})}
 \int_{\H(\Q) \Z_\G(\A) \backslash \H(\A)} A^{[\lambda,\mu]} \cdot \Psi(h)dh,\] 
where $h_{U'} =4^{-1} | \Z_\G(\Q) \backslash \Z_\G(\Af) / (\Z_\G(\Af) \cap {U'})|$ and $A^{[\lambda,\mu]} \in U(\mathfrak{k}_\C)$ is an element for which $A^{[\lambda,\mu]} \cdot \Psi_\infty = \omega_{\Psi_\infty}(X_0 \otimes v^{[\lambda,\mu]})$.
\end{prop}

\begin{proof}
By \cite[Corollary 5.5]{borel2}, there exists a $\mathcal{V}^\lambda$-valued rapidly decreasing differential form $\eta$ of degree five on $\Sh_\G(U)$ such that $\omega_c := \omega_\Psi + d\eta$ is compactly supported. We have
\begin{align*}
    \langle  \mathcal{Z}_{\H, B}^{[\lambda,\mu]} , [\omega_\Psi] \rangle &=\langle  {\rm cl}_B(\iota_*^{[\lambda,\mu]} \mathbf{1}_{\Sh_{\H}({U'})})  , [\omega_c] \rangle \\
    &=\langle \iota_*^{[\lambda,\mu]} {\rm cl}_B( \mathbf{1}_{\Sh_{\H}({U'})}) , [\omega_c] \rangle \\ 
    &=\langle {\rm cl}_B( \mathbf{1}_{\Sh_{\H}({U'})}) , \iota^{[\lambda,\mu]*} [\omega_c] \rangle \\
    &=\frac{1}{(2\pi i)^3}
 \int_{\Sh_{\H}({U'})} \iota^{[\lambda,\mu]*} \omega_c,
\end{align*}
where $\iota^{[\lambda,\mu]*}: \iota^* V^\lambda \rightarrow F(0)$ is the $\H$-equivariant projection dual to the inclusion $\iota^{[\lambda,\mu]} :F(0) \rightarrow \iota^* V^\lambda$ defined by $1 \mapsto v^{[\lambda, \mu]} \in V^\lambda$, where $v^{[\lambda, \mu]}$ is the vector defined in Lemma \ref{lemma:choiceofsomehighestwv}. According to \cite[\S 5.6]{borel2}, we have $$\int_{\Sh_{\H}({U'})} \iota^{[\lambda,\mu]*} d\eta=0.$$ Hence, using Theorem \ref{theo:non-vanishing} we have
\begin{align*}
    \langle  \mathcal{Z}_{\H, B}^{[\lambda,\mu]} , [\omega_\Psi] \rangle
 &=\frac{1}{(2\pi i)^3} \int_{\Sh_{\H}({U'})} \iota^{[\lambda,\mu]*}\omega_\Psi\\
 &= \frac{1}{(2\pi i)^3} \int_{\Sh_{\H}({U'})} \omega_\Psi\left(X_0 \otimes v^{[\lambda, \mu]}\right)(h)dh\\
 &=\frac{1}{(2\pi i)^3}
 \int_{\H(\Q) \backslash \H(\A) / \Z_\H(\R) K_{\H, \infty} {U'}} A^{[\lambda,\mu]} \cdot \Psi(h)dh \\
 &=\frac{h_{{U'}}}{(2\pi i)^3}
 \int_{\H(\Q) \Z_\G(\A) \backslash \H(\A) / {U'}} A^{[\lambda,\mu]} \cdot \Psi(h)dh \\
  &=\frac{h_{{U'}}}{(2\pi i)^3\cdot \mathrm{vol}({U'})}
 \int_{\H(\Q) \Z_\G(\A) \backslash \H(\A)} A^{[\lambda,\mu]} \cdot \Psi(h)dh,
\end{align*}
where the third equality follows from Theorem \ref{theo:non-vanishing} as $\omega_{\Psi_\infty}(X_0 \otimes v^{[\lambda, \mu]})$ is non-zero and thus it is of the form $A^{[\lambda,\mu]} \cdot \Psi_\infty$, for some $A^{[\lambda,\mu]} \in U(\mathfrak{k}_\C)$, because $\Psi_\infty$ is the highest weight vector of the minimal $K_\infty$-type $\tau_{(\lambda_2 +2, \lambda_3 + 2, - \lambda_1 -4)}$.  Moreover, the fourth equality follows from the fact that $\Psi$ is fixed by the center of $\G$, whence, using that $|\Z_\H(\R) /(\Z_\G \cap \H)(\R)|=4$, the constant $h_{U'}$ is equal to $4^{-1} | \Z_\G(\Q) \backslash \Z_\G(\Af) / (\Z_\G(\Af) \cap {U'})|$.
\end{proof}
 
\begin{remark}
In view of Proposition \ref{period}, we immediately notice that if $\pi$ is not $\H$-distinguished, namely \[ \int_{\H(\Q) \Z_\G(\A) \backslash \H(\A)} \varphi_\pi(h)dh = 0,\]
for any cusp form $\varphi_\pi$ in the space of $\pi$, we have that ${\rm pr}_{\pi^\vee}\mathcal{Z}_{\H, B}^{[\lambda,\mu]} = 0$.
As we discuss later in \S \ref{sec:cycleclassformulaandstandardlvalues}, the $\H$-distinguishability is related to the property of $\pi$ being a (functorial) lift from $G_2$, which is (conjecturally) equivalent to the fact that the Spin $L$-function of $\pi$ has a pole at $s=1$.
\end{remark}

\subsubsection{The pairing in absolute Hodge cohomology}

Let
\[ \langle \, , \, \rangle_\mathcal{H} : H^{7}_{\mathcal{H}}(\Sh_{\G}(U)/\R, \mathcal{V}_\R(4)) \times H^{6}_{\mathcal{H}, c}(\Sh_{\G}(U)/\R, \mathcal{V}_\R(3)) \to \R \]
be the natural pairing between absolute Hodge cohomology and compactly supported cohomology as constructed in \cite[\S 4.2]{BeilinsonAHC}. In order to ease notations, we will denote by $H^*_{B,c}(i)$ and $H^*_{B}(i)$ the cohomology groups $H^*_{B, c}(\Sh_\G(U), \mathcal{V}_F(i))$ and $H^*_B(\Sh_\G(U), \mathcal{V}_F(i))$, respectively. Recall from \S \ref{sectAHC} that absolute Hodge cohomology and compactly supported cohomology live in exact sequences
\begin{equation} \label{eqAHC}
0 \rightarrow \mathrm{Ext}^1_{\mathrm{MHS}_{\R}}(\R(0), H^6_B(4)) \to
H^{7}_{\mathcal{H}}(\Sh_{\G}(U), \mathcal{V}_\R(4)) \to \mathrm{Hom}_{\mathrm{MHS}_{\R}}(\R(0), H^7_B(4)) \rightarrow 0,
\end{equation}
\begin{equation} \label{eqcsAHC}
0 \rightarrow \mathrm{Ext}^1_{\mathrm{MHS}_{\R}}(\R(0), H^5_{B,c}(3)) \to
H^{6}_{\mathcal{H}, c}(\Sh_{\G}(U), \mathcal{V}_\R(3)) \to \mathrm{Hom}_{\mathrm{MHS}_{\R}}(\R(0), H^6_{B, c}(3)) \rightarrow 0,
\end{equation}
which are deduced from the description of absolute Hodge cohomology as a cone of a diagram of complexes of Hodge structures. Let $[\omega_\Psi] \in H^6_{B, c}(\Sh_{\G}(U), \mathcal{V}_\R^\lambda(3))$ be the compactly supported cohomology class of the harmonic differential form $\omega_\Psi$. This class is of Hodge type $(3,3)$ and hence, since $W_0 H^6_{B, c}(3) = H^6_{B, c}(3)$, it naturally lives in the space $\mathrm{Hom}_{\mathrm{MHS}_{\R}}(\R(0), H^6_{B, c}(3)) = W_0 H_{B, c}(3) \cap F^{0} H_{B, c}(3)_\C$. Denote by $\widetilde{[\omega_\Psi]}$ any lift of $[\omega_\Psi]$ in $H^{6}_{\mathcal{H}, c}(\Sh_{\G}(U), \mathcal{V}_\R(3))$ via the surjection of the exact sequence \eqref{eqcsAHC}. 

\begin{proposition} \label{period2}
The pairing $\langle \mathcal{Z}_{\H, \mathcal{H}}^{[\lambda, \mu]}, \widetilde{[\omega_\Psi]} \rangle_{\mathcal{H}}$ depends only on $[\omega_\Psi]$ and not on the choice of lift. We denote this value by $\langle  \mathcal{Z}_{\H, \mathcal{H}}^{[\lambda,\mu]} , [\omega_\Psi] \rangle_\mathcal{H}$. Moreover, the pairing is given by the natural Poincar\'e duality pairing. In particular, we have \[\langle  \mathcal{Z}_{\H, \mathcal{H}}^{[\lambda,\mu]} , [\omega_\Psi] \rangle_\mathcal{H} = \frac{h_{U'}}{(2\pi i)^3\cdot \mathrm{vol}(U')}
 \int_{\H(\Q) \Z_\G(\A) \backslash \H(\A)} A^{[\lambda,\mu]} \cdot \Psi(h)dh.\]
\end{proposition}

\begin{proof}
We give a sketch of the proof and we refer to \cite{BeilinsonAHC} or to \cite[\S 5.1]{BKK} for the facts used here.
It follows from the description of the pairing between absolute Hodge cohomology and compactly supported cohomology given in \cite[\S 4.2]{BeilinsonAHC} that, since our cycle class $\mathcal{Z}_{\H, \mathcal{H}}^{[\lambda, \mu]}$ lives in the subspace $\mathrm{Ext}^1_{\mathrm{MHS}_{\R}}(\R(0), H^6_B(4))$ of 
$H^{7}_{\mathcal{H}}(\Sh_{\G}(U), \mathcal{V}_\R(4))$, the map
\[ \langle [\mathcal{Z}_{\H, \mathcal{H}}^{[\lambda, \mu]}],  - \rangle : H^{6}_{\mathcal{H}, c}(\Sh_{\G}(U), \mathcal{V}_\R(3)) \to \R \] factors through $\mathrm{Hom}_{\mathrm{MHS}_{\R}}(\R(0), H^6_{B, c}(3))$ and coincides with the natural Poincar\'e duality pairing 
  
\begin{eqnarray*}
\mathrm{Ext}^1_{\mathrm{MHS}_{\R}}(\R(0), H^6_B(4)) \otimes \mathrm{Hom}_{\mathrm{MHS}_{\R}}(\R(0), H^6_{B, c}(3)) &\to& \mathrm{Ext}^1_{\mathrm{MHS}_{\R}}(\R(0), H^6_B(4) \otimes H^6_{B, c}(3)) \\
&\xrightarrow{\cup} & \mathrm{Ext}^1_{\mathrm{MHS}_{\R}}(\R(0),H^{12}_{B, c}(7)) \\ 
&\xrightarrow{\mathrm{Tr}}& \mathrm{Ext}^1_{\mathrm{MHS}_{\R}}(\R(0), \R(1)) = \R,
\end{eqnarray*}
This shows the first two assertions. The last formula follows from Proposition \ref{period}.
\end{proof}

\section{Integral representation and residue of the Spin \texorpdfstring{$L$}{L}-function} \label{Section5}

In this section, using the the result of \cite{Pollack-Shah}, we explain the precise connection between the period integral appearing in the statement of Proposition \ref{period} and the residue of the spin $L$-function of $\pi$ in the case where the cubic totally real \'etale algebra $E$ over $\Q$ defining $\H$ is of the form $\Q \times F$, with $F$ a quadratic real \'etale algebra over $\Q$. We start by recalling well known analytic properties of some Eisenstein series for $\GL_2$. 

\subsection{Eisenstein series for $\GL_2$} Let $\T_2$ denote the maximal diagonal torus of $\GL_2$ and let $\B_2=\T_2 \mathbf{U}_2$ denote the standard Borel. We denote by $\delta$ the character of $\T_2$ defined by $\mathrm{diag}(t_1, t_2) \mapsto {t_1}/{t_2}$ and we regard $\delta$ as a character of $\B_2$ by extending it trivially to the unipotent radical. Let $\Phi \in \mathcal{S}(\A^2)$ be a Schwartz-Bruhat function. Following Jacquet, for any $s \in \C$, we attach to $\Phi$ the function $f_\Phi \in \mathrm{Ind}_{\B_2(\A)}^{\GL_2(\A)} \delta^s$ defined by
$$
f_\Phi(h, s)=|\det(h)|^s \int_{\A^\times} \Phi((0,t)h)|t|^{2s} d^\times t
$$
and the Eisenstein series
$$
E_\Phi(h,s)=\sum_{\gamma \in \B_2(\Q) \backslash \GL_2(\Q)} f_\Phi(\gamma h, s).
$$
In the statement of the following Lemma, we denote by $\widehat{\Phi}(0) = \int_{\A^2} \Phi(x,y)dx dy$
the value at $0$ of the Fourier transform of $\Phi$.

\begin{lemma} \label{eisenstein} \leavevmode \begin{enumerate}
    \item The Eisenstein series $E_\Phi(h,s)$ is absolutely convergent for $Re(s)$ big enough and has a meromorphic continuation to $\C$.
    \item We have
$$
E_\Phi(h,s)=\frac{|\det(h)|^{s-1} \widehat{\Phi}(0)}{2(s-1)}+R(h,s)
$$
where $R(h,s)$ is an entire function in $s$ for any $h \in \GL_2(\A)$.
\end{enumerate} 
\end{lemma}

\begin{proof}
Statement (1) is \cite[Proposition 19.2]{jacquet}. According to \cite[Lemma (4.2)]{jacquet-shalika} and its proof, we have
$$
E_\Phi(h,s)=\frac{c|\det(h)|^{s-1} \widehat{\Phi}(0)}{s-1}+R(h,s)
$$
where $R(h,s)$ is holomorphic for $Re(s)>0$ and $c=(s-1)\int_{|t| \leq 1} |t|^{2(s-1)} d^\times t$, the integral being over the set $\{t \in \A^\times/\Q^\times\,:\, |t| \leq 1\}$. By Iwasawa--Tate we have $c=\frac{1}{2}$.
\end{proof}

\subsection{Fourier coefficients}\label{Fouriercoefficientsgsp6}

Here we discuss the definition and basic properties of some Fourier coefficients for cusp forms for $\G$, which appear in the integral representation of the Spin $L$-function of \cite{Pollack-Shah}.

\subsubsection{The Siegel parabolic}\label{subsubFCsiegel}
We let $Q=L_3 U_3$ denote the standard Siegel parabolic subgroup of $\G$, with Levi $L_3 \simeq \GL_3 \times \GL_1$. Explicitly, 
\[L_3 = \left\{m(g,\mu) = \left(\begin{smallmatrix} g & \\ & \mu {^tg^{-1}} \end{smallmatrix} \right) \;|\; g \in \GL_3, \mu \in \GL_1  \right\}, \]\[U_3 = \left\{ n(u)=\left(\begin{smallmatrix} I_3 & u \\ & I_3 \end{smallmatrix} \right),\; \; u \in M_{3}\;|\; u^t=u \right\}. \] 
Denote ${\rm Sym}(3)= \{ \alpha \in M_{3}\;| \; \alpha^t=\alpha\}$. To each $\alpha \in {\rm Sym}(3)(\Q)$, we associate the unitary character
 $\psi_\alpha: U_3(\Q) \backslash U_3(\A) \to \C^\times$ by $  n(u) \in U_3(\A) \mapsto  e( {\rm Tr}(\alpha  u))$
where $e: \Q \backslash \A \to \C^\times$ is  the additive character with $e_\infty(x) := e^{2\pi i x}$ for $ x \in \R$, and conductor 1 at the finite places. 
For each $\alpha \in {\rm Sym}(3)(\Q)$, we define a Fourier coefficient along $U_3$ for a cuspidal automorphic representation $\pi$ of $\G(\A)$ as follows.

\begin{definition}\label{Fouriercoeffsiegelrk2}
Let $\Psi$ be a cusp form in the space of $\pi$. Define  \[ \Psi_{U_3, \psi_\alpha}(g):= \int_{U_3(\Q) \backslash U_3(\A)} \psi_\alpha^{-1}(u) \Psi(ug)du. \]
\end{definition}

We let $L_3(\Q)$ acts on ${\rm Sym}(3)(\Q)$ via the right action $ \alpha  \cdot m(g,\mu)  = \mu^{-1}g^t \alpha g $.

\begin{lemma}
Let $\alpha, \beta \in {\rm Sym}(3)(\Q)$. If there exists $m \in L_3(\Q)$ such that $\beta =   \alpha \cdot m,$ then 
\[ \Psi_{U_3,\psi_\beta}(g)=\Psi_{U_3, \psi_\alpha}(m g)  \]
\end{lemma}

\begin{proof}
Suppose   that $\beta =   \alpha \cdot m$ with $m=m(g,\mu)$. The result follows from the equality
\[ \psi_\beta (n(u)) = e( {\rm Tr}( \mu^{-1}g^t \alpha g  u)) = e( {\rm Tr}( \alpha g   u \mu^{-1}g^t )) = \psi_\alpha(m n(u) m^{-1}).   \]

\end{proof}

In this manuscript, we are interested in Fourier coefficients associated to the set of rank two elements  of ${\rm Sym}(3)(\Q)$, which we denote by ${\rm Sym}^{{\rm rk } 2}(3)(\Q)$. Let $D \in \Q^\times$ and let $F$ denote the \'etale quadratic extension $\Q(\sqrt{D})$ of $\Q$. If $D$ is not a square then $F$ is a field, else $F=\Q \times \Q$.

\begin{definition}
We let $\psi_D: U_3(\Q) \backslash U_3(\A) \to \C^\times$ be the unitary character \[\psi_{D} : n(u) \mapsto  e( {\rm Tr}(\alpha_D u)) = e(u_{33} -D u_{22})\] associated to $\alpha_D = \left( \begin{smallmatrix} 0  &    & \\  & -D  &  \\ & & 1 \end{smallmatrix} \right) \in  {\rm Sym}^{{\rm rk } 2}(3)(\Q)$.
\end{definition}

We have the following:
\begin{lemma}
A set of representatives of ${\rm Sym}^{{\rm rk } 2}(3)(\Q) /_{\sim M_3(\Q)}$ is given by  \[ \{ \alpha_D \;:\; D \in \Q^\times / (\Q^\times)^2 \}. \]
\end{lemma}

In view of these two lemmas, the set of Fourier coefficients associated to the Siegel parabolic and a rank 2 symmetric matrix is parametrized by the set of \'etale quadratic algebras of $\Q$.

\subsubsection{Fourier coefficients of type $(4 \, 2)$}\label{subsubsec:fc42}
We now turn our attention to Fourier coefficients associated to the unipotent orbit of $\G$ associated to the partition $(4\,2)$. The corresponding unipotent subgroup is the unipotent radical subgroup of the non-maximal standard parabolic $P=L_P\cdot U_P$, which arises as the intersection of the Siegel parabolic $Q$ with the Klingen parabolic. Notice that $P$ has Levi $L_P=\GL_2 \times \GL_1^2$, given by \[ \left\{ \left( \begin{smallmatrix} a &  & & \\  & g & &  \\  & & \mu a^{-1} & \\  & {} & & \mu {^tg^{-1}} \end{smallmatrix} \right) \; : \; a,\mu \in \GL_1, \; g \in \GL_2 \right \}. \] 
Following \cite[\S 2.1]{Pollack-Shah}, we define a unitary character which we still denote $\psi_D: U_P(\Q) \backslash U_P(\A) \to \C^\times$ as follows. Every element of $U_P / [U_P, U_P]$ can be expressed as the product of $n(v) \tilde{n}(u)$, where 
\[n(v) = \left( \begin{smallmatrix} 1 & v_1  & v_2   & & & \\   & 1  &     & & & \\   &   & 1  & & & \\   &   &     & 1 & & \\   &   &     & -v_1 &1  & \\   &   &     &-v_2 & & 1 \end{smallmatrix} \right) \in \G,\; \tilde{n}(u) = \left( \begin{smallmatrix} 1 &    &     & & & \\   & 1  &     & & u_{22} & u_{23}\\   &   & 1  & &u_{23} & u_{33}\\   &   &     & 1 & & \\   &   &     &   &1  & \\   &   &     &  & & 1 \end{smallmatrix} \right) \in U_3. \]
We will denote by $N_v$ (resp. $N_u$) the set of the $n(v)$'s (resp. $\tilde{n}(u)$'s).
If $n \equiv n(v)\tilde{n}(u)$ modulo $[U_P,U_P]$, define \[ \psi_D(n):= e(v_1+u_{33} -Du_{22})=e(v_1)\psi_D(n(u)). \]
Let $\pi$ be a cuspidal automorphic representation of $\G(\A)$. We define the following Fourier coefficients.

\begin{definition}\label{Fouriercoeff}
Let $\Psi$ be a cusp form in the space of $\pi$. Define  \[ \Psi_{U_P,\psi_D}(g):= \int_{U_P(\Q) \backslash U_P(\A)} \psi_D^{-1}(u) \Psi(ug)du. \]

\end{definition}

In the following proposition, we relate these Fourier coefficients to the ones for the Siegel parabolic associated to rank 2 symmetric matrices.

\begin{proposition} \label{PropFC}
For a cusp form $\Psi$ in the space of $\pi$, the following two conditions are equivalent.
\begin{enumerate}
    \item $\Psi_{U_P,\psi_D}(g) \not \equiv 0$.  
    \item There exists $\alpha \in {\rm Sym}^{{\rm rk } 2}(3)(\Q)$ with $\alpha \sim_{ L(\Q)} \alpha_D$ such that $\Psi_{U_3,\alpha}(g) \not \equiv 0$.
\end{enumerate}
\end{proposition} 

\begin{proof}
Fourier expand $ \Psi_{U_3,\psi_D}(g)$ over $N_v$ to get 
\[\Psi_{U_3,\psi_D}(g) = \int_{(\Q \backslash \A)^2}\Psi_{U_3,\psi_D}(n(v) g) dv + \sum_{\gamma \in {\rm Stab}_L(\psi_D)(\Q) \backslash L(\Q)} \Psi_{U_P,\psi_D}(\gamma  g).  \]
 The term \[ \int_{(\Q \backslash \A)^2} \Psi_{U_3,\psi_D}(n(v) g) dv = \int_{N_u(\Q) \backslash N_u(\A)} \psi_D^{-1}(\tilde{n}(u))\int_{U_K(\Q) \backslash U_K(\A)} \Psi(n_k \tilde{n}(u) g) d n_k d \tilde{n}(u) \]
and the inner integral vanishes because of cuspidality of $\Psi$ along the unipotent radical $U_K$ of the Klingen parabolic. Thus 
\[\Psi_{U_3,\psi_D}(g) =   \sum_{\gamma} \Psi_{U_P,\psi_D}(\gamma  g). \]
This relation implies the result as follows. If $\Psi_{U_3,\psi_D}(g) \not \equiv 0$, the Fourier coefficient $ \Psi_{U_P,\psi_D}(g)$ does not vanish identically. Viceversa, if $\Psi_{U_P,\psi_D}(g) \not \equiv 0$ then there is a character $\psi'$ in the $L(\Q)$-orbit of $\psi_D$ such that $\Psi_{U_3,\psi'}(g) \not \equiv 0$.
\end{proof}

\subsection{The Spin $L$-function and its residue at $s=1$}

Let $\pi$ denote any cuspidal automorphic representation of $\G(\A)$ with trivial central character. Let $S$ denote a finite set of places of $\Q$ containing the ones where $\pi$ is ramified and the archimedean place. If ${\rm Spin} : {\rm Spin}_7(\C) \to \GL(V_8)$ denotes the $8$-dimensional spin representation, the partial Spin $L$-function of $\pi$ is defined to be 
\[ L^S(s, \pi, {\rm Spin}):= \prod_{\ell \not\in S}\frac{1}{{\rm det}( 1 - \ell^{-s} {\rm Spin}(s_{\pi_\ell}))}, \] 
where $s_{\pi_\ell}$ denotes the Satake parameter of the unramified local component $\pi_\ell$. Let $\H$ be the group \eqref{thegroupH} associated to the \'etale cubic algebra $\Q \times F$, where $F=\Q(\sqrt{D})$ with either $D \not\equiv 1 \in \Q^\times_{>0} / (\Q^\times )^2$, in which case $F$ is a real quadratic field, or $D \equiv 1$ mod $(\Q^\times )^2$, in which case $F=\Q \times \Q$. For any cusp form $\Psi \in V_\pi$, Pollack-Shah \cite{Pollack-Shah} give an integral representation 
$$
\mathcal{I}(\Phi, \Psi, s)=\int_{\Z(\A)\H(\Q) \backslash \H(\A)}E_\Phi(h_1,s)\Psi(h)dh.
$$
of $L^S(s, \pi, {\rm Spin})$. For any $\Phi$ and $\Psi$, the integral $\mathcal{I}(\Phi, \Psi, s)$ is absolutely convergent for $Re(s)$ big enough and has a meromorphic continuation to $\C$. According to \cite[Proposition 7.1]{Gan-Gurevich}, for $Re(s)$ big enough we have the unfolding
$$
\mathcal{I}(\Phi, \Psi, s)=\int_{U_{B_{\H}}(\A)Z(\A)\backslash \H(\A)} f_\Phi(h_1, s) \Psi_{U_P, \psi_D}(h)dh
$$
where $U_{B_{\H}}$ is the unipotent radical of the upper triangular Borel subgroup $B_{\H}$ of $\H$ and $\Psi_{U_P,\psi_D}$ is the Fourier coefficient of Definition \ref{Fouriercoeff}.

\begin{theorem}[\cite{Pollack-Shah}] \label{Pollackshahongsp6}
For a set $\Sigma$ of places of $\Q$, denote  \[\mathcal{I}_\Sigma(\Phi, \Psi, s) = \int_{ U_{B_\H(\Q_\Sigma)} Z_\G(\Q_\Sigma) \backslash \H(\Q_\Sigma)} f(h_1, \Phi_\Sigma, s)\Psi_{U_P,\psi_D} (h) dh. \]
Let $\Psi$ be a cusp form in the space of $\pi$. Then, for any factorizable Schwartz-Bruhat function $\Phi$ on $\A^2$ and up to enlarging $S$, we have  
\[ \mathcal{I}(\Phi, \Psi, s) = \mathcal{I}_S(\Phi, \Psi, s) L^S(s, \pi, {\rm Spin}). \]
Moreover, there exists a cusp form $\tilde{\Psi}$ in the space of $\pi$ and a factorizable Schwartz-Bruhat function $\Phi$ on $\A^2$ such that \[ \mathcal{I}(\Phi, \tilde{\Psi}, s) = \mathcal{I}_\infty(\Phi, \Psi, s) L^S(s, \pi, {\rm Spin}).\]
\end{theorem}

Note that if $\pi$ does not support a rank two Fourier coefficient (for the Siegel parabolic $Q$) and thus, by Proposition \ref{PropFC}, a Fourier coefficient for $P$, the integral $\mathcal{I}(\Phi, \Psi, s)$ is identically zero.

\begin{corollary}[\cite{Pollack-Shah}]\label{cor:PS}
Suppose that $\pi$ supports a rank two Fourier coefficient. Then, the partial Spin $L$-function $L^S(s, \pi, {\rm Spin})$ has meromorphic continuation in $s$, is holomorphic outside $s=1$, and has at worst a simple pole at $s=1$.
\end{corollary}

As we explain in the later sections, using results of Gan and Gan-Gurevich, they further prove that when $L^S(s, \pi, {\rm Spin})$ has a simple pole at $s=1$, $\pi$ lifts to the split $G_2$ under the exceptional theta correspondence. This observation is based on the following key relation between the residue at $s=1$ of $L^S(s, \pi , {\rm Spin} )$ and the automorphic period we have introduced in \S \ref{ss:thepairing}.

\begin{prop}\label{periodvsresidue}
For any factorizable Schwartz-Bruhat function $\Phi$ on $\A^2$, we have
    \[ 
   \frac{\widehat{\Phi}(0)}{2}\cdot \int_{\Z(\A)\H(\Q) \backslash \H(\A)}\Psi(h) dh = \mathrm{Res}_{s=1} \left( \mathcal{I}_S(\Phi, \Psi, s) L^S(s,\pi , {\rm Spin}) \right).
    \]
\end{prop}

\begin{proof}
Thanks to Lemma \ref{eisenstein}, the residue at $s=1$ of $\mathcal{I}(\Phi, \Psi, s)$ equals
\[ \frac{\widehat{\Phi}(0)}{2}\cdot \int_{\Z(\A)\H(\Q) \backslash \H(\A)}\Psi(h) dh.\]
The result then follows from Theorem \ref{Pollackshahongsp6}.
\end{proof}

We now state our first main result. Let $\pi$ denote a cuspidal automorphic representation of 
 $\PGSp_6(\A)$ for which $\pi_\infty$ is the discrete series of Hodge type (3,3) in the $L$-packet of $V^\lambda$ with $\lambda=(\lambda_2+\lambda_3, \lambda_2, \lambda_3,0)$. Let $\mathcal{Z}_{\H, B}^{[\lambda, \mu]}, \mathcal{Z}_{\H, \mathcal{H}}^{[\lambda,\mu]} ,$ and $\omega_\Psi$ be as in \S \ref{SectRealizations} and \S \ref{ss:thepairing}. Let $\Psi^{[\lambda,\mu]}$ denote $A^{[\lambda,\mu]}\cdot\Psi$, where $A^{[\lambda,\mu]} \in U(\mathfrak{k}_\C)$ is the operator defined in Proposition \ref{period}.

\begin{theorem}\label{theoremcyclebetti1}
We have 
\begin{align*}
\langle  \mathcal{Z}_{\H, \mathcal{H}}^{[\lambda,\mu]} , [\omega_\Psi] \rangle_\mathcal{H} = \langle  \mathcal{Z}_{\H, B}^{[\lambda,\mu]} , [\omega_\Psi] \rangle = C \cdot \mathrm{Res}_{s=1} \left( \mathcal{I}_S(\Phi, \Psi^{[\lambda,\mu]}, s) L^S(s,\pi , {\rm Spin}) \right),
\end{align*}
where
$$
C=\frac{ \widehat{\Phi}(0) h_{{U'}}}{2(2\pi i)^3\cdot \mathrm{vol}({U'})}. 
$$
\end{theorem}

\begin{proof}
By Proposition \ref{period} and Proposition \ref{period2}, we have that 
\[ \langle  \mathcal{Z}_{\H, \mathcal{H}}^{[\lambda,\mu]} , [\omega_\Psi] \rangle_\mathcal{H} = \langle  \mathcal{Z}_{\H, B}^{[\lambda,\mu]} , [\omega_\Psi] \rangle = \frac{h_{{U'}}}{(2\pi i)^3\cdot \mathrm{vol}({U'})}
 \int_{\H(\Q) \Z_\G(\A) \backslash \H(\A)} \Psi^{[\lambda,\mu]}(h)dh,\] 
where $U' = U \cap \H(\Af)$ and $h_{U'} =4^{-1} | \Z_\G(\Q) \backslash \Z_\G(\Af) / (\Z_\G(\Af) \cap {U'})|$. By Proposition \ref{periodvsresidue}, we have
    \begin{align*}
        \langle  \mathcal{Z}_{\H, B}^{[\lambda,\mu]} , [\omega_\Psi] \rangle       &= C \cdot
       \mathrm{Res}_{s=1}\left( \mathcal{I}_S(\Phi, \Psi^{[\lambda,\mu]}, s) L^S(s,\pi , {\rm Spin})\right)
    \end{align*} 
where $$C = \frac{ \widehat{\Phi}(0) h_{{U'}}}{2(2\pi i)^3\cdot \mathrm{vol}({U'})}.$$ This finishes the proof.
\end{proof}

Let us fix a Schwartz-Bruhat function $\Phi$ such that $\widehat{\Phi}(0) \neq 0$.

\begin{corollary}\label{CorollaryNonVanishingCycle}
Suppose that $\pi$ satisfies the following hypotheses: \begin{itemize}
    \item $\mathcal{I}_S(\Phi, \Psi^{[\lambda,\mu]} , 1) \ne 0$ for some $\mu $,
    \item the partial $L$-function $L^S(s,\pi , {\rm Spin})$ has a pole at $s=1$.
\end{itemize} 
Then   
\begin{align*}
        \langle  \mathcal{Z}_{\H, B}^{[\lambda,\mu]} , [\omega_\Psi] \rangle = \langle  \mathcal{Z}_{\H, \mathcal{H}}^{[\lambda,\mu]} , [\omega_\Psi] \rangle_\mathcal{H} \ne 0.
\end{align*}
\end{corollary}

\begin{proof}
By \cite[Theorem 1.3]{Pollack-Shah} the function $L^S(s, \pi, \mathrm{Spin})=1$ has at most a simple pole. As a consequence $\mathrm{Res}_{s=1}  L^S(s,\pi , {\rm Spin}) \neq 0$. By Theorem \ref{theoremcyclebetti1}, under the assumption that  $\mathcal{I}_S(\Phi, \Psi^{[\lambda,\mu]} , 1) \ne 0$, this implies that $\langle  \mathcal{Z}_{\H, B}^{[\lambda,\mu]} , [\omega_\Psi] \rangle \neq 0$.
\end{proof}

\begin{remark}
 If the automorphic representation $\pi$ supports a rank two Fourier coefficient and its partial Spin $L$-function has a (necessarily simple) pole at $s=1$, by the results of \cite{Pollack-Shah} it is $\H$-distinguished, namely the map $\mathcal{P}_\H \in {\rm Hom}_{\H(\A)} ( \pi, \mathbf{1})$ defined by
\[ \Psi \mapsto \mathcal{P}_\H(\Psi) := \int_{Z(\A)\H(\Q) \backslash \H(\A)} \Psi(h) d h   \]
is not identically zero. Then, asking $\mathcal{I}_S(\Phi, \Psi^{[\lambda,\mu]} , 1) \ne 0$ for some $\mu$ is equivalent to asking that the map obtained as the composition of $\mathcal{P}_\H$ with an $\H(\R)$-equivariant embedding $\pi_\infty \to \pi$ restricts non-trivially to the minimal $K_\infty$-type of $\pi_\infty$. 
\end{remark}

Let $H^6_{\mathcal{M}}(\Sh_\G(U), \mathscr{V}^\lambda_F(3))_{hom}$ denote the $F$-vector space defined in section \ref{section:ahreal} and let $H^6_{\mathcal{M}}(\Sh_\G(U), \mathscr{V}^\lambda_F(3))_{hom}[\pi_f^\vee]$ denote its $\pi_f^\vee$-isotypical component. This is a finite dimensional $L$-vector space, where $L$ is the number field introduced in $\S$ \ref{section:Hodgetheory}. Tate conjecture for the motive attached to $\pi$ (see Conjecture \ref{ConjectureBT} (3)) predicts the equality
$$
-\mathrm{ord}_{s=1} L(s, \pi, \mathrm{Spin})=\dim_L H^6_{\mathcal{M}}(\Sh_\G(U), \mathscr{V}^\lambda_F(3))_{hom}[\pi_f^\vee].
$$ 

\begin{cor} \label{cor:tateconjecture}
If $\mathcal{I}_S(\Phi, \Psi^{[\lambda,\mu]} , 1) \ne 0$, then we have
$$
-\mathrm{ord}_{s=1} L^S(s, \pi, \mathrm{Spin}) \leq \dim_L H^6_{\mathcal{M}}(\Sh_\G(U), \mathscr{V}^\lambda_F(3))_{hom}[\pi_f^\vee].
$$
\end{cor}

\begin{proof}
If $L^S(s, \pi, \mathrm{Spin})$ does not have a pole at $s=1$, there is nothing to prove. If not, Corollary \ref{CorollaryNonVanishingCycle} implies that the projection of $\mathcal{Z}_{\H, B}^{[\lambda,\mu]}$ to the $\pi_f^\vee$-isotypical component is non-zero, showing the result.
\end{proof}

The following result verifies a weaker form of Conjecture 1.1(3) for the motive $M(\pi_f^\vee)(3)$ at the cost of supposing that $\mathcal{I}_S(\Phi, \Psi^{[\lambda,\mu]} , 1)$ is non-zero for some $\mu $.

\begin{corollary} \label{theoremcyclebetti2}
Suppose that $\pi$ satisfies the hypotheses of Corollary \ref{CorollaryNonVanishingCycle} and that $\textbf{(St)}$ holds. Then ${\rm pr}_{\pi^\vee} \mathcal{Z}_{\H, \mathcal{H}}^{[\lambda,\mu]}$ and its Hecke translates generate $H^1_\mathcal{H}(M(\pi_f^\vee)_{\R}(4))$.
\end{corollary} 

\begin{proof}
If $\pi_p$ is the Steinberg representation, it follows from the second statement of Lemma \ref{exloc} and its proof, that $H^1_\mathcal{H}(M(\pi_f^\vee)_\R(4))$ is a rank one module over the full Hecke algebra of level $U$. Hence the result follows by Corollary \ref{CorollaryNonVanishingCycle}.
\end{proof}

\section{Exceptional theta lifts from \texorpdfstring{$G_2$}{G2} to \texorpdfstring{$\PGSp_6$}{PGSp6}} \label{Section6}

In this section, we discuss the exceptional theta correspondence for the dual reductive pair $(G_2,\PGSp_6)$ and describe the set of Fourier coefficients associated to the Heisenberg parabolic for cuspidal automorphic forms of $G_2(\A)$. It has as its solely purpose to fix notations and to recall some well known results that will be used later so that the knowledgeable reader might skip it.

\subsection{Split \texorpdfstring{$G_2$}{G2} and \texorpdfstring{$E_7$}{E7}} In this section we will follow the exposition of the Appendix of \cite{harris-khare-thorne} by Savin.

\subsubsection{The group $G_2$}\label{subsub:groupg2}

Let $\mathbb{H}$ be the algebra of Hamilton quaternions over $\Q$ with the usual basis $\{1,i,j,k\}$. The conjugate $\bar{a}$ of an element $a=\alpha_0+\alpha_1i+\alpha_2j+\alpha_3k \in \mathbb{H}$ is $\bar{a}=\alpha_0-\alpha_1i-\alpha_2j-\alpha_3k$. The split octonion algebra over $\Q$ is $\mathbb{O}=\mathbb{H} \oplus \mathbb{H}$ with multiplication
$$(a,b) \cdot (c,d)=(ac+d\bar{b}, \bar{a}d+cb).$$
Then $\mathbb{O}$ is a non-commutative, non-associative $\Q$-algebra. However it is alternative, which means that for any $x,y \in \mathbb{O}$ we have $x \cdot (x \cdot y)=(x \cdot x) \cdot y$ and $(x \cdot y) \cdot y=x \cdot (y \cdot y)$ (see \cite{jacobson}). If $x=(a,b)$, let $\overline{x}=(\overline{a},-b)$. Then $x \mapsto \overline{x}$ is a $\Q$-linear involution on $\mathbb{O}$ satisfying $\overline{x \cdot y}= \overline{y} \cdot \overline{x}$. The norm $\mathrm{N}: \mathbb{O} \rightarrow \Q$ is the quadratic form defined by $x \mapsto x\cdot \overline{x}=\overline{x} \cdot x$.
The trace $\mathrm{Tr}: \mathbb{O} \rightarrow \Q$ is defined by $x \mapsto x+\overline{x}$.
For any $x, y, z \in \mathbb{O}$, the properties 
\begin{align*}
\mathrm{N}(x\cdot y) &= \mathrm{N}(x) \mathrm{N}(y),\\
\mathrm{Tr}(x\cdot y)&=\mathrm{Tr}(y\cdot x),\\
\mathrm{Tr}(x\cdot(y\cdot z))&=\mathrm{Tr}((x \cdot y) \cdot z)
\end{align*}
are satisfied. For $x,y \in \mathbb{O}$, we write $y \in x^\perp$ if $y$ is orthogonal to $x$ with respect to the bilinear form $(x,y) \mapsto {\rm Tr}(x \cdot \overline{y})$, which means that $x\cdot \overline{y} +  y \cdot \overline{x} = 0$.\\

Let $l=(0,1) \in \mathbb{O}$ so that $\{1,i,j,k, l,li,lj,lk \}$ is a basis of $\mathbb{O}$. From this, one constructs another useful basis $\{s_1,s_2,s_3,s_4,t_1,t_2,t_3,t_4\}$, where
\[ s_1=\frac{1}{2}(i+li), s_2=\frac{1}{2}(j+lj),
s_3=\frac{1}{2}(k+lk), s_4=\frac{1}{2}(1+l), \] \[ t_1=\frac{1}{2}(i-li), t_2=\frac{1}{2}(j-lj), t_3=\frac{1}{2}(k-lk), t_4=\frac{1}{2}(1-l). \]
The following multiplication table for this basis is given in Table 1 of the Appendix of \cite{harris-khare-thorne}.
\begin{table}[ht]
\begin{tabular}{|c ||c|c|c||c|c|c||c|c|}
\hline
 & $s_1$ & $s_2$ & $s_3$ & $t_1$ & $t_2$ & $t_3$ & $s_4$ & $t_4$ \\ \hline \hline
 $s_1$ & $0$ & $-t_3$ & $t_2$ & $s_4$ & $0$ & $0$ & $0$ & $s_1$    \\ \hline
 $s_2$ & $t_3$ & $0$ & $-t_1$ & $0$ & $s_4$ & $0$ & $0$ & $s_2$    \\ \hline
 $s_3$ & $-t_2$ & $t_1$ & $0$ & $0$ & $0$ & $s_4$ & $0$ & $s_3$    \\ \hline \hline
 $t_1$ & $t_4$ & $0$ & $0$ & $0$ & $s_3$ & $-s_2$ & $t_1$ & $0$    \\ \hline
 $t_2$ & $0$ & $t_4$ & $0$ & $-s_3$ & $0$ & $s_1$ & $t_2$ & $0$    \\ \hline
 $t_3$ & $0$ & $0$ & $t_4$ & $s_2$ & $-s_1$ & $0$ & $t_3$ & $0$    \\ \hline \hline
 $s_4$ & $s_1$ & $s_2$ & $s_3$ & $0$ & $0$ & $0$ & $s_4$ & $0$    \\ \hline
 $t_4$ & $0$ & $0$ & $0$ & $t_1$ & $t_2$ & $t_3$ & $0$ & $t_4$    \\
\hline
\end{tabular}

\end{table}

We define
$$
G_2:=\{g \in \mathrm{GL}(\mathbb{O}) \,|\, g(x \cdot y)=(gx) \cdot (gy), \forall x,y \in \mathbb{O}\}.
$$
to be the group of automorphisms of $\mathbb{O}$. We note that $G_2$ acts transitively on non-zero elements of trace zero and norm zero. We will denote the set of trace zero octonions by either $\mathbb{O}^0$ or $V_7$, where the latter notation emphasises that this set defines the standard irreducible 7-dimensional representation of $G_2$ and induces an embedding 
\[ G_2 \hookrightarrow {\rm SO}_7.\]

\subsubsection{The dual reductive pair} \label{Section:dualpair}

We consider the Albert algebra $J$ over $\Q$, which is the set of matrices
$$
A=\begin{pmatrix}
d & \overline{z} & y\\
z & e & \overline{x}\\
\overline{y} & x & f
\end{pmatrix}
$$
where $d, e, f \in \Q$ and $x, y, z \in \mathbb{O}$. The algebra $J$ is equipped with a cubic form, called the determinant, which is given by
\[ \det(A) = def - d N(x) - e N(y) - f N(z) + \mathrm{Tr}(zyx).\]
The group of isogenies of this form is a group of type $E_6$ and its orbits on $J$ are classified by the rank. We will need to consider the set $\Omega$ of rank $1$ elements $A \in J$, i.e. those $A \neq 0$ such that $A^2=\mathrm{Tr}(A) \cdot A$. This condition means that the entries of $A$ satisfy the equalities
\begin{equation} \label{Equationrank1}
\begin{array}{ll}
{\mathrm{N}(x)=ef, \;\; \mathrm{N}(y)=df, \;\; \mathrm{N}(z)=de,} \\
{dx=\overline{y}\cdot \overline{z}, \;\; ey=\overline{z}\cdot \overline{x}, \;\; fz=\overline{x}\cdot \overline{y}.}
\end{array}
\end{equation}
Let $G$ denote the split adjoint group of type $E_7$, which is constructed from $J$ by the Koecher-Tits construction (see Section 3 of \cite{kobayashi-savin}). The group $G$ has a maximal parabolic $P=MN$ and its opposite $\bar{P}=M \bar{N}$, with $N \simeq J$ and such that the action under conjugation of the Levi $M$ on $N$ gives an isomorphism of $M$ and the group of similitudes of the cubic form on $J$
\[ M \cong \{ g \in \GL(J) \, | \, \det(gA) = \lambda \det(A) \text{ for some } \lambda \in \G_m \text{ and } \forall A \in J \}. \] The group $G_2$ can be realized as a subgroup of $M$ via its action on $J$ by the rule 
$$
g \cdot \begin{pmatrix}
d & \overline{z} & y\\
z & e & \overline{x}\\
\overline{y} & x & f
\end{pmatrix}=\begin{pmatrix}
d & g\overline{z} & gy\\
gz & e & g\overline{x}\\
g\overline{y} & gx & f
\end{pmatrix}.
$$
This action has fixed points $J_3$, the Jordan algebra of symmetric $3 \times 3$ matrices with entries in $\Q$. Note that the left action of $ \GL_3$ on $J_3 \cong N$ given by
\begin{equation} \label{EqactionM}
g \cdot A = \det(g)^{-1} g A g^t
\end{equation}
extends to an action on $J$ preserving the determinant form up to scalar, thus defining an embedding of $ \GL_3$ into $M$. Then $\GL_3$ is the centraliser of $G_2$ in $M$ and $Q= \GL_3 U_3$ (which is the Siegel parabolic of $ \mathrm{PGSp}_6$) is the centralizer of $G_2$ in $P$. Similarly, the opposite $\overline{Q}$ is the centralizer of $G_2$ in $\overline{P}$. This gives the dual reductive pair $(G_2, \mathrm{PGSp}_6)$ in $G$.

\subsection{Fourier coefficients for $G_2$} 

\subsubsection{Root system and parabolic subgroups}\label{subsub:RootsandHeisenberg}
Let $T$ be a (rank 2) maximal split torus over $\Q$ in $G_2$ and let $\Delta$, resp. $\Delta^+ \subset \Delta$, be the set of roots, resp. a subset of positive roots, for $G_2$. Let $a$, resp. $b$, denote the long, resp. short, simple root in $\Delta^+$. Then 
\[\Delta^+ = \{a, b, a+b, a+2b, a+3b,2a+3b \}. \]
We let $B = T U$ denote the Borel subgroup of $G_2$ associated to $\Delta^+$. Other than $B$, there are two proper standard parabolic subgroups $P_a$ and $P_b$ of $G_2$, such that $P_a \cap P_b = B$. They are characterized by the following. For any $\alpha \in \Delta^+$, denote by $x_\alpha: \mathbf{G}_a \hookrightarrow U$
the one parameter unipotent subgroup associated to $\alpha$. Then, for each $r \in \{ a , b \}$, the Levi $L_r$ of $P_r$ is isomorphic to $\GL_2$ and contains $x_r$. We fix an isomorphism $\GL_2 \simeq L_r$ such that  $ {\matrix{1}{u}{}{1}} \mapsto x_r(u)$.

Let $U_a$ be the unipotent radical of $P_a$. It is a 3-step nilpotent group of dimension 5 with filtration
\[U_a \supset U_1 \supset U_2 \supset \{1\},\]
where $U_a / U_1$ is generated by $\{x_{b}, x_{a+b}\}$, $U_1/U_2$ is isomorphic to the one parameter unipotent subgroup $x_{a+2b}$, and $U_2$ is generated by $\{x_{a+3b}, x_{2a+3b} \}$. As representations of $L_a$, $U_a / U_1$ is the standard representation, while $U_1/U_2$ is the determinant (cf. \cite[\S 2.4]{Gan-SavinHoweduality}). 

We denote by $H := P_b$ the so-called Heisenberg parabolic and let $L_H U_H$ denote  its Levi decomposition. The unipotent radical $U_H$ is of dimension 5 and admits the filtration \[ U_H \supset [U_H, U_H] \supset \{1\}, \]
with  $ U_H/[U_H, U_H]$ being the four dimensional abelian unipotent group generated by 
\[\{x_a, x_{a+b}, x_{a+2b}, x_{a+3b} \}, \]
while $[U_H, U_H]$ is isomorphic to the one parameter unipotent subgroup $x_{2a+3b}$. 

\subsubsection{An embedding of \texorpdfstring{$\SL_3$}{SL3} into \texorpdfstring{$G_2$}{G2}}\label{subsub:SL3intoG2}
The group $G_2$ acts transitively on the set
\[ \Gamma_c := \{ x \in \mathbb{O}^0\,|\, N(x) = - c \}. \]
By \cite[Theorem 4]{jacobson}, the stabilizer of an element $y_0 \in \Gamma_1$ is isomorphic to $\SL_3$. Choose $y_0$ such that the unipotent radical $U_{\SL_3}$ of the upper triangular Borel of $\SL_3$ is generated by the one-parameter subgroups 
\[ \{ x_a, x_{a+3b}, x_{2a+3b}\}.\]
In terms of the basis chosen in \S \ref{subsub:groupg2}, this is achieved by choosing $y_0 = s_4 - t_4$. In this case, one shows (cf. \cite[Lemma 2]{rallis-schiffmann}) that the stabilizer of $y_0$ leaves invariant the subspace $\langle s_1, s_2, s_3 \rangle$ and is identified with $\SL_3 = \SL(\langle s_1, s_2, s_3\rangle)$.

\subsubsection{The Lie algebra of $G_2$} \label{subsectLieG2}
The multiplication map on $\mathbb{O}$ induces a map $V_7 \otimes V_7 \to V_7$ given by $x \otimes y \mapsto \frac{xy - yx}{2}$. This map is alternating, hence it induces a $G_2$-equivariant map $\wedge^2 V_7 \to V_7$ which is surjective.
Then the Lie algebra $\mathfrak{g}_2$ of $G_2$ can be identified with the kernel of this map. Under this identification, one has an explicit description of the action of $\mathfrak{g}_2$ on $V_7$, namely
\[ (w \wedge x) \cdot v = \langle x, v \rangle w - \langle w, v \rangle x. \] We will also need (cf. \cite[\S 22.2]{FultonHarris}) the decomposition
\begin{align}\label{decompositionLiealgebraG2}
    \mathfrak{g}_2 = \mathfrak{sl}_3 \oplus {\rm Std}_3 \oplus {\rm Std}_3^*, 
\end{align} 
where ${\rm Std}_3$ is the standard representation of $\SL_3$ with basis $\{v_1,v_2,v_3\}$ and ${\rm Std}_3^*$ is its dual with basis $\{\delta_1,\delta_2,\delta_3\}$ and where we denote by $E_{ij}$, $1 \leq i < j \leq 3$ the standard basis vectors of $\mathfrak{sl}_3$. The identification between the two descriptions (cf. \cite[\S 2.2]{Pollack-G2}) of $\mathfrak{g}_2$ is given by $E_{ij} = t_j \wedge s_i$, $1 \leq i < j \leq 3$, $v_i = (s_4 - t_4) \wedge s_i + t_{i+1} \wedge t_{i+2}$ and $\delta_i = (s_4 - t_4) \wedge t_i + s_{i+1} \wedge s_{i+2}$, $1 \leq i \leq 3$, where indices are taken modulo $3$. Moreover, the component $\mathfrak{sl}_3$ is the Lie algebra of the copy of $\SL_3$ embedded into $G_2$ as above. In particular, $E_{12}, E_{13}$ and $E_{23}$ are root vectors for the roots $a, 2a + 3b$ and $a + 3b$ respectively. Moreover, the vectors $v_1, v_2$ and $\delta_3$ are root vectors for the roots $a+b, b$ and $a+2b$, respectively. Via \eqref{decompositionLiealgebraG2}, the Lie algebra $\mathfrak{u}_H$ of $U_H$ is \begin{align}\label{decompositionuh}
    \mathfrak{u}_H = \mathfrak{u}_{\SL_3} \oplus \Q v_1 \oplus \Q \delta_3,
\end{align}
Under \eqref{decompositionLiealgebraG2} the Lie algebra $\mathfrak{l}_H$ of the Levi $L_H$ is generated by the Cartan subalgebra and the root vectors $v_2$, $\delta_2$.

\subsubsection{Fourier coefficients}\label{subsub:FcforG2} 

We now describe the Fourier coefficients for $G_2$ associated to the Heisenberg parabolic. We closely follow \cite{Pollack-G2} and refer to it for more details. In order to describe the Fourier coefficients associated to $H$, we need to study the $L_H$-representation $V_H :=U_H/[U_H, U_H]$. As a $\GL_2$-representation, $V_H$ is isomorphic to  ${\rm Sym}^3({\rm Std}_2) \otimes {\rm det}^{-1}({\rm Std}_2)$, where ${\rm Std}_2$ denotes the standard representation of $\GL_2$.  Under the identification of \eqref{decompositionuh}, (a representative of) an element of $V_H(\Q)$ can be written as 
\[  x_a(\lambda_1) x_{a+b}(\lambda_2/3) x_{a+2b}(\lambda_3/3)x_{a+3b}(\lambda_4),\; \text{ with } \lambda_i \in \Q, \]
which corresponds to the binary cubic polynomial \[p(x,y)= \lambda_1 x^3 + \lambda_2 x^2y + \lambda_3 xy^2 + \lambda_4 y^3\] 
where $x,y$ form a basis of ${\rm Std}_2$. Associated to $p$, there is the cubic $\Q$-algebra $R$ with basis $\{1,i,j\}$ with multiplicative table \begin{align*}
ij &= - a d \\ 
i^2 &= -a c +b i - a j \\ 
j^2 &= -bd +d i - c j.
\end{align*}
\begin{example}\label{exampleofcubicpols} \leavevmode
\begin{enumerate} \item (\cite[3.2]{Gross-Lucianovic})
If $p(x,y) = x^2y - xy^2$ then the associated $\Q$-algebra $R$ is isomorphic to $\Q^3$.
\item (\cite[3.3]{Gross-Lucianovic}) If $p(x, y) = x^3 - D xy^2$ (or equivalently $p(x,y) = -D x^2 y + y^3$ using the action of ${\matrix 0  1 1 0 }$)  then the associated $\Q$-algebra $R$ is isomorphic to $\Q \oplus \Q(\sqrt{D})$.
\end{enumerate}
\end{example}
There is an action of $\GL_2(\Q)$ on the set of bases $\{1,i,j \}$ of a given cubic algebra $R$, which makes the association $p(X,Y) \mapsto (R, \{1,i,j\})$ $\GL_2(\Q)$-equivariant. Since any cubic algebra admits a basis of this shape, we have the following.

\begin{proposition}[{\cite[Proposition 2.1]{Gross-Lucianovic}}] There is a bijection between the $\GL_2(\Q)$-orbits on $V_H(\Q)$ and the set of isomorphism classes of cubic $\Q$-algebras. Moreover, each orbit has a well-defined discriminant in $\Q^\times / (\Q^\times )^2$.
\end{proposition}

Let $e: \Q \backslash \A \to \C^\times$ be the additive character introduced in \S \ref{subsubFCsiegel}. Let $\langle , \rangle $ denote the symplectic pairing on $V_H$ defined as follows. If $v,v' \in V_H$ correspond to $p(x,y)$ and $p'(x,y)$ respectively, then \[ \langle v , v' \rangle = \lambda_1 \lambda_4' - \tfrac{1}{3} \lambda_2 \lambda_3' + \tfrac{1}{3} \lambda_3 \lambda_2' - \lambda_4 \lambda_1'.\] 
Any character $\psi: U_H(\Q) \backslash U_H(\A) \to \C^\times$  factors through $V_H(\A)$, hence we consider the projection $\bar{n}$ of $n \in U_H(\A)$ to $V_H(\A)$, which, by \eqref{decompositionuh}, can be written as $$\overline{n} = x_a(\lambda_1')x_{a+b}(\lambda_2'/3) x_{a+2b}(\lambda_3'/3)x_{a+3b}(\lambda_4').$$
If $v \in V_H(\Q)$ corresponds to $p(x,y)$, we then define $\psi_v : U_H(\Q) \backslash U_H(\A) \to \C^\times$ by \[ n \mapsto e( \langle v , \overline{n} \rangle) = e( \lambda_1 \lambda_4' - \tfrac{1}{3} \lambda_2 \lambda_3' + \tfrac{1}{3} \lambda_3 \lambda_2' - \lambda_4 \lambda_1'). \] 
The character $\psi_v$ is non-degenerate if and only if $v$ corresponds to an \'etale cubic algebra over $\Q$. In this manuscript, we are interested in \'etale cubic algebras of the form $\Q \times F$, with $F$ of either the form $\Q(\sqrt{D})$ (with  $ \Q^\times / (\Q^\times)^2 \ni D \not\equiv 1$) or $\Q \times \Q$ (with $D \equiv 1$ mod $(\Q^\times)^2$).

\begin{definition} Let $\psi_{H,D}: U_H(\Q) \backslash U_H(\A) \to \C^\times$ denote the character associated to $\Q \times F$. Given a cusp form $\varphi$ for $G_2(\A)$, define 
\[ \varphi_{U_H,\psi_{H,D}} (g) := \int_{U_H(\Q) \backslash U_H(\A)} \psi^{-1}_{H,D}(n) \varphi(n g) d n.\]

\end{definition}

\subsection{The theta lift from $G_2$ to $\mathrm{PGSp}_6$}\label{sub:thetaliftg2pgsp6}

Let $\Pi=\bigotimes'_v \Pi_v$ denote the restricted tensor product of the minimal representations $\Pi_v$ of $E_7(\Q_v)$ over all places $v$ of $\Q$. A unitary model of the minimal representation is given by $L^2(\Omega)$, where recall that $\Omega$ denotes the subset of rank 1 elements in $J$. There is a unique up to a non-zero scalar embedding
$$
\theta: \Pi \rightarrow \mathcal{A}(E_7(\Q) \backslash E_7(\A))
$$
of $\Pi$ in the space $\mathcal{A}(E_7(\Q) \backslash E_7(\A))$ of automorphic forms of $E_7$ (see \cite{ginzburg-rallis-soudry1}, \cite{kobayashi-savin}). For $f \in \Pi$ and $\varphi \in \mathcal{A}(G_2(\Q) \backslash G_2(\A))$, we define a function $\Theta(f,\varphi)$ on $\mathrm{PGSp}_6(\A)$ by
$$
\Theta(f,\varphi)(g)=\int_{G_2(\Q) \backslash G_2(\A)} \theta(f)(g'g)\varphi(g')dg'.
$$

\begin{definition} Let $\sigma$ be a cuspidal automorphic representation of $G_2(\A)$. \begin{enumerate}
    \item Define $\Theta(\sigma)$ to be the span of the functions $\Theta(f,\varphi)$, where $f \in \Pi$ and $\varphi$ runs through the cusp forms in the contragradient $\sigma^\vee$ of $\sigma$. 
    \item We say that a cuspidal automorphic representation $\pi$ of $\mathrm{PGSp}_6(\A)$ is a $\Theta$-lift of $\sigma$ if it appears as an irreducible subquotient of $\Theta(\sigma)$.
\end{enumerate} 
\end{definition}

If a $\Theta$-lift of $\sigma$ exists, then its local constituents are compatible with the local  Theta correspondence between $G_2$ and $\mathrm{PGSp}_6$.

\begin{proposition} \label{proplocalglobal}\
 Let $\pi$ be a $\Theta$-lift of $\sigma$, then $\pi_v$ is an irreducible subquotient of  $\Theta(\sigma_v)$.
\end{proposition}

\begin{proof}
See \cite[Theorem 1.7 (i)]{harris-khare-thorne}.
\end{proof}

After imposing certain local conditions on $\sigma$, in the next section we use one of the main results of \cite{Ginzburg-Rallis-Soudry2} to show that $\Theta(\sigma)$ is non-zero and cuspidal, thus proving the existence of a non-trivial $\Theta$-lift of $\sigma$. Before doing so, we first recall the properties of the local theta correspondence needed later.

\subsubsection{Discrete series and a conjecture of Gross}

Let $T_c$ denote a compact torus in $G_2(\R)$, which is contained in the maximal compact subgroup $K_{G_2} \simeq (\mathrm{SU}_2 \times \mathrm{SU}_2)/ \mu_2$ of $G_2(\R)$. We abuse notation denoting again by $a,b$ the simple positive roots for $T_c$ (with the short root $b$ which we assume to be compact) and $\Delta^+$ the resulting set of positive roots. Then, $\rho = \tfrac{1}{2} \sum_{\alpha \in \Delta^+} \alpha = 3 a + 5b$. The set of positive compact roots is given by \[ \Delta_c^+ = \{ b , 2 a + 3 b \},  \]
which, in the notation of \cite{li}, is $\{ 2 \epsilon_2,2 \epsilon_1\}$. The Weyl group $\mathfrak{W}_{G_2}$ is isomorphic to the dihedral group $D_6$ of 12 elements and it is generated by $w_a$ and $w_b$, where $w_\alpha$ denotes the reflections around the line orthogonal to $\alpha$. The Weyl group $\mathfrak{W}_{K_{G_2}} \simeq (\Z / 2\Z)^2$ is generated by $w_b$ and $w_{2a+3b} = w_a w_b w_a w_b w_a$.\\

Let $\gamma$ be a dominant weight for $G_2$ with respect to $T_c$. The set of equivalence classes of irreducible discrete series of $G_2(\R)$ associated to $\gamma$ has cardinality equal to $|\mathfrak{W}_{G_2} / \mathfrak{W}_{K_{G_2}}| = 3$. Choose representatives $\{w_1, w_2, w_3 \}$ of $ \mathfrak{W}_{G_2} / \mathfrak{W}_{K_{G_2}} $ such that
$w_i\rho$ is dominant for $K_{G_2}$. Then, for any $1 \leq i \leq 3$, there exists an irreducible discrete series $\sigma_\infty^\Gamma$ of Harish-Chandra parameter $\Gamma = w_i (\gamma + \rho)$ and minimal $K_{G_2} $-type $\Gamma + \delta_{G_2} - 2 \delta_{K_{G_2}}$, where $\delta_{G_2}$, resp. $\delta_{K_{G_2}}$, is the half-sum of roots, resp. compact roots, which are positive with respect to the Weyl chamber in which $\Gamma$ lies. Precisely, if we let $w_1 = {\rm id}$, $w_2= w_a$, and $w_3 = w_b w_a$, then 
\begin{align*}
    w_1 \rho &= \rho = 3 \epsilon_1 + \epsilon_2, \\ 
    w_2 \rho &= 2 a + 5 b = 2 \epsilon_1 + 4 \epsilon_2, \\
    w_3 \rho &= a + 4b =  \epsilon_1 + 5 \epsilon_2.
\end{align*} 
We let $\mathcal{D}_{3,1}$, $\mathcal{D}_{2,4}$, and $\mathcal{D}_{1,5}$ denote the sets of discrete series of $G_2(\R)$ whose Harish-Chandra parameter lies in the Weyl chamber corresponding to $w_1\rho$, $w_2\rho$, and $w_3\rho$ respectively. Elements of $\mathcal{D}_{3,1}$ are the quaternionic discrete series, while elements of $\mathcal{D}_{2,4}$ are the generic discrete series.\\
 
Gross has given a precise conjectural description of the entire discrete spectrum of the dual pair $(G_2, \PGSp_6)$ (cf. \cite[Conjecture 1.2]{li}). Recall that there are four families of discrete series for $\PGSp_6(\R)$, indexed by the set of Hodge types up to conjugation. In particular, the discrete series of $\PGSp_6(\R)$ of Hodge type $(4,2)$, resp. $(6,0)$, are the generic, resp. holomorphic discrete series. 

\begin{conjecture}[Gross]
Let $\Pi_\infty$ be the minimal representation of $E_7(\R)$. The discrete spectrum of the restriction of $\Pi_\infty$ to the dual pair $G_2(\R) \times \PGSp_6(\R)$ is the direct sum of all tensor products $\sigma_\infty \otimes \theta(\sigma_\infty)$, where $\sigma_\infty$ belongs to the discrete series of $G_2$. If $\sigma_\infty$ has infinitesimal character $\gamma + \rho = r \epsilon_1+ s \epsilon_2$ and belongs to either $\mathcal{D}_{3,1}$, $\mathcal{D}_{2,4}$, or $\mathcal{D}_{1,5}$, then $\theta(\sigma_\infty)$ is the discrete series of $\PGSp_6(\R)$ with infinitesimal character $(r, \tfrac{1}{2}(r+s),\tfrac{1}{2}(r-s))$ and Hodge type $(3,3)$, $(4,2)$, or $(5,1)$ respectively.
\end{conjecture}

Partial results towards the conjecture of Gross were showed by Li for discrete series in $\mathcal{D}_{3,1}$ (cf. \S \ref{subsubLiresult} below) and for the generic family $\mathcal{D}_{2,4}$ by Harris-Khare-Thorne in \cite[Theorem 1.5 and Theorem 1.7(ii)]{harris-khare-thorne} using the main result of Savin's appendix to \cite{harris-khare-thorne} and the non-vanishing of the global theta lift given by \cite[Corollary 4.2]{Ginzburg-Rallis-Soudry2}. Li also gave evidence to the predictions of Gross for a proper subset $\mathcal{D}_{1,5}'$ of $\mathcal{D}_{1,5} $ in \cite[Theorem 4.3]{li}. We also note that the remaining equivalence class of holomorphic discrete series of $\PGSp_6(\R)$ (of Hodge type $(6,0)$) is realized in an exceptional theta correspondence studied by Gross-Savin between the compact real form $G_2^c(\R)$ of $G_2$ and $\PGSp_6(\R)$ and moreover this is the only Hodge type that appears in that correspondence (cf. \cite[Theorem 3.5]{GrossSavin}).

\subsubsection{Quaternionic discrete series and their theta lift}\label{subsubLiresult}

We describe the main result of \cite{li}. We first notice that a discrete series $\sigma_\infty^{x,y}$ of Harish-Chandra parameter $x \epsilon_1 + y \epsilon_2$ lies in the set of quaternionic discrete series $\mathcal{D}_{3,1}$ if $x,y$ are two non-negative integers such that $x-3 \geq y-1 \geq 0$ and $x-y$ is even. The minimal $K_{G_2}$-type of $\sigma_\infty^{x,y}  \in \mathcal{D}_{3,1}$ is given by \[ {\rm Sym}^{x+1}({\rm Std}_{\epsilon_1}) \boxtimes {\rm Sym}^{y-1}({\rm Std}_{\epsilon_2}),\] 
 where ${\rm Std}_{\epsilon_1}$ (resp. ${\rm Std}_{\epsilon_2}$) is the standard representation of the $\mathrm{SU}_2$ corresponding to the long root $\epsilon_1$ (resp. the short root $\epsilon_2$).

\begin{proposition}\label{propLi}
Let $\Pi_\infty$ denote the minimal representation of $E_7(\R)$. We have \[{\Pi_\infty}|_{{G_2(\R) \times \mathrm{PGSp}_6(\R)}} \supseteq \bigoplus_{\sigma_\infty^{x,y} \in \mathcal{D}_{3,1}} \sigma_\infty^{x,y}  \otimes \theta(\sigma_\infty^{x,y})\] where $\theta(\sigma_\infty^{x,y}) \in P(V^{\lambda}),$ with $\lambda =(x-3,\tfrac{1}{2}(x+y)-2,\tfrac{1}{2}(x-y)-1,0)$, is the discrete series $ \pi_\infty^{3,3}$ of Hodge type $(3,3)$ and Harish-Chandra parameter $(\tfrac{1}{2}(x+y),  \tfrac{1}{2}(x-y),-x)$. 
\end{proposition}

\begin{proof}
See \cite[Theorem 1.1]{li} and \cite[Theorem 5.4]{huang-pandzic-savin}.
\end{proof}

The set $\mathcal{D}_{3,1}$ contains an important family of discrete series, which were studied by Gross and Wallach in \cite{Gross-Wallach1} and \cite{Gross-Wallach2}.

\begin{definition}\label{def:quatdiscseries}
For every $n \geq 2$, the quaternionic discrete series $\sigma_n$ is the element of $\mathcal{D}_{3,1}$ of Harish-Chandra parameter $(2n -1) \epsilon_1 + \epsilon_2$ and minimal $K_{G_2}$-type \[ {\rm Sym}^{2n}({\rm Std}_{\epsilon_1}) \boxtimes \mathbf{1}.\] 
\end{definition}

A fundamental property of the members of this family is that they admit (unique) models with respect to the unipotent radical of the Heisenberg parabolic and non-degenerate characters corresponding to totally real \'etale cubic algebras. Recall, as in \S \ref{subsub:FcforG2}, that a non-degenerate character $\psi: U_H(\R) \to \C^\times$ corresponds to a cubic algebra, whose discriminant is either positive or negative. The first type corresponds to the $\GL_2(\R)$-orbit on $V_H(\R)$ given by $\R^3$, while the second to the $\GL_2(\R)$-orbit of $\R \times \C$. A representative $\psi: U_H(\R) \to \C^\times$ of the totally real orbit is given by $e^{2 \pi i f}$, where $f: U_H(\R) \to \R$ is non-zero on the one parameter unipotent subgroups $x_{a+b}$ and $x_{a+2b}$  and trivial on $x_{a}$ and $x_{a+3b}$ (cf. \cite[\S 6]{Gan-Gross-Savin}). A special case of the main result of \cite{Wallach} gives the following.

\begin{proposition}\label{Wallachprop}
Let $\psi$ be a non-degenerate character of $V_H(\R)$. There is (at most) a one dimensional space of $\psi$-equivariant linear functionals on $\sigma_n$. Moreover, \[ {\rm dim }\; {\rm Hom}_{U_H(\R)} (\sigma_n , \psi) = 1, \]
exactly when $\psi$ corresponds to a totally real cubic algebra.
\end{proposition}

\subsubsection{The non-archimedean theta correspondence}

We describe the properties of the non-archimedean theta correspondence which will be later needed to study the global theta correspondence. Let $\sigma$ be an irreducible admissible representation of $G_2(\Q_p)$. The maximal $\sigma$-isotypic quotient of the minimal representation $\Pi_p$ of $E_7(\qp)$ can be expressed as $\sigma \boxtimes \Theta(\sigma)$, with $\Theta(\sigma)$ a smooth representation of $\mathrm{PGSp}_6(\Q_p)$ which is called the big theta lift of $\sigma$. 

\begin{proposition}\label{propositiononlocalthetaprop}
For an irreducible admissible representation $\sigma$ of $G_2(\Q_p)$, $\Theta(\sigma)$ has finite length with unique irreducible quotient (if non-zero) $\theta(\sigma)$. Moreover, one has the following.
\begin{enumerate}
    \item Let $\sigma$ be an unramified generic representation of $G_2(\Q_p)$ with Satake parameter $s$, then  $\pi= \theta(\sigma)$ is the unramified representation of $\mathrm{PGSp}_6(\Q_p)$ whose Satake parameter is $\varphi \circ s$, where $\varphi: G_2 \hookrightarrow {\mathrm{Spin}_7}$ is the map of $L$-groups.
    \item Let ${\rm St}_{G_2}$ (resp. ${\rm St}_{\mathrm{PGSp}_6}$) be the Steinberg representation of $G_2(\Q_p)$ (resp. ${\mathrm{PGSp}_6}(\Q_p)$); then $\theta({\rm St}_{G_2})={\rm St}_{\mathrm{PGSp}_6}$.
\end{enumerate}
\end{proposition}

\begin{proof}
See \cite[Theorem 1.2, Theorem 15.3(v)]{Gan-SavinHoweduality} and \cite[Proposition 3.1]{GrossSavin}.
\end{proof}

\section{Cuspidality and Fourier coefficients of the global theta lift}\label{sectiononcuspidalityandfouriercoefficients}

In this section, based on the works  \cite{Ginzburg-Rallis-Soudry2}, \cite{GrossSavin}, and the appendix of Savin in \cite{harris-khare-thorne}, we give a criterion on the cuspidality of representations in the image of the exceptional theta lift and on their possession of Fourier coefficients of type $(4 \, 2)$.

\subsection{Cuspidality of the global lift}\label{cuspidalityandperiodvanishing}

Let $V$ denote the unipotent subgroup of $\SL_3$ (embedded into $G_2$ as in \S \ref{subsub:SL3intoG2}) generated by the roots $a+3b$ and $2a+3b$. We further consider the subgroup $\SL_2$ embedded into $G_2$ via the Levi of the ``long root'' parabolic $P_a$ and denote, for any cusp form $\varphi$ for $G_2(\A)$, \[ \varphi^{\SL_2 V}(g) := \int_{\SL_2(\Q) \backslash \SL_2(\A)} \int_{V(\Q) \backslash V(\A)} \varphi(v m g) dv dm.  \]
We will now show that the above period vanishes whenever $\varphi$ is not globally generic. We are thankful to David Ginzburg for kindly sharing with us a proof of this fact. 

\begin{lemma}\label{lem:periodsl2n}
Let $\sigma$ be a cuspidal automorphic representation of $G_2(\A)$, which is not globally generic. For any cusp form $\varphi \in V_\sigma$ and $g \in G_2(\A)$, we have $\varphi^{\SL_2 V}(g) = 0$.
\end{lemma}

\begin{proof}

Let $Z$ denote the unipotent subgroup of $G_2$ generated by the roots $a+2b, a+3b,$ and $2a+3b$. Let $\varphi \in V_\sigma$. If we Fourier expand the period  $\varphi^{\SL_2 V}(g)$ along the 1-dimensional unipotent subgroup $x_{a+2b}(r)$ of $G_2$, we get \[\varphi^{\SL_2 V}(g) = \varphi^{\SL_2 Z}(g) + \sum_{\psi} \varphi^{\SL_2 Z,\psi}(g),  \]
where the sum runs over non-trivial additive characters $\psi : Z(\Q) \backslash Z(\A) \to \C^\times$ supported on the root $a+2b$,  $\varphi^{\SL_2 Z}(g)$ is the period of $\varphi$ over $[\SL_2 Z]$, and
\[\varphi^{\SL_2Z,\psi}(g) : = \int_{\SL_2(\Q) \backslash \SL_2(\A)} \int_{Z(\Q) \backslash Z(\A)} \varphi( u m g) \psi(u) du dm. \]
By \cite[Lemma 2.1]{Ginzburg-Rallis-Soudry2}, $\varphi^{\SL_2 Z}(g) = 0$ for all $\varphi$ in $\sigma$ and $g \in G_2(\A)$. Hence,  $\varphi^{\SL_2 V}(g) = 0$ if and only if $\varphi^{\SL_2Z,\psi}(g) = 0$ for all non-trivial $\psi$. We now argue by contradiction. Suppose that  $\varphi^{\SL_2Z,\psi}(g) \ne 0$ for a certain $\psi$.  We claim that this implies that $\sigma$ supports Whittaker Fourier coefficients, thus contradicting our hypothesis.

Let $U_a$ be the unipotent radical of $P_a$ introduced in \S \ref{subsub:RootsandHeisenberg}. Since $V$ is normal in $U_a$, we can consider the quotient $V_0 = U_a / V$, which is isomorphic to the Heisenberg group in three variables and it is generated by the roots $b, a+b,$ and $a+2b$. The center of $V_0$ is generated by the root $a+2b$ and is identified with the quotient $Z_0:=Z/V$. As $\SL_2$ is embedded into $G_2$ via the Levi $L_a$ of $P_a$, it acts trivially on the quotient $Z_0$ . Therefore, $D := \SL_2 V_0$ is a Jacobi group in the sense of \cite[Definition in p.619]{Ikeda-JacobiFC}. Let $\widetilde{D(\A}):=\widetilde{\SL_2(\A}) V_0(\A)$, with $\widetilde{\SL_2(\A})$ denoting the metaplectic cover of $\SL_2(\A)$, and denote by $C^\infty_{\psi}(D(\Q) \backslash \widetilde{D(\A}) )$ the space of functions $f$ on $D(\Q) \backslash \widetilde{D(\A})$ such that $f(zvh) = \psi(z) f(vh)$ for any $z \in Z_0(\A)$, $v \in V_0(\A)$, $h \in \widetilde{\SL_2(\A})$. For any Schwartz function $\Phi \in \mathcal{S}(\A)$, we let $\theta^{\Phi}_{\SL_2} \in C_{\psi}^\infty(D(\Q) \backslash \widetilde{D(\A)})$ be the theta function defined in \cite[p. 620]{Ikeda-JacobiFC}.
By \cite[Proposition 1.3]{Ikeda-JacobiFC}, if $W$ is a closed subspace of $C^\infty_{\psi^{-1}}(D(\Q) \backslash \widetilde{D(\A}) )$ which is invariant under right translation of $V_0(\A)$, the functions of the form
\begin{equation} \label{EquationIkeda}
vh \mapsto \overline{\theta_{\SL_2}^{\Phi_1}}(vh) \int_{V_0(\Q) \backslash V_0(\A)} f(u h) \theta_{\SL_2}^{\Phi_2}(uh) du,
\end{equation}
with $v \in V_0(\A)$, $h \in \widetilde{\SL_2(\A})$, $f \in W$, $\Phi_1,\Phi_2 \in \mathcal{S}(\A)$, generate a dense subspace of $W$. We apply this to the space $W$ given by the closure of the subspace generated by the right $V_0(\A)$-translations of
\[ \varphi^{Z, \psi}(g) := \int_{Z(\Q) \backslash Z(\A)} \varphi( u  g) \psi(u) d u. \]
Assume that $\varphi^{\SL_2 Z, \psi}$ is not identically zero. By considering right translates of $\varphi$ we can assume that $\varphi^{\SL_2 Z ,\psi}(1)$ is non-zero. This implies that the integral
\begin{align*} 
   I_1(\varphi, \Phi_1, \Phi_2):= &\int_{\SL_2(\Q) \backslash \SL_2(\A)} \overline{\theta_{\SL_2}^{\Phi_1}}(m) \int_{V_0(\Q) \backslash V_0(\A)} \varphi^{Z,\psi}(u m) \theta_{\SL_2}^{\Phi_2}(um) d u dm
\end{align*}
is non-zero for some choice of data $(\Phi_1, \Phi_2)$. Note that the integral $I_1(\varphi, \Phi_1, \Phi_2)$ is well defined because as the functions in Equation \eqref{EquationIkeda} are not genuine for our space $W$. Since $\theta_{\SL_2}^{\Phi_2}(z g) = \psi(z) \theta_{\SL_2}^{\Phi_2}(g)$
for all $z \in Z_0(\A)$, we can write $I_1(\varphi, \Phi_1, \Phi_2)$
as  
\begin{align*} 
\int_{\SL_2(\Q) \backslash \SL_2(\A)} &\int_{(\Q \backslash \A)^5} \varphi( x_b(v_1)x_{a+b}(v_2)x_{a+2b}(r_1)x_{a+3b}(r_2)x_{2a+3b}(r_3) m ) \cdot  \\ &\cdot \theta_{\SL_2}^{\Phi_2}(x_b(v_1)x_{a+b}(v_2)x_{a+2b}(r_1) m) \overline{\theta_{\SL_2}^{\Phi_1}}(m) dv_i d r_i dm.
\end{align*}
The integral $I_1(\varphi, \Phi_1, \Phi_2)$ is the residue of the global zeta integral which calculates the standard $L$-function for $\varphi$ when $\varphi$ admits Whittaker coefficients. Namely, by the Siegel-Weil formula $\overline{\theta_{\SL_2}^{\Phi_1}}(m)$ is the residue at $s=3/4$ of an Eisenstein series ${\rm Eis}_{\widetilde{\SL}_2}(m, s)$ (depending on $\Phi_1$) on the metaplectic cover of $\SL_2$ normalized as in \cite[\S 2]{GinzburgStdG2}. By \cite[Theorem 4]{GinzburgStdG2} $I_1(\varphi, \Phi_1, \Phi_2)$ is the residue of 
\begin{align*} 
   I_2(\varphi, \Phi_1, \Phi_2, s) :=  \int_{\SL_2(\Q) \backslash \SL_2(\A)} &\int_{(\Q \backslash \A)^5} \varphi( x_b(v_1)x_{a+b}(v_2)x_{a+2b}(r_1)x_{a+3b}(r_2)x_{2a+3b}(r_3) m ) \cdot  \\ &\cdot \overline{\theta_{\SL_2}^{\Phi_2}}(x_b(v_1)x_{a+b}(v_2)x_{a+2b}(r_1) m) {\rm Eis}_{\widetilde{\SL}_2}(m,s) dv_i d r_i dm. 
\end{align*}
We can now prove our claim. Suppose that  $\varphi^{\SL_2 Z,\psi}(1) \ne 0$, then $I_1(\varphi, \Phi_1, \Phi_2)$ is not zero for some choice $(\Phi_1, \Phi_2)$. This implies that, for ${\rm Re}(s)$ large enough, the integral $I_2(\varphi, \Phi_1, \Phi_2, s)$ is not zero. By \cite[Theorem 1]{GinzburgStdG2}, $I_2(\varphi, \Phi_1, \Phi_2, s)$ unfolds to the Whittaker model and thus contains a Whittaker coefficient of $\varphi$ as an inner integration. This shows that if $\varphi^{\SL_2 Z,\psi}(1) \ne 0$ for some choice of data, the Whittaker coefficient for $\varphi$ is non-trivial and thus $\sigma$ is globally generic. This finishes the proof.
\end{proof}

\begin{theorem} \label{cuspidality}
Let $\sigma$ be a cuspidal automorphic representation of $G_2(\A)$. Assume that \begin{enumerate}
    \item $\sigma$ is not globally generic;
    \item there exists a finite place $p$ such that $\sigma_p$ is generic.
\end{enumerate}  Then the big theta lift $\Theta(\sigma)$ of $\sigma$ to $\PGSp_6$ is cuspidal.
\end{theorem}

\begin{proof}
We show the result by using the tower of theta lifts from $G_2$ and its properties studied in \cite{Ginzburg-Rallis-Soudry2}. If $\sigma$ lifts trivially to $\mathrm{PGSp}_6$ then there is nothing to prove, so suppose that $\sigma$ has a non-zero theta lift $\pi$ to $\mathrm{PGSp}_6$. Then, by  \cite[Theorem A]{Ginzburg-Rallis-Soudry2} $\pi$ is cuspidal if and only if the lifts of $\sigma$ to $\mathrm{PGSp}_4$ and $\mathrm{PGL}_3$ are both zero. By \cite[Theorem 4.1 (3)]{Ginzburg-Rallis-Soudry2}, the lift to $\mathrm{PGSp}_4$ is zero if and only if $$
\varphi^{\mathrm{SL}_3}(g) =  \int_{[\mathrm{SL}_3]} \varphi(xg) dx = 0 \text{ and } \varphi^{\mathrm{SU}(2,1)}(g) = \int_{[\mathrm{SU}(2,1)]} \varphi(xg) dx=0
$$
for any $g \in G_2(\A)$, any $\varphi \in V_{\sigma^\vee}$. Here, $\mathrm{SL}_3$ embeds into $G_2$ as the stabilizer of a norm $-1$ vector (cf. \S \ref{subsub:SL3intoG2}), while $\mathrm{SU}(2,1)$ is realized as the stabilizer of a norm $-c$ vector, with $c$ not a square in $\Q$. We argue by contradiction. Suppose that $ \sigma^\vee$ has a non-trivial $\mathrm{SU}(2,1)$-functional. This implies that, at every finite $v$, $\sigma_v$ admits one. By Frobenius reciprocity, 
\[
\Hom_{\mathrm{SU}(2,1)}(\sigma^\vee_v, \C) = \Hom_{G_2}( {\text{c-Ind}}_{\mathrm{SU}(2,1)}^{G_2}(\C), \sigma_v)
\]
and hence, since $\sigma_v$ is irreducible, one deduces that each local component $\sigma_v$ of $\sigma$ is a quotient of $C^\infty_c(G_2(\Q_v) / \mathrm{SU}(2,1)(\Q_v))$. In particular,   $\sigma_p$ is identified with such a quotient. This is a contradiction as, by hypothesis, $\sigma_p$ is generic but, by \cite[Lemma 4.10]{GrossSavin}, $C^\infty_c(G_2(\qp)/\mathrm{SU}(2,1)(\qp))$ does not admit a Whittaker functional. The same argument also shows the vanishing of $\varphi^{\mathrm{SL}_3}$. We claim finally that the theta lift of $\sigma$ to $\mathrm{PGL}_3$ also vanishes.  
Since $\sigma$ is  not globally generic, Lemma \ref{lem:periodsl2n} shows that, for all $\varphi \in \sigma$,  $\varphi^{\SL_2 V}(g)= 0$. We can then apply \cite[Theorem 4.1(4)]{Ginzburg-Rallis-Soudry2} to deduce that the theta lift of $\sigma$ to $\mathrm{PGL}_3$ is zero and conclude the proof.

\end{proof}

\begin{corollary} \label{corocusptheta}
Let $\sigma$ be a cuspidal automorphic representation of $G_2(\A)$. Assume that \begin{enumerate}
    \item $\sigma_\infty$ is a discrete series;
    \item there exists a finite place $p$ such that $\sigma_p$ is Steinberg.
\end{enumerate} Then $\Theta(\sigma)$ is cuspidal.
\end{corollary}

\begin{proof}
We distinguish two cases. We first suppose that $\sigma$ is globally generic. Then  we apply \cite[Theorem 1.7(ii)]{harris-khare-thorne} to deduce that its theta lift is cuspidal. If, instead, $\sigma$ is not globally generic, the result follows from Theorem \ref{cuspidality} as the Steinberg representation $\sigma_p = {\rm St}_{G_2}$ is generic.
\end{proof}

\subsection{Calculation of orbits} \label{subsetorbits}

This preparatory section presents an elementary but crucial calculation needed in the proof of Proposition \ref{PropFCperiod}.\\

Let $e: \Q \backslash \A \rightarrow \C^\times$ be the standard non-trivial character introduced in \S \ref{Fouriercoefficientsgsp6} and let $A \in J(\Q)$. We define the character $\psi_A: N(\Q) \backslash N(\A) \rightarrow \C^\times$ by $\psi_A(X) = e(\mathrm{Tr}(A \circ X))$ where $A \circ X=\frac{1}{2}(AX+XA)$ is the Jordan product. Recall from \S \ref{Fouriercoefficientsgsp6} that, for any $B \in J_3(\Q)$,  we define a character $\psi_B: U_3(\Q) \backslash U_3(\A) \rightarrow \C^\times$ by $\psi_B(n(X)) = e(\mathrm{Tr}(B X))$. In particular, we have denoted $\psi_D$ the character associated to
\[ \alpha_D =\begin{pmatrix} 0
 &  & \\
 & -D & \\
 &  & 1
\end{pmatrix} \in J_3(\Q). \]  Define 
\[ \omega(\Q):=\{A \in \Omega(\Q) \,|\, \psi_A|_{U_3(\A)}=\psi_D \}, \]
i.e. the set of rank 1 matrices in $J(\Q)$ inducing the same character as $\alpha_D$ on the unipotent radical of the Siegel parabolic. In the following, we will always see $\omega(\Q)$ inside $\overline{N}(\Q)$. In particular, if $g \in \GL_3(\Q) \subseteq M(\Q)$, its action on $A$ is the dual action to \eqref{EqactionM}, namely $g \cdot A = \det(g) (g^t)^{-1} A g^{-1}$. Finally, denote by $A(x,y,z)$ the matrix ${\begin{pmatrix}
0 & \overline{z} & y\\
z & -D & \overline{x}\\
\overline{y} & x & 1
\end{pmatrix}} \in J$.

\begin{lemma}\label{lemmabeforeorbits}
We have
\[\omega(\Q)=\left\{ A(x,y,z) \, : \,  \mathrm{Tr}(x)=\mathrm{Tr}(z)=0, \mathrm{N}(x)=-D, N(z)=0,  z \in x^\perp, y=-D^{-1}zx \right \}.\]
\end{lemma}

\begin{proof} Let 
$$
A= \begin{pmatrix}
d & \overline{z} & y\\
z & e & \overline{x}\\
\overline{y} & x & f
\end{pmatrix} \in J
$$
Similarly to the proof of \cite[Lemma 3.4]{GrossSavin}, the condition $\psi_A|_{U_3(\A)}=\psi_D$  is equivalent to
$$
d=0, \;\; e=-D, \;\; f=1,
$$
$$
\overline{x}=-x, \;\; \overline{y}=-y, \;\; \overline{z}=-z
$$
This together with the condition that $A$ has rank $1$ (Equation \eqref{Equationrank1}) give
$$
\mathrm{N}(x)=-D, \;\; \mathrm{N}(y)=\mathrm{N}(z)=0, $$
$$
yz=0, \;\; zx=-Dy, \;\; xy=z.  
$$
We claim that these conditions imply that $z \in x^\perp$, which means that $z\overline{x}+x\overline{z}=0$, or equivalently $zx=-xz$. Indeed, multiplying $z = xy$ on the left by $x$ and using alternativity, we obtain
$$
xz = x(xy) = (xx)y = Dy = -zx.
$$
Finally, as $\mathrm{N}(x)=-D$ and $\mathrm{Tr}(x)=0$, we have $x^2=D$ and hence $x^{-1}=D^{-1}x$, which implies that
$$y = x^{-1} z = D^{-1} xz.$$
This shows one inclusion of the statement.

In the other direction let $x, z \in \mathbb{O}$ be as in the right hand side of the statement. We have to show that $y := -D^{-1}zx$ has norm and trace equal to zero and that $xy=z$. We have $$
\mathrm{N}(y)=(-D)^{-2} \mathrm{N}(z) \mathrm{N}(x)=0
$$
and
\[
\mathrm{Tr}(y)=-D^{-1}\mathrm{Tr}(zx) =  D^{-1} \mathrm{Tr}(z \overline{x}) = 0
\]
as $z \in x^\perp$.
Hence $\mathrm{Tr}(y)=0$. Moreover
$$
xy=-D^{-1}x(zx)=D^{-1}x(xz)=D^{-1}(xx)z=z. $$
This shows that $A \in \omega(\Q)$ and concludes the proof of the lemma.
\end{proof}

As for any $A(x,y,z) \in \omega(\Q)$, the octonion $y=-D^{-1}zx$ is determined by $x$ and $z$ and we will often denote $A(x,y,z)$ by $A(x,z)$. Note that there is an action of $G_2(\Q)$ on the set $\omega(\Q)$ given by the action on the coefficients. The following proposition describing the orbits of this action will be essential.

\begin{proposition} \leavevmode \label{orbits}
The group $G_2(\Q)$ acts on $\omega(\Q)$ with a finite number of orbits. Moreover, representatives of the orbits and their respective stabilizers are given as follows.
\begin{enumerate}
    \item If $D$ is a square in $\Q^\times$:
    \begin{enumerate}
        \item $A_3 = A(x,0)$, where $x = (s_4-t_4) \sqrt{D}$ and $\mathrm{Stab}_{G_2(\Q)}(A(x, 0)) \cong \SL_3$, where $\SL_3$ is embedded   into $G_2$ as \S \ref{subsub:SL3intoG2}.
        \item  $A_2 = A(x, t_3)$ with $\mathrm{Stab}_{G_2}(A_2) = \SL_2 V \subset \SL_3$, where $\SL_2$ and $V$ embed into $\SL_3$ as in \S \ref{cuspidalityandperiodvanishing}. 
        \item $A_1 = A(x, s_3)$ with $\mathrm{Stab}_{G_2}(A_1) = \SL_2 \overline{V} \subset \SL_3$, where $\SL_2$ is as in (1)(b) and $\overline{V}$ is the opposite unipotent subgroup to $V$. 
        \item $A_0 = A(x, s_1+t_3)$ with $\mathrm{Stab}_{G_2}(A_0) = U_D,$ where  $U_D$ denotes the unipotent radical of the upper-triangular Borel of $\SL_3$ (denoted by $U_{\SL_3}$ in \S \ref{cuspidalityandperiodvanishing}).
    \end{enumerate}
    \item If $D$ is not square in $\Q^\times$:
    \begin{enumerate}
        \item $A_1 = A(x,0) \in \omega(\Q)$, for any $x \neq 0$ for which $N(x)=-D$, with \[\mathrm{Stab}_{G_2(\Q)}(A(x,0)) \cong \SU_3^{D},\]
        where $\SU_3^{D} = \SU(x^\perp)$ is the unitary group for the restriction of the norm form to the $3$-dimensional $\Q(\sqrt{D})$-subspace of $\mathbb{O}^0$ orthogonal to $x$ \footnote{
        The unitary group $\SU_3^{D}$ is a form of $\SL_3$ which splits over $\Q(\sqrt{D})$ and which is isomorphic to $\SU(2,1)$, resp. $\SL_3$, if $D<0$, resp. $D>0$, over $\R$.}
        \item $A_0 = A(x, z)$, for any norm zero $z$ in $x^\perp$, with $\mathrm{Stab}_{G_2}(A_0) \simeq U_{D}$, where $U_D$ denotes the unipotent radical of the upper-triangular Borel of $\SU_3^{D}$.
    \end{enumerate}
\end{enumerate}
\end{proposition}

\begin{proof}
\textit{Step 1.} By \cite[Theorem 1]{rallis-schiffmann}, the group $G_2$ acts transitively on the set of trace zero elements of norm $-D$ and hence on the sets $A(x, 0)$.
The description of the stabilizer in (1)(a) follows from \cite[Theorem 4]{jacobson} or \cite[Lemma 2]{rallis-schiffmann}. The description of the stabilizer in (2)(a) follows from \cite[Theorem 3]{jacobson} or \cite[Lemma 3]{rallis-schiffmann}. More precisely, according to \cite[Lemma 3]{rallis-schiffmann} the subspace $x^\perp$ of $\mathbb{O}^0$ of elements which are orthogonal to $x$ has the structure of a $3$-dimensional $\Q(\sqrt{D})$-vector space and the action of $\mathrm{Stab}_{G_2}(x)$ on $x^\perp$ induces an isomorphism $\mathrm{Stab}_{G_2}(x) \simeq \SU_3^{D}$.

\textit{Step 2.} We now study the remaining $G_2$-orbits when $D$ is a square in $\Q$. Again, we can assume that $D = 1$. Recall from \S \ref{subsub:SL3intoG2} that $\SL_3$ embeds into $G_2$ as the stabilizer of $s_4 - t_4$. This identification is explicitly given as follows (cf. \cite[Lemma 2]{rallis-schiffmann}). An element of $g \in \SL_3$ induces an action on $\mathbb{O}^0$ fixing $s_4 - t_4$ and given by the left multiplication by $g$ on $\langle s_1, s_2, s_3 \rangle$ and by $(g^t)^{-1}$ on $\langle t_1, t_2, t_3 \rangle$. One verifies that this actions respects multiplication and hence defines an element in $G_2$.
Assume $z \neq 0$ is such that $A(x,z) \in \omega(\Q)$. Since $z$ is trace zero and orthogonal to $x = s_4 - t_4$ we can write $z = z_1 + z_2$ with $z_1 = \sum_i \alpha_i s_i$ and $z_2 = \sum_i \beta_i t_i$. Since the group $\SL_3$ acts transitively on the non-zero elements of $\langle s_1, s_2, s_3 \rangle$ and $\langle t_1, t_2, t_3 \rangle$, then the cases where $z_1 = 0$ or $z_2 = 0$ give rise to exactly two orbits. When $z_1 = 0$, taking $z_2 = t_3$ as a generator of this orbit, the corresponding stabilizer is \[  \left\{ \begin{pmatrix} * & * & * \\ * & * & * \\ 0 & 0 & 1 \end{pmatrix} \right\}\subset \SL_3, \] which coincides with $\SL_2 V$ as in (1)(b). Similarly, when $z_2 = 0$, taking $z_1 = s_3$ as the generator of the orbit, then the stabilizer is \[\left\{ \begin{pmatrix} * & * & 0 \\ * & * & 0 \\ * & * & 1 \end{pmatrix} \right\}\subset \SL_3.\]
This is nothing but $\SL_2 \overline{V}$, with $\SL_2$ which again embeds in the Levi of the long root parabolic $P_a$ and $\overline{V}$ is the opposite unipotent subgroup to $V$ generated by the negative roots $-a-3b$, $-2a-3b$. Finally we treat the case $z_1, z_2 \neq 0$. Write
\begin{equation*}
z = \alpha_1 s_1 + \alpha_2 s_2 + \alpha_3 s_3 + \beta_1 t_1 + \beta_2 t_2 + \beta_3 t_3.
\end{equation*}
The condition $N(z) = 0$ translates then in
\begin{equation} \label{eqortho} \alpha_1 \beta_1 + \alpha_2 \beta_2 + \alpha_3 \beta_3 = 0.
\end{equation} We can assume that $z_2 = t_3$. Then $\alpha_3 = 0$ by \eqref{eqortho} and, using the action of the stabilizer of $t_3$, we can assume that $z_1 = s_1$. It is then immediate to check that the stabilizer of $A(s_4-t_4,s_1+t_3)$ is as in (1)(d). This concludes the proof of (1).

\textit{Step 3.} We finally deal with the case where $D$ is not a square in $\Q$. By Witt's theorem, the group $\SU_3^{D}$ acts transitively on the isotropic vectors of the three dimensional $\Q(\sqrt{D})$ vector space $x^\perp$. We thus have two orbits for $G_2(\Q)$ on $\omega(\Q)$, generated by $A(x,0)$ and $A(x,z)$, where $z$ is any non-zero vector in $x^\perp$ with zero norm. We are now left with calculating the stabilizer of the latter orbit.  
The action of $\SU_3^{D}$ on $x^\perp$ is given by its natural action on $\Q(\sqrt{D})^3$. More precisely, after extending scalars to $\Q(\sqrt{D})$, we can decompose \[ x^\perp \otimes_\Q \Q(\sqrt{D}) = \Q(\sqrt{D}) \langle s_1,s_2,s_3 \rangle \oplus \Q(\sqrt{D}) \langle t_1,t_2,t_3 \rangle.   \] 
The projection to the first component induces an isomorphism of $\Q(\sqrt{D})$-vector spaces $x^\perp \simeq \Q(\sqrt{D}) \langle s_1,s_2,s_3 \rangle$ (cf. \cite[Lemma 3]{rallis-schiffmann}), with $\SU_3^{D}$ acting naturally on the basis $\{s_1,s_2,s_3\}$. Here, we choose the Hermitian form (with respect to the extension $\Q(\sqrt{D})/\Q$) defining $\SU_3^{D}$ given by \[\left( \begin{smallmatrix}   &  &  \sqrt{D}^{-1} \\  & 1 &   \\ - \sqrt{D}^{-1} &  &  \end{smallmatrix} \right) \in \GL_3(\Q(\sqrt{D})).\] We can then suppose that $z$ is sent to $s_1$ and the corresponding stabilizer is given by \[ \left \{ \begin{pmatrix} 1 & * & * \\ 0 & * & * \\ 0 & * & * \end{pmatrix} \right \} \cap \SU_3^{D} = U_D.\]
\end{proof}

\subsection{Non-vanishing of Fourier coefficients I}

Recall that we have denoted by $\Pi=\bigotimes'_v \Pi_v$ the minimal representation of the group $E_7$. Moreover, in \S \ref{sub:thetaliftg2pgsp6}, for $f \in \Pi$ and $\varphi \in \mathcal{A}(G_2(\Q) \backslash G_2(\A))$, we have defined the function $\Theta(f,\varphi)$ on $\mathrm{PGSp}_6(\A)$ by
\begin{equation} \label{Eqthetadef}
\Theta(f,\varphi)(g)=\int_{G_2(\Q) \backslash G_2(\A)} \theta(f)(g'g)\varphi(g')dg'.
\end{equation}
For any $A \in J(\Q)$ and $f \in \Pi$, consider the Fourier coefficient
$$
\theta(f)_A(g)=\int_{N(\Q) \backslash N(\A)} \theta(f)(ng) \psi_A^{-1}(n)dn.
$$
We then have the Fourier expansion (cf. \cite[\S A.3]{harris-khare-thorne})
\begin{equation} \label{EqFourier}
\theta(f)(g)=\theta(f)_0(g)+\sum_{A \in \Omega(\Q)} \theta(f)_A(g),
\end{equation}
where $\Omega(\Q) \subset J(\Q)$ is the subset of rank $1$ elements. 

The following Lemma will be used in the proof of Proposition \ref{PropFCperiod}. Its proof is similar to the one of \cite[Lemma 4.6]{GrossSavin} but we give details for the convenience of the reader. Let $A_0$ be the representative of the open $G_2$-orbit on $\omega(\Q)$ given in Proposition \ref{orbits}.  Note that there is no harm in conjugating   $A_0 \in J(\Q)$ by an element of the Levi $\GL_3(\Q)$ of the Siegel parabolic of $\PGSp_6$. Thus, conjugating by $\mathrm{diag}(n,n,n)$, $A_0$ gets multiplied by $n^2$ and so we can assume that the entries $x,y,z$ of $A_0$ are in $\mathbb{O}(\Z)$. 

\begin{lemma} \label{technical}
Let $S$ denote a finite number of places containing 2 and $\infty$, and let $f = \otimes_v'f_v \in \Pi$ be such that, for $v \notin S$, we have $f_v=f_v^0$ where $f_v^0$ denotes the spherical vector normalized such that $f_v^0(A_0)=1$. Let $\Q_S=\prod_{v \in S} \Q_v$. If $g \in G_2(\A)$, we write $g=g_Sg^S$ where $g_S \in G_2(\Q_S)$ and $g^S \in \prod_{v \notin S} G_2(\Q_v)$. Then there exists a non-zero constant $c_{A_0}$ such that for every $g \in G_2(\A)$ we have
$$
\theta(f)_{A_0}(g)=c_{A_0} f_S(g_S^{-1}A_0)\prod_{v \notin S} \chi_v(g_v)
$$
where $f_S=\bigotimes_{v \in S} f_v$ and $\chi_v$ is the characteristic function of $U_D(\Z_v) \backslash G_2(\Z_v)$.
\end{lemma}

\begin{proof} By uniqueness of local functionals (\cite[Theorem A.4]{harris-khare-thorne}), there exists a non-zero scalar $c_{A_0}$ such that for any $g \in E_7(\A)$, we have $\theta(f)_{A_0}(g)=c_{A_0} (\Pi(g)f)(A_0)$. For $g \in G_2(\A)$ we have $(\Pi(g)f)(A_0)=f(g^{-1}A_0)$ where $g^{-1}A_0$ is the result of the natural action of $g^{-1}$ on the off diagonal entries of $A_0$. Hence $\theta(f)_{A_0}(g)=c_{A_0}f(g^{-1}A_0)=c_{A_0} \prod_v f_v(g_v^{-1}A_0)$ for $g \in G_2(\A)$. Let us prove that for any $p \notin S$, we have $f_p(g_p^{-1}A_0)=\chi_p(g_p)$. So let $g_p \in G_2(\Q_p)$ be such that $f^0_p(g_p^{-1}A_0) \neq 0$ and let $x',y',z'$ denote the off diagonal entries of $g_p^{-1}A_0$. According to \cite[Theorem A.5]{harris-khare-thorne} the spherical vector $f_p^0$ is supported in $J(\Z_p)$. Hence $x',y',z' \in \mathbb{O}(\Z_p)$. Consider $\mathbb{O}(\F_p)$ the split octonion algebra over $\F_p$. The projections of $(x,y,z)$ and $(x',y',z')$ to $\mathbb{O}(\F_p)$ are $G_2(\F_p)$-conjugated by the proof of Step 1 in Proposition \ref{orbits}, which is still valid over the base field $\F_p$ as long as $p \neq 2$. It follows from Hensel lemma that $(x,y,z)$ and $(x',y',z')$ are $G_2(\Z_p)$-conjugated. Therefore the function $g_p \mapsto f_p^0(g_p^{-1}A_0)$ is supported in $U_D(\Z_p) \backslash G_2(\Z_p) \subset U_D(\Q_p) \backslash G_2(\Q_p)$. Since $f_p^0$ is $G_2(\Z_p)$-invariant, for $g_p \in G_2(\Z_p)$ we have $f_p^0(g_p^{-1}A_0)=f_p^0(A_0)=1$. This completes the proof.
\end{proof}

\begin{proposition} \label{PropFCperiod}
Let $\sigma$ be a cuspidal automorphic representation of $G_2(\A)$ as in Theorem \ref{cuspidality} and let $\varphi \in \sigma^\vee$ be a cuspidal form. Then, the following conditions are equivalent
\begin{enumerate}
    \item $\Theta(f, \varphi)_{U_P, \psi_D}(1) \neq 0$ for some choice of $f$.
    \item $\varphi^{U_D}(g) \neq 0$ for some  $g \in G_2(\A)$.
\end{enumerate}
In particular, if any of the conditions holds then $\Theta(\sigma)$ is non-zero. 
\end{proposition} 

\begin{proof}
Recall first that, according to Proposition \ref{PropFC}, we have $\Theta(f, \varphi)_{U_P, \psi_D} \neq 0$ if and only if $\Theta(f, \varphi)_{U_3, \alpha} \neq 0$ for some   $\alpha \in {\rm Sym}^{{\rm rk } 2}(3)(\Q)$ with $\alpha \sim_{ L(\Q)} \alpha_D$. We write
\begin{eqnarray*}
\Theta(f,\varphi)_{U_3, \psi_D}(1) &=& \int_{U_3(\Q) \backslash U_3(\A)} \Theta(f,\varphi)(u) \psi_D^{-1}(u) du \\
&=& \int_{G_2(\Q) \backslash G_2(\A)} \int_{U_3(\Q) \backslash U_3(\A)} \sum_{A \in \Omega(\Q)} \theta(f)_A(ug) \varphi(g)  \psi_D^{-1}(u) du dg,
\end{eqnarray*}
where in the second equality we used the definition \eqref{Eqthetadef} of $\Theta(f, \varphi)$ and the Fourier expansion \eqref{EqFourier} of $\theta(f)$. Since $U_3 \subseteq N$, we have that $\theta(f)_A(ug) = \psi_A(u) \theta(f)_A(g)$ and \[ \int_{U_3(\Q) \backslash U_3(\A)} \psi_A(u) \psi_D^{-1}(u) = \begin{cases} \mathrm{vol}(U_3(\Q) \backslash U_3(\A)) & \text{ if }  \psi_D = \psi_A|_{U_3(\A)} \\ 0 & \text{ otherwise. }\end{cases} \] Hence we have
\begin{equation} \label{eqtmp1}
\Theta(f,\varphi)_{U_3, \psi_D}(1) = \mathrm{vol}(U_3(\Q) \backslash U_3(\A)) \int_{G_2(\Q) \backslash G_2(\A)} \sum_{A \in \omega(\Q)} \theta(f)_A(g) \varphi(g) dg.
\end{equation}

Let $(A_i)_i$ be the finite representatives of the orbits of the action of $G_2(\Q)$ on $\omega(\Q)$ as given by Proposition \ref{orbits}, and write $\mathrm{Stab}_{A_i}$ for the stabilizers of $A_i$ in $G_2$. The integral on the right hand side of \eqref{eqtmp1} becomes
\[ \sum_{i} \int_{G_2(\Q) \backslash G_2(\A)} \sum_{g' \in \mathrm{Stab}_{A_i}(\Q) \backslash G_2(\Q)} \theta(f)_{A_i}(g'g) \varphi(g) dg = \sum_{i} \int_{\mathrm{Stab_{A_i}(\Q)} \backslash G_2(\A)} \theta(f)_{A_i}(g) \varphi(g) dg.\]
Observe now that, by \cite[Theorem A.4]{harris-khare-thorne} we have $\theta(f)_{A_i}(g) = c_{A_i} f(g^{-1} A_i)$ for any $g \in G_2$. Hence, since $\mathrm{Stab}_{A_i}(\A)$ fixes the matrix $A_i$, we deduce that the function $g \mapsto \theta(f)_{A_i}(g)$ is left $\mathrm{Stab}_{A_i}(\A)$-invariant. Making an inner integration over $\mathrm{Stab}_{A_i}(\Q) \backslash \mathrm{Stab}_{A_i}(\A)$ in each term of the outer sum, we deduce that the above equals
\[ \sum_{i} \int_{\mathrm{Stab}_{A_i}(\A) \backslash G_2(\A)} \theta(f)_{A_i}(g) \varphi^{\mathrm{Stab}_{A_i}}(g) dg,\]
where $\varphi^{\mathrm{Stab}_{A_i}}(g)$ denotes the period of $\varphi$ over $\mathrm{Stab}_{A_i}(\Q) \backslash \mathrm{Stab}_{A_i}(\A)$. We now analyse two different possibilities. If $D$ is not a square in $\Q$, then, by $(2)$ of Proposition \ref{orbits},  $G_2(\Q)$ acts on $\omega(\Q)$ with two orbits, one closed and one open. Let $A_0,A_1$ denote representatives of these two orbits with stabilizers $\mathrm{Stab}_{A_0} = U_D$ and $\mathrm{Stab}_{A_1} = \SU_3^{D}$ in $G_2$. By the proof of Theorem \ref{cuspidality}, $\varphi^{\SU_3^{D}}(g)=0$, and hence the only surviving term is the one corresponding to the orbit represented by $A_0$. If $D$ is a square in $\Q$, then by $(1)$ of Proposition \ref{orbits}, $G_2(\Q)$ acts on the set $\omega(\Q)$ with four orbits, three closed and one open. Let $A_i, 0 \leq i \leq 3$ denote representatives of those orbits, with $A_0$ representing the open one. The corresponding stabilizers are are $U_D$, $\SL_3$, $\SL_2 V$ and its conjugate $\SL_2 \overline{V}$. By the proof of Theorem \ref{cuspidality}, we have $\varphi^{\SL_3}(g)=0$. By hypothesis  $\sigma$ (and $\sigma^\vee$) is not globally generic, hence Lemma \ref{lem:periodsl2n} implies that  $\varphi^{\SL_2 V}(g)=\varphi^{\SL_2 \overline{V}}(g)=0$. From this, we deduce that, for any $D$,
\begin{equation} \label{Eqtmp1}
\Theta(f,\varphi)_{U_3, \psi_D}(1) = \int_{U_D(\A) \backslash G_2(\A)} \theta(f)_{A_0}(g) \varphi^{U_D}(g) dg ,
\end{equation}
where $\varphi^{U_D}(g)$ is the constant term of $\varphi$ along $U_D$. This shows that if $\Theta(f,\varphi)_{U_3, \psi_D}(1) \neq 0$ then $\varphi^{U_D} \neq 0$ since the period appears as an inner integral of the Fourier coefficient.

We now show the converse, i.e. that if $\varphi^{U_D} \neq 0$ then, for some choice of $f \in \Pi$, the Fourier coefficient $\Theta(f,\varphi)_{U_3, \psi_D}$ does not vanish. Let $S$ be as in Lemma \ref{technical}. By enlarging $S$ if necessary, we can assume that the cusp form $\varphi$ is $G_2(\Z_v)$-invariant for all $v \notin S$. By Lemma \ref{technical}, the integral of \eqref{Eqtmp1} equals
$$
c_{A_0} \cdot \left( \int_{U_D(\Q_S) \backslash G_2(\Q_S)} f_S(g^{-1}A_0)\varphi^{U_D}(g) dg \right) \cdot \prod_{v \notin S} vol(U_D(\Z_v) \backslash G_2(\Z_v), dg_v).
$$
It remains to show that, when $\varphi^{U_D} \neq 0$ then for a good choice of $f$ at the places in $S$, the integral
$$
\int_{U_D(\Q_S) \backslash G_2(\Q_S)} f_S(g^{-1}A_0)\varphi^{U_D}(g)dg \neq 0.
$$
It follows from \cite[Theorem A.4]{harris-khare-thorne} that $f_S$ can be any smooth compactly supported function on $\Omega(\Q_S)$. Let $g_0 \in G_2(\Q_S)$ be such that $\varphi^{U_D}(g_0) \neq 0$. We can take a non-negative $f$ supported in a sufficiently small neighborhood of $g_0$ to ensure the non-vanishing of the integral. This finishes the proof of the proposition.
\end{proof}

\subsection{Non-vanishing of Fourier coefficients II}

The purpose of this section is to prove the following result.

\begin{theorem}\label{thm:thetaliftproperties}
Let $F$ denote a quadratic \'etale algebra and $\sigma = \sigma_\infty \otimes \sigma_f$ be a cuspidal automorphic representation of $G_2(\A)$ such that \begin{itemize}
    \item $\sigma_\infty$ is a non-generic discrete series with infinitesimal character $r \epsilon_1 + s \epsilon_2$.
    \item there exists a finite prime $p$ such that $\sigma_p$ is Steinberg.
    \item  The representation $\sigma$ supports Fourier coefficient associated to the cubic algebra $\Q \times F$.
\end{itemize}
The theta lift $\Theta(\sigma) = \otimes_v'\Theta(\sigma_v)$ is a non-zero cuspidal automorphic representation of $\PGSp_6(\A)$. Moveover, if $\pi$ denotes any non-zero irreducible subquotient of $\Theta(\sigma)$, then
\begin{itemize}
    \item $\pi_\infty$ is a discrete series of infinitesimal character $(r, \tfrac{1}{2}(r+s),\tfrac{1}{2}(r-s))$.
    \item $\pi_p$ is Steinberg.
    \item The representation $\pi$ supports a non-trivial Fourier coefficient of type $(4\, 2)$ associated to $F$.
\end{itemize}
\end{theorem}

\begin{remark}
As it will follow from the proof, the condition of $\sigma_\infty$ being not generic can be replaced by $\sigma$ not being locally generic, i.e. that there exists one local component of $\sigma_v$ of $\sigma$ which is not generic.
\end{remark}



Let us first fix some notations first. Recall from \S \ref{Section:dualpair} that the centralizer of $G_2$ in $M$ is $\GL_3$ and let
\[ U_0 = \left\{ \begin{pmatrix} 1 & a & b \\ 0 & 1 & 0 \\ 0 & 0 & 1 \end{pmatrix} \right\} \]
be the unipotent radical of its Borel subgroup of upper triangular matrices. Note that the unipotent subgroup $U_0 U_3$ is the unipotent radical of the parabolic subgroup $P$ of $\PGSp_6$ of Levi $\GL_2 \times \GL_1^2$ appearing in \S \ref{subsubsec:fc42}.

\begin{definition} \label{Definitionpsi0}
Define the character $\psi_0: U_0(\Q) \backslash U_0(\A) \to \C^\times$ by sending \[ \psi_0(u) = e(a). \]
\end{definition}
Note that $\psi_0 \psi_D$ is the character (simply denoted by $\psi_D$) on $U_P(\Q) \backslash U_P(\A)$ introduced in \S \ref{subsubsec:fc42}.

As explained in \S \ref{subsetorbits}, we view the space $\Omega$ of rank 1 elements in $J$ inside $\overline{N}$ so that $U_0$ acts on $\Omega$ via the natural right action of $\GL_3 \subseteq M$ on $\overline{N}$. Then, we let $U_0$ act on the left on $\omega$ and hence on the triples $(x,y,z)$ of off-diagonal terms by the rule
\begin{equation} \label{Equationactionu}
u^{-1} \cdot (x,y,z) = (x + ay +bz , y, z). 
\end{equation}

\subsubsection{The relation between $U_P$ and $U_H$}

In what follows, we relate the unipotent subgroup $U_0$ to the unipotent radical $U_H$ of the Heisenberg parabolic. Such a relation will be employed in Proposition \ref{prop:comparisonbetweenFC} to establish a relation between Fourier coefficients for the Heisenberg parabolic of $G_2$-cusp forms and Fourier coefficients of type $(4\,2)$ of their theta lifts.

Before stating our result, we make the following comments on the choice of representatives of the open orbits in Proposition \ref{orbits}. First, suppose that $D = d^2$, with $d \in \Q^\times$. There is no harm in assuming $ d \in \Z$. Recall that the stabilizer in $G_2$ of the vector $s_4 - t_4$ can be identified with $\SL_3 = \SL(\langle s_1, s_2, s_3 \rangle)$. Since the Heisenberg parabolic $H = L_H \cdot U_H$ is the stabilizer of the flag $\langle s_1, t_3 \rangle$, its unipotent radical $U_H$ contains $U_D = \mathrm{Stab}_{G_2}(A_0)$, where \[A_0 = A(d(s_4 - t_4), s_1 - t_3, d( s_1 + t_3)) \in J(\Z) \] is the representative of the open orbit of the action of $G_2$ on $\omega(\Q)$ as in Proposition \ref{orbits}. Moreover, $U_H / U_D$ is 2-dimensional and supported on the roots $a+b$ and $a+2b$. Let us now suppose that $D$ is not a square in $\Q^\times$.
The vector $x = s_2 + D t_2$ is a trace zero octonion of norm $-D$ and orthogonal to $t_3$. We choose the representative of the open orbit to be \[A_0 = A ( s_2 + D t_2, s_1, t_3) \in J(\Z). \]

\begin{lemma}\label{lemma:Hactionontriple} 
There is a natural surjection $ p: U_{H} \to U_{0}$ inducing an isomorphism
\[ U_{H} / U_D \to U_0. \]
\end{lemma}

\begin{proof}
By the description of the action in \eqref{Equationactionu} and the linear independence of the coordinates $(x,y,z)$ of the representative of the open orbit, one sees that $U_0$ acts freely on it. Hence, the result follows from showing that any element in $U_{H}$ acts on the triple $(x,y,z)$ as an element of $U_0$ and vice versa.

\textit{Case 1.} We start with the case where $D$ is a square in $\Q^\times$. 
The action of $U_0$ is given by
\begin{equation} \label{action0}
u^{-1} \cdot (d(s_4 - t_4), s_1 - t_3, d(s_1 + t_3)) =  (d(s_4 - t_4) + (a+db) s_1 +(db-a)t_3, s_1 - t_3, d(s_1 + t_3)).
\end{equation}
Since any element of $U_H$ fixes $s_1$ and $t_3$, it suffices to show that $U_H$ acts on $(s_4 - t_4)$ as an element of $U_0$. We verify this by studying the action of the Lie algebra. By \eqref{decompositionuh}, we know that the Lie algebra of $U_H$ is generated by the Lie algebra of the unipotent upper-triangular subgroup $U_D$ in $\SL_3$ and by the vectors $v_1$ and $\delta_3$. Using the explicit action of the action of the Lie algebra given in \S \ref{subsectLieG2}, one checks that 
\[ E_{ij} \cdot (s_4 - t_4) = 0, \]
\[ v_1 \cdot (s_4 - t_4) =  s_1, \]
\[ \delta_3 \cdot (s_4 - t_4) = t_3. \]
The above equations show that, for $u_1 = x_{a+b}(\lambda_1)$ and $u_2 = x_{a+2b}(\lambda_2)$ for some scalars $\lambda_1, \lambda_2$, we have \[ u_1 \cdot (d(s_4 - t_4)) = d(s_4 - t_4 + \lambda_1 s_1), \]
\[ u_2 \cdot (d(s_4 - t_4)) = d(s_4 - t_4 + \lambda_2 t_3).\]
This gives the desired isomorphism: if $u \in U_H /U_D$ is identified with the product of $x_{a+b}(\lambda_1) x_{a+2b}(\lambda_2)$, then, from Equation \eqref{action0}, we see that it gets sent to the element \[\begin{pmatrix} 1 &  d\tfrac{\lambda_1 - \lambda_2}{2} &   \tfrac{\lambda_1 + \lambda_2}{2} \\ 0 & 1 & 0 \\ 0 & 0 & 1 \end{pmatrix} \in U_0. \]

 \textit{Case 2.} We now suppose that $D$ is not a square in $\Q^\times$. Similarly to Case 1, it suffices to calculate $u\cdot (s_2 + D t_2)$ for any $u \in U_H$. As above, one checks that
\begin{align*}
E_{12} \cdot (s_2 + D t_2) &= s_1,\\  
E_{23} \cdot (s_2 + D t_2) &= D t_3,\\
E_{13} \cdot (s_2 + D t_2) &= 0,\\
v_1 \cdot (s_2 + D t_2)  &=  t_3,\\
\delta_3 \cdot (s_2 + D t_2)  &= - D s_1.
\end{align*}
This implies that if $u \in V_H = U_H / [U_H,U_H]$ equals to $x_a(\lambda_1)x_{a+b}(\lambda_2) x_{a+2b}(\lambda_3)x_{a+3b}(\lambda_4)$, then \[ u \cdot (s_2 + D t_2) = s_2 + D t_2 + (\lambda_1 - \lambda_3 D) s_1 + (\lambda_2 + D \lambda_4) t_3. \]
In particular, $U_D$ embeds into $U_H$ as the subgroup of matrices with $\lambda_1 = \lambda_3 D$ and $\lambda_2 = - \lambda_4 D$, and the map $p: U_H/U_D \to U_0$ sends $u$ to the element \[\begin{pmatrix} 1 &  \lambda_1 - \lambda_3 D &   \lambda_2 + \lambda_4 D \\ 0 & 1 & 0 \\ 0 & 0 & 1 \end{pmatrix} \in U_0. \]

\end{proof}

\begin{corollary}\label{cor:onthecharacters}
Under the isomorphism $p:U_H/U_D \to U_0$, we have \[ \psi_{H,D} = \psi_0 \circ p,\]
where $\psi_{H,D}:U_H(\Q) \backslash U_H(\A) \to \C^\times$ is the character corresponding to the \'etale cubic algebra $\Q \times \Q(\sqrt{D})$. 
\end{corollary}

\begin{proof}

We start with the case where $D$ is a square in $\Q^\times$. For simplicity, we can (and do) assume that $D=1$.  From Lemma \ref{lemma:Hactionontriple}, if $n \in U_H /U_D$ is identified with the product of $x_{a+b}(\lambda_1) x_{a+2b}(\lambda_2)$, it is sent via $p$ to \[\begin{pmatrix} 1 &  \tfrac{\lambda_1 - \lambda_2}{2} &   \tfrac{\lambda_1 + \lambda_2}{2} \\ 0 & 1 & 0 \\ 0 & 0 & 1 \end{pmatrix} \in U_0. \]  
Hence, the character $\psi_0 \circ p: U_H(\Q) \backslash U_H(\A) \to \C^\times$ sends $n \mapsto e(\tfrac{\lambda_1 - \lambda_2}{2})$. We now show that this corresponds to the character $\psi_{H,D}$ associated to $\Q \times \Q \times \Q$ as in \S \ref{subsub:FcforG2}. Recall that each character on $U_H(\Q) \backslash U_H(\A)$ is of the form $n \mapsto e( \langle w, \overline{n} \rangle)$, where $\overline{n}$ denotes the projection of $n$ to  $U_H/[U_H,U_H]$ and $w \in U_H(\Q)/[U_H(\Q),U_H(\Q)]$ corresponds to a binary  cubic form \[ f_w(x,y)= \lambda_1 x^3 + \lambda_2 x^2y + \lambda_3 xy^2 + \lambda_4 y^3, \] with $\lambda_i \in \Q$. Furthermore, as $\overline{n} = x_a(\lambda_1')x_{a+b}(\lambda_2'/3) x_{a+2b}(\lambda_3'/3)x_{a+3b}(\lambda_4')$  corresponds to
$f'(x,y)= \lambda_1' x^3 + \lambda_2' x^2y + \lambda_3' xy^2 + \lambda_4' y^3,$ the pairing is \[\langle w, \overline{n} \rangle = \lambda_1\lambda_4' - \tfrac{\lambda_2\lambda_3'}{3} + \tfrac{\lambda_3\lambda_2'}{3} - \lambda_4\lambda_1'. \]
Then, the character $\psi_0 \circ p$ corresponds to an element $w_D$ for which $\lambda_1,\lambda_4=0$ and $\lambda_2,\lambda_3 =  1/2$, namely the binary cubic polynomial $f_D(x,y) =  \tfrac{1}{2}(x^2y + xy^2)$. The latter is in the $L_H(\Q)$-orbit corresponding to the cubic algebra $\Q^3$. Indeed, if we let  $ g = \left(\begin{smallmatrix}  2 & \\ & -2 \end{smallmatrix} \right) \in L_H(\Q)$ act on $f_D$, we get \[ g \cdot f_D(x,y) = - \tfrac{1}{4} f_D(2x, -2y) =  \tfrac{1}{8}  (8x^2y - 8xy^2)=x^2y - xy^2, \]
which corresponds to $\Q^3$ by Example \ref{exampleofcubicpols}(1).

We now suppose that $D$ is not a square in $\Q^\times$. Then, by Lemma  \ref{lemma:Hactionontriple}, if  \[n \equiv x_a(\lambda_1)x_{a+b}(\lambda_2) x_{a+2b}(\lambda_3)x_{a+3b}(\lambda_4)\, \text{ mod } [U_H,U_H],\]  the character $\psi_0 \circ p: U_H(\Q) \backslash U_H(\A) \to \C^\times$ sends $n \mapsto e(\lambda_1 - \lambda_3 D)$. This character is associated to the binary cubic polynomial $f_D(x,y) = D x^2y - y^3$, which corresponds to $\Q \times \Q(\sqrt{D})$ by Example \ref{exampleofcubicpols}(2).
\end{proof}

\subsubsection{Comparison of Fourier coefficients}

The following proposition can be paired with Proposition \ref{PropFCperiod} to give three equivalent ways of proving that the theta lift of an automorphic representation of $G_2$ does not vanish.

\begin{proposition}\label{prop:comparisonbetweenFC} Let $\sigma$ be a cuspidal automorphic representation of $G_2(\A)$ as in Theorem \ref{cuspidality} and let $\varphi \in \sigma^\vee$ be a cuspidal form. The following conditions are equivalent
\begin{enumerate}
    \item $\Theta(f, \varphi)_{U_P, \psi_D}(1) \neq 0$ for some choice of $f \in \Pi$,
    \item $\varphi_{U_H, \psi_{H, D}}(g) \neq 0$ for some  $g \in G_2(\A)$.
\end{enumerate}
In particular, if any of the conditions holds then  $\Theta(\sigma)$ is non-zero.
\end{proposition}

\begin{proof}
Decomposing $U_P = U_0 U_3$, we have
\[ \Theta(f, \varphi)_{U_P, \psi_D}(1) = \int_{U_0(\Q) \backslash U_0(\A)} \int_{U_3(\Q) \backslash U_3(\A)} \Theta(f,\varphi)(u u') \psi_{D}^{-1}(u') \psi_{U_0}^{-1}(u)  du' d u. \]
As in the proof of Proposition \ref{PropFCperiod}, this equals
\[ \int_{U_0(\Q) \backslash U_0(\A)} \int_{U_D(\A) \backslash G_2(\A)} \theta(f)_{A_0}(u g) \varphi^{U_D}(g) \psi_{U_0}^{-1}(u) dg d u. \]
Exchanging integrals and  making an inner integration over $U_D(\A) \backslash U_H(\A)$, we get
\[ \int_{U_H(\A) \backslash G_2(\A)}\int_{U_D(\A) \backslash U_H(\A)}  \left ( \int_{U_0(\Q) \backslash U_0(\A)} \theta(f)_{A_0}(u u' g) \psi_{U_0}^{-1}(u) du \right )  \varphi^{U_D}(u'g)du' dg. \]

The isomorphism $p:U_H/U_D \cong U_0$ of Lemma \ref{lemma:Hactionontriple} induces \[ U_0(\Q) \backslash U_0(\A) \cong U_H(\Q) U_D(\A) \backslash U_H(\A) \]
such that $\psi_{H,D} = \psi_0 \circ p$ (cf. Corollary \ref{cor:onthecharacters}).
Thus, we can write the integral as
\[
\int_{U_H(\A) \backslash G_2(\A)} \int_{ U_D(\A) \backslash U_H(\A)} \left ( \int_{U_H(\Q)U_D(\A) \backslash U_H(\A)} \theta(f)_{A_0}(u u' g)  \psi_{H,D}^{-1}(u) du\right )  \varphi^{U_D}(u'g) du'  dg.
\]
Exchanging integrals, we have 

\[
\int_{U_H(\A) \backslash G_2(\A)} \int_{ U_H(\Q)U_D(\A) \backslash U_H(\A)} \left ( \int_{U_D(\A) \backslash U_H(\A)} \theta(f)_{A_0}(u u' g)  \varphi^{U_D}(u'g)  du'\right )  \psi_{H,D}^{-1}(u)du  dg 
\]
\[  =\int_{U_H  \backslash G_2(\A)}  \int_{U_H(\Q)U_D(\A) \backslash U_H(\A)}  \left ( \int_{U_H(\Q)U_D(\A) \backslash U_H(\A)}\sum_{\gamma \in U_D  \backslash U_H(\Q)} \theta(f)_{A_0}(u \gamma u' g)   \varphi^{U_D}(\gamma u'g) du'\right)\psi_{H,D}^{-1}(u)  du   dg
\]
\[  =\int_{U_H  \backslash G_2(\A)}  \int_{U_H(\Q)U_D(\A) \backslash U_H(\A)} \sum_{\gamma  } \left ( \int_{U_H(\Q)U_D(\A) \backslash U_H(\A)} \theta(f)_{A_0}(\gamma u  u' g)  \varphi^{U_D}(\gamma u'g) du' \right ) \psi_{H,D}^{-1}(u) du dg
\]
Changing variable $u' \mapsto u'' = \gamma  u u'= u \gamma u''$ in the inner integral, the above becomes
\[ \int_{U_H(\A) \backslash G_2(\A)}  \int_{U_H(\Q)U_D(\A) \backslash U_H(\A)}  \left ( \int_{U_D(\A) \backslash U_H(\A)} \theta(f)_{A_0}(u'' g)  \varphi^{U_D}(u^{-1} u'' g) du'' \right ) \psi_{H,D}^{-1}(u) du dg
\]
which, after rearranging the integrals, equals to 
\[ \int_{U_H(\A) \backslash G_2(\A)} \int_{U_D(\A) \backslash U_H(\A)} \theta(f)_{A_0}(u'' g) \varphi_{U_H, \psi_{H, D}}(u''g) du'' dg \]
\[ = \int_{U_D(\A) \backslash G_2(\A)} \theta(f)_{A_0}(g)\varphi_{U_H, \psi_{H,D}}(g)dg. \]
This shows that $(1)$ implies $(2)$. The proof of the converse is identical as the one given in Proposition \ref{PropFCperiod}. 

\end{proof}

We are now ready to give a proof of Theorem \ref{thm:thetaliftproperties}.
\begin{proof}[Proof of Theorem \ref{thm:thetaliftproperties}]
Let $\sigma$ be a cuspidal automorphic representation satisfying the hypotheses of the Theorem. We first apply Corollary \ref{corocusptheta} to deduce that $\Theta(\sigma)$ is cuspidal. Moreover, by Proposition \ref{prop:comparisonbetweenFC}, the theta lift supports a Fourier coefficient of type $(4\, 2)$ and, in particular, it is non-zero. Let $\pi$ be an irreducible subquotient of $\Theta(\sigma)$. Its component at $p$ is the Steinberg representation by the compatibility between the global and local correspondences (Proposition \ref{proplocalglobal})
and by Proposition \ref{propositiononlocalthetaprop}(2). We are now left to prove the statement on its archimedean component. As $\pi$ is unitary, $\pi_\infty$ is a unitarizable Harish-Chandra module by \cite[Theorem 4]{flath}. Moreover, as $\sigma_\infty$ is a discrete series with infinitesimal character $r \epsilon_1+s \epsilon_2$, it follows from the discussion in \cite[p. 204]{li} and by table 1 on \cite[p. 375]{li99} that $\Theta(\sigma_\infty)$ has infinitesimal character $(r, \frac{1}{2}(r+s),\frac{1}{2}(r-s))$, which is strongly regular in the sense of \cite[Definition 1.5]{salamanca-riba}. By another application of Proposition \ref{proplocalglobal}, $\pi_\infty$ is a subquotient of $\Theta(\sigma_\infty)$, hence has a strongly regular infinitesimal character. As a consequence, we can apply \cite[Theorem 1.8]{salamanca-riba} to deduce that $\pi_\infty$ is cohomological. By \cite[Corollary 2.8]{KretShin}, since $\pi_p$ is Steinberg and $\pi_\infty$ is cohomological, $\pi_\infty$ is a discrete series with infinitesimal character $(r, \frac{1}{2}(r+s),\frac{1}{2}(r-s))$.
\end{proof}

\begin{remark}
Suppose that $\sigma_\infty$ is a discrete series in $\mathcal{D}_{3,1}$. 
If $\Theta(\sigma_\infty)$ admits a unique irreducible quotient $\theta(\sigma_\infty)$, then by the results of \cite{li}, $\theta(\sigma_\infty)$ is a discrete series of Hodge type $(3,3)$. This implies that $\pi_\infty=\theta(\sigma_\infty)$ is a discrete series of Hodge type $(3,3)$. Although Howe duality conjecture for the pair $(G_2, \mathrm{PGSp}_6)$ is known at non-archimedean places \cite{Gan-SavinHoweduality}, the conjecture is still open at the archimedean place.
\end{remark}

\section{The cycle class formula and the standard motive for \texorpdfstring{$G_2$}{G2}}\label{sec:cycleclassformulaandstandardlvalues}
 
We conclude this article with the arithmetic applications described in the introduction.

\subsection{The relation between $L$-functions of $G_2$ and ${\rm PGSp}_6$}

The dual group of $G_2$ is $G_2(\C)$, which can be realized as the intersection ${\rm SO}_7(\C) \cap {\rm Spin}_7(\C)$. More precisely, we have the commutative diagram 

\begin{equation}\label{eq:diagramLF}
    \xymatrix{ G_2(\C) \ar@{^{(}->}[r] \ar@{^{(}->}[d]_{\zeta} \ar@/^2pc/[rr]^{{\rm Std}} & {\rm SO}_7(\C)\ar@{^{(}->}[r] \ar@{^{(}->}[d] & {\rm GL}_7(\C) \ar@{^{(}->}[d]  \\ {\rm Spin}_7(\C) \ar@{^{(}->}[r] \ar@/_2pc/[rr]^{{\rm Spin}}  & {\rm SO}_8(\C)\ar@{^{(}->}[r] & {\rm GL}_8(\C),  }
\end{equation} 
where ${\rm Std}: G_2(\C) \to \GL(V_7)$ is the standard representation given by trace zero octonions, ${\rm Spin} : {\rm Spin}_7(\C) \to \GL(V_8)$ is the 8-dimensional spin representation, while the embedding $\zeta$ is defined from the fact that the stabilizer in ${\rm Spin}_7(\C)$ of a generic vector of $V_8$ is isomorphic to $G_2(\C)$. From the commutative diagram, one immediately sees that  
\[ {V_8}_{|_{G_2}} = V_7 \oplus \mathbf{1}. \] 
In particular, if $\pi_\ell$ is an unramified smooth representation of $\PGSp_6(\Q_\ell)$ with Satake parameter $s_{\pi_\ell}$ belonging to $\zeta(G_2(\C))$, then 
\[L(s, \pi_\ell, {\rm Spin}) = L(s, \pi_\ell, {\rm Std}) \zeta_\ell(s), \]
where \[ L(s, \pi_\ell, {\rm Std}):=\frac{1}{{\rm det}( 1 - \ell^{-s} {\rm Std}(s_{\pi_\ell}))} \]
denotes the Euler factor at $\ell$ of the 7-dimensional standard $L$-function for $G_2$.

Let now $\pi$ be a cuspidal automorphic representation of $\PGSp_6$, which is unramified outside a finite set of places $S$ containing the archimedean place. As a special case of Langlands functoriality, one expects that if $L^S(s, \pi, {\rm Spin})$ has a simple pole at $s=1$, then $\pi$ is a functorial lift from either  $G_2$ or $G_2^c$, where recall that $G_2^c$ denotes the form of $G_2$  which is compact at $\infty$ and split at all finite places of $\Q$. We invite the reader to consult   \cite{Ginzburg-Jiang}, \cite{Gan-Gurevich}, \cite{Pollack-Shah}, and \cite{Gan-SavinExceptionalSW} for results in this direction. Moreover, the existence of a pole is usually related to the non-vanishing of a certain period. The following result \footnote{We point out that Proposition \ref{equivalences} is not really needed in the following (it is cited in the proof of Theorem 8.6, but only to know that the $L$-function of the lift from $G_2$ to $\PGSp_6$ has a pole at $s = 1$), but it might be of independent interest.}, summarizing and complementing known results in this direction, gives equivalence conditions for $\pi$ to be a weak functorial lift from $G_2$.

\begin{proposition} \label{equivalences} Suppose that $\pi$ satisfies the hypotheses \textbf{(DS)} and \textbf{(St)} of \S \ref{ss:cohomologyloc}. Then $\pi$ is tempered and the following statements are equivalent: \begin{enumerate}
        \item The partial $L$-function $L^S(s,\pi,{\rm Spin})$ has a simple pole at $s=1$,
        \item For almost all $\ell$, the Satake parameter $s_{\pi_\ell} \in \zeta(G_2(\C))$,
        \item  There exists a cuspidal automorphic representation $\sigma$ of either $G_2$ or $G_2^c$ such that $\pi$ is a weak functorial lift of $\sigma$.
    \end{enumerate}
Moreover, if $\pi$ supports a Fourier coefficient of rank 2 associated to the quadratic extension $F$ these conditions are equivalent to
\begin{enumerate}[resume]
    \item $\pi$ is $\H$-distinguished, with $\H= \GL_2 \boxtimes \GL_{2,F}^*$, i.e. that there exists a cusp form $\Psi$ in $\pi$ such that \[\int_{\Z(\A)\H(\Q) \backslash \H(\A)}\Psi(h) dh \ne 0. \] 
\end{enumerate}
If one of the first three conditions hold, the residue at $s=1$ of the partial $L$-function $L^S(s,\pi, {\rm Spin)}$ is given by
\[{ \rm Res}_{s=1} L^S(s,\pi, {\rm Spin)} =  L^S(1, \sigma, {\rm Std}) \prod_{\ell \in S}(1-\ell^{-1}).  \]
\end{proposition}

\begin{proof}
Since $\pi$ is cohomological and it is Steinberg at a finite place, we can apply \cite[Lemma 2.7]{KretShin} to deduce that $\pi$ is essentially tempered at all places. As $\pi$ has trivial central character, this is equivalent to being tempered. The equivalence between $(2)$ and $(3)$ and the implication $(1) \implies (3)$ follow from \cite[Theorem 1.1]{Gan-SavinExceptionalSW}. By \cite[Proposition 5.2]{Gan-Gurevich}, if $\pi$ is $\H$-distinguished then its big theta lift to $G_2$ is non-zero and is contained in the space of cusp forms on $G_2$. By the compatibility between the local and global theta correspondence, this implies that every local component $\pi_v$ appears in the local theta correspondence. When $v$ is a finite unramified place for $\pi$, \cite[Proposition 5.1]{Gan-Gurevich} implies that $s_{\pi_v} \in \zeta(G_2(\C))$. This shows $(4) \implies (3)$.

We next prove that $(1) \implies (4)$, for which we'll use the hypothesis on the existence of a Fourier coefficient of rank 2. By \cite[Theorem 2.7]{Pollack-Shah} (cf. Theorem \ref{Pollackshahongsp6}), given a cusp form $\Psi$ in $\pi$, there exists a cusp form $\tilde{\Psi}$ and a Schwartz-Bruhat function $\Phi$ such that  
\[ \mathcal{I}(\Phi, \tilde{\Psi}, s) = \mathcal{I}_\infty(\Phi, \Psi, s) L^S(s,\pi, {\rm Spin)} \]
By 
Proposition \ref{periodvsresidue}, taking residues at $s=1$ on both sides we have
    \[ 
    \frac{\widehat{\Phi}(0)}{2} \cdot \int_{\Z(\A)\H(\Q) \backslash \H(\A)}\tilde{\Psi}(h) dh = \mathrm{Res}_{s=1} \left( \mathcal{I}_\infty(\Phi, \Psi, s) L^S(s,\pi , {\rm Spin}) \right),
    \]
where $c>0$ is the constant of Lemma \ref{eisenstein}. We now use \cite[Proposition 12.1]{Gan-Gurevich} to deduce that there exists local data $\Phi_\infty$ and $\Psi_\infty$ such that $\mathcal{I}_\infty(\Phi, \Psi, 1) \ne 0$. Hence, up to modifying $\Psi$ and $\Phi$ at $\infty$, we obtain \[ \widehat{\Phi}(0)\cdot \int_{\Z(\A)\H(\Q) \backslash \H(\A)}\tilde{\Psi}(h) dh = C \cdot \mathrm{Res}_{s=1}  L^S(s,\pi , {\rm Spin}),
    \]
with $C$ a certain non-zero constant in $\C$. Note finally that we have the freedom to choose $\Phi$ such that $\widehat{\Phi}(0) \ne 0$. This follows from the fact that, given the two non-zero linear maps $l_1: \mathcal{S}(\A^2) \to \C, \Phi \mapsto \mathcal{I}_\infty(\Phi, \Psi, 1) $ and $l_2:\mathcal{S}(\A^2) \to \C, \Phi \mapsto \widehat{\Phi}(0) $, ${\rm ker}(l_1) \cup {\rm ker}(l_2) \ne \mathcal{S}(\A^2)$. This shows that if $L^S(s, \pi, \mathrm{Spin})$ has a simple pole and $\pi$ supports a Fourier coefficient of rank 2, then $\pi$ is $\H$-distinguished.

We finally show the implication $(2) \implies (1)$. The commutative diagram \eqref{eq:diagramLF} implies that 
\begin{align*}
    L^S(s, \pi, {\rm Spin}) = L^S(s, \pi, {\rm Std}) \zeta^S(s),
\end{align*}
where $L^S(s, \pi, {\rm Std})$ is the partial $L$-function of $\pi$ associated to the standard 7-dimensional representation of ${\rm Spin}_7$. By \cite[Theorem 1.1.1]{Labesse-Schwermer}, the restriction to ${\rm Sp}_6(\A)$ of $\pi$ contains a cuspidal automorphic representation $\pi^\flat$, such that (up to possibly enlarging $S$) \[L^S(s, \pi, {\rm Std}) = L^S(s, \pi^\flat, {\rm Std}). \]  
By \cite[Corollary 2.2 \& Lemma 2.3]{KretShin}, there exists a cuspidal automorphic representation $\pi^\sharp$ of $\GL_7(\A)$ such that \[  L^S(s, \pi^\flat, {\rm Std}) = L^S(s, \pi^\sharp),\]
where $L^S(s, \pi^\sharp)$ denotes the standard $L$-function of $\pi^\sharp$. We claim that $L^S(1, \pi^\sharp)\ne 0$. By \cite[Theorem (1.3)]{Jacquet-ShalikaNonvanishing}, $L(s, \pi^\sharp)\ne 0$ for any $s$ with ${\rm Re}(s)=1$. If we write
\[L^S(s, \pi^\sharp) = L(s, \pi^\sharp) \prod_{\ell \in S} L(s, \pi^\sharp_\ell)^{-1}, \]
then our claim follows from the fact that each $L(s, \pi^\sharp_\ell)$ has no pole at $s=1$ (cf. \cite[p. 317]{Rudnick-Sarnak}). 
This implies that  \[ L^S(1, \pi, {\rm Std})  \ne 0. \]
Thus, $L^S(s, \pi, {\rm Spin})$ has a simple pole at $s=1$. This proves that $(2)$ implies $(1)$.

If now we assume $(3)$, i.e. that $\pi$ is a weak functorial lift of $\sigma$, then (up to possibly enlarging $S$) \[ L^S(1, \sigma, {\rm Std}) =  L^S(1, \pi, {\rm Std})  \ne 0, \]
where the first equality is a consequence of the fact that the Satake parameters of $\sigma$ and $\pi$ agree almost everywhere. In particular \[{\rm Res}_{s=1}L^S(s, \pi, {\rm Spin}) = L^S(1, \sigma, {\rm Std}) {\rm Res}_{s=1} \zeta^S(s)= L^S(1, \sigma, {\rm Std}) \prod_{\ell \in S}(1-\ell^{-1}),  \]
showing the final claim.
\end{proof}

\begin{remark}
Let $\pi$ be as in Corollary \ref{CorollaryNonVanishingCycle}. Assuming \textbf{(St)}, $\pi$ is a weak functorial lift of $\sigma$ as in Proposition \ref{equivalences} and Theorem \ref{theoremcyclebetti1} reads as
\[ \langle \mathcal{Z}_{\H, \mathcal{H}}^{[\lambda, \mu]}, [\omega_\Psi] \rangle_\mathcal{H} = C \cdot \mathcal{I}_S(\Phi, \Psi^{[\lambda, \mu]}, 1) \cdot \prod_{\ell \in S} (1 - \ell^{-1}) \cdot  L^S(1, \sigma, \mathrm{Std}). \]
\end{remark}

\subsection{Galois representations of $G_2$-type}

The following result for the compact form of $G_2$ is shown in \cite[Theorem 11.1 and Corollary 11.3]{KretShin}. The same proof works for the split form of $G_2$ as long as one has some information on its lift to $\mathrm{PGSp}_6$, and we only sketch it for the convenience of the reader.

\begin{theorem}\label{thmonsplitGalois}
Let $\sigma$ be a cuspidal automorphic representation of $G_2(\A)$ or $G_2^c(\A)$ which lifts to a non-zero cuspidal automorphic representation $\pi$ of $\G(\A)$ such that
\begin{itemize}
    \item $\pi_\infty$ is cohomological,
    \item $\pi_p$ is the Steinberg representation at some finite prime $p$.
\end{itemize}
Then, for each prime $\ell$ and $\iota : \C \cong \overline{\Q}_\ell$, there exists a Galois representation $\rho_\sigma = \rho_{\sigma, \iota} : \mathrm{Gal}(\overline{\Q} / \Q) \to G_2(\overline{\Q}_\ell)$ such that
\begin{itemize}
    \item For every finite place $v \neq \ell$ where $\sigma$ is unramified, $\rho_\sigma$ is unramified at $v$. Moreover, the semisimple part of $\rho_\sigma({\rm Frob}_v)$ is conjugate to the Satake parameter $\iota(s_{\sigma_v})$ in $G_2(\overline{\Q}_\ell)$.
    \item $\rho_{\sigma_\ell}$ is de Rham, and it is crystalline if $\sigma$ is unramified at $\ell$.
    \item $\zeta \circ \rho_\sigma = \rho_\pi$, where $\pi$ is a theta lift of $\sigma$, and $\zeta: G_2(\C) \to \mathrm{Spin}_7(\C)$ is the embedding appearing in \eqref{eq:diagramLF}.
    \item the Zariski closure of the image of $\rho_\sigma$ maps onto either the image of a principal $\SL_2$ in $G_2$ or onto $G_2$.
\end{itemize}
\end{theorem}

\begin{proof}
    By \cite[Theorem A]{KretShin}, there exists a representation $\rho_\pi: \mathrm{Gal}(\overline{\Q} / \Q) \to {\rm Spin}_7(\overline{\Q}_\ell)$ attached to $\pi$. By the proof of \cite[Theorem 11.1]{KretShin}, one has that the image of $\rho_\pi$ is contained in $G_2(\overline{\Q}_\ell)$, and thus we have $\rho_\sigma$ such that $\zeta \circ \rho_\sigma = \rho_\pi$ for a suitable choice of embedding $\zeta: G_2(\C) \to \mathrm{Spin}_7(\C)$ fitting in the diagram \eqref{eq:diagramLF}. 
Hence, by \cite[Theorem A]{KretShin} and Proposition \ref{propositiononlocalthetaprop}(1), the representation $\rho_\sigma: \mathrm{Gal}(\overline{\Q} / \Q) \to G_2(\overline{\Q}_\ell)$ satisfies the desired first three properties.
Finally, by \cite[Theorem A $(v)$]{KretShin}, the Zariski closure of $\rho_\pi$ must map onto either a principal $\SL_2$ in $\rm{ SO}_7 \cap G_2$, or $G_2$.
\end{proof}

In the following Proposition, we describe several cases where Theorem \ref{thmonsplitGalois} applies. Before doing that, we need to introduce some notation. Let $\omega_1,\omega_2$ denote the two fundamental weights for $G_2$, where $\omega_1$ is the highest weight of the standard representation and $\omega_2$ of the 14-dimensional adjoint representation. According to our convention on the root system for $G_2$ in \S \ref{subsub:RootsandHeisenberg}, $\omega_1 = a+2b$ and $\omega_2 = 2a+3b$.

\begin{proposition} \label{prop:assumptions}
    Let $\sigma$ be a cuspidal automorphic representation of $G_2^{\circ}(\A)$ with $\circ \in \{ \emptyset, c\}$
    such that $\sigma_\infty$ is a discrete series of infinitesimal character $(r,s)$ where $r-3 \geq s-1 \geq 0$ and $r-s$ is even (if $\circ = c$ then $\sigma_\infty$ is the irreducible algebraic representation of $G_2^c(\R)$ of highest weight $(s-1) \omega_1 + \tfrac{1}{2}(r-s-2) \omega_2 $)  and $\sigma_p$ is Steinberg at some finite place $p$.
    Suppose that one of the following conditions holds.
    \begin{enumerate}
    \item We have $\circ = c$ and there exists $\alpha \in \sigma$ and a quaternion subalgebra $D$ of the non-split octonions $\mathbb{O}^c$ such that \[ P^C_\alpha := \int_{C(\Q) \backslash C(\A)} \alpha(v) d v \ne 0,   \]
   where $C$ is the centralizer of $D$ in $G_2^c$.
    \item We have $\circ = \emptyset$ and either $\sigma$ is globally generic or $\sigma_\infty$ is non-generic and $\sigma$ supports a Fourier coefficient of type $(4\,2)$ corresponding to $\Q \times F$, where $F$ is a real quadratic \'etale $\Q$-algebra.
    \end{enumerate}
Then there exists a non-trivial small theta lift $\pi$ of $\sigma$ to $\G(\A)$ which is a cuspidal automorphic representation, such that $\pi_\infty$ is a discrete series with infinitesimal character $(r, \frac{1}{2}(r+s), \frac{1}{2}(r-s))$ and $\pi_p$ is Steinberg.
\end{proposition}

\begin{proof}
    Let $\sigma$ be as in assumption (1). Since the Steinberg representation is generic, then by \cite[Corollary 4.9]{GrossSavin} the big theta lift of $\sigma$ to $\G(\A)$  has a non-trivial cuspidal irreducible subquotient $\pi$, which is unramified at almost all places. The infinitesimal character of its archimedean component is given in \cite[Theorem 3.5]{GrossSavin}, and by \ref{proplocalglobal} and by Proposition \ref{propositiononlocalthetaprop} we have that $\pi_p$ is Steinberg. Under the assumption (2), if $\sigma$ is globally generic, the result follows from \cite[Theorem 1.7]{harris-khare-thorne}, and if $\sigma_\infty$ is not generic and $\sigma$ supports a Fourier coefficient as in the statement, the result follows from Theorem \ref{thm:thetaliftproperties}.
\end{proof}

By construction, the composition of the Galois representation $\rho_\pi$ (and thus $\rho_\sigma$) with the Spin representation appears in $H^6_{\text{\'et}}(\Sh_{\G,\overline{\Q}}, \mathcal{V}^\lambda_\ell(3))$, where the latter denotes the direct limit of the cohomology at level $U$ in coefficients in the $\ell$-adic lisse sheaf associated to an irreducible algebraic representation $V^\lambda$ of $\G$, as $U$ varies. This direct limit is a smooth admissible $\overline{\Q}_\ell$-representation of $\G(\A_f)$, endowed with an action of $\mathrm{Gal}(\overline{\Q} / \Q)$ commuting with the one of $\G(\A_f)$. Let $\sigma$ and $\pi$ be as in the statement of Theorem \ref{thmonsplitGalois}. Choose an embedding of the rationality field $L$ of $\pi$ in $\overline{\Q}_\ell$. Then by Lemma \ref{lemmaonBettiisotypicpart}, the $\pi_f^\vee$-isotypic component of $H^6_{\text{\'et},!}(\Sh_{\G,\overline{\Q}}, \mathcal{V}^\lambda_\ell(3))$ is 8-dimensional $\overline{\Q}_\ell$-vector space, and we have 
\[
H^6_{\text{\'et},!}(\Sh_{\G,\overline{\Q}}, \mathcal{V}^\lambda_\ell(3))[\pi_f^\vee] = V_{{\rm Spin} \circ \rho_\pi} \otimes \pi_f^\vee 
= V_{{\rm Spin} \circ \zeta \circ \rho_\sigma} \otimes \pi_f^\vee.
\]
If the image of $\rho_\sigma$ is Zariski dense in $G_2(\overline{\Q}_\ell)$, we have ${\rm Spin} \circ \zeta \circ \rho_\sigma = {\rm Std} \circ \rho_\sigma \oplus \mathbf{1}$, where ${\rm Std} \circ \rho_\sigma$ is the irreducible ``standard'' Galois representation attached to $\sigma$. If not, by Theorem \ref{thmonsplitGalois}, the image of $\rho_\sigma$ is Zariski dense onto a principal $\xi: \SL_2(\overline{\Q}_\ell) \to G_2(\overline{\Q}_\ell)$. Then, the  branching law of \cite[(7.1)]{GrossMinusculePrincipal} gives that ${\rm Spin} \circ \zeta \circ \rho_\sigma= {\rm Sym}^6 \circ \rho_\sigma \oplus \mathbf{1}$, where $ {\rm Sym}^6 \circ \rho_\sigma$ is the irreducible symmetric sixth power Galois representation attached to $\sigma$. Denote by $M_\ell(\pi_f)$ the Galois representation $V_{{\rm Spin} \circ \rho_\pi}$ and let $M_\ell(\sigma_f)$ be either the Galois representation $V_{{\rm Std} \circ \rho_\sigma}$ or $V_{{\rm Sym}^6 \circ \rho_\sigma}$. Then, we have that $  M_\ell(\sigma_f)^{G_\Q}  = 0$ and  $ M_\ell(\pi_f)$ decomposes as the direct sum
\begin{equation} \label{EquationEtaleDecomposition}
M_\ell(\pi_f) = M_\ell(\sigma_f) \oplus \mathbf{1},
\end{equation}
where $\mathbf{1}$ denotes the one dimensional trivial representation.

\begin{remark}
In the case where $\rho_\sigma$ is not Zariski dense in $G_2(\overline{\Q}_\ell)$,  the Satake parameter $s_{\sigma_p} \in \xi(\SL_2(\C)) $ for any unramified prime $p$. By Langlands reciprocity principle, $\sigma$ should be the functorial lift of a cuspidal automorphic representation $\tau$ of $\mathrm{PGL}_2(\A)$, while $V_{{\rm Sym}^6 \circ \rho_\sigma}$ should be a geometric realization of the motive of the symmetric sixth power of $\tau$.  
\end{remark}

\subsection{On a question of Gross and Savin}

Tate conjecture predicts the existence of a cycle which gives rise to the trivial representation appearing in the decomposition of Equation \eqref{EquationEtaleDecomposition}. In \cite{GrossSavin}, Gross and Savin, inspired by local computations, conjectured that this cycle should come from a Hilbert modular 3-fold inside $\Sh_\G$. Theorem \ref{TheoGS} below supports this expectation for certain cuspidal automorphic representations $\sigma$ of $G_2$ and $G_2^c$.

Let $\sigma$ be a cuspidal automorphic representation of $G_2^{\circ}(\A)$ with $\circ \in \{ \emptyset, c\}$
such that $\sigma_\infty$ is a discrete series of infinitesimal character $(r,s)$ with $r-3 \geq s-1 \geq 0$ and $r-s$ even and $\sigma_p$ is Steinberg at some finite place $p$ and let $\pi$ be the small theta lift of $\sigma$ given by Proposition \ref{prop:assumptions}. Let $V^\lambda$ denote an irreducible algebraic representation of $\G$ of highest weight $\lambda= (r -3, \frac{1}{2}(r+s)-2, \frac{1}{2}(r-s)-1, 0)$. Note that $V^\lambda|_\H$ contains the trivial representation by Lemma \ref{lemma:branchinglaw}. Let $U \subset \G(\A_f)$ denote a neat compact open subgroup such that $\pi_f^U \neq 0$. For any $\frac{1}{2}(r+s)-2 \geq \mu \geq \frac{1}{2}(r-s)-1$, let
\[ \mathcal{Z}_{\H, \text{\'et}}^{[\lambda,\mu]} := \rm{cl}_{\text{\'et}}(\mathcal{Z}_{\H, \mathcal{M}}^{[\lambda,\mu]}) \in H^6_{\text{\'et}}(\Sh_\G(U)_{\overline{\Q}}, \mathcal{V}^\lambda_\ell(3))^{G_\Q}\] be the \'etale realization of the motivic class $\mathcal{Z}_{\H, \mathcal{M}}^{[\lambda,\mu]}$ (see Definition \ref{etaleclass}), where $\H = \GL_2 \boxtimes \GL_{2,F}^*$. Fix a vector $\Psi_f \in \pi_f^U$. By composing the projection to the $\pi_f^\vee$-isotypic component together with the projection given by the vector $\Psi_f$, we get
\[ \mathcal{Z}_{\H, \text{\'et}}^{\sigma} : = \Psi_f( {\rm pr}_{\pi^\vee}(\mathcal{Z}_{\H, \text{\'et}}^{[\lambda,\mu]})) \in ( M_\ell(\sigma_f) \oplus \mathbf{1})^{G_\Q}=\mathbf{1}. \]
By (the proof of) Lemma \ref{lemmaonBettiisotypicpart}, there exists a cuspidal automorphic representation $\pi^{3,3}=\pi_\infty^{3,3} \otimes \pi_f$ of $\G(\A)$ whose archimedean component is a discrete series of Hodge type $(3,3)$ with the same infinitesimal character of $\pi_\infty$ and whose non-archimedean part $\pi_f$ is the same as the one of $\pi$. Let $\Psi = \Psi_\infty \otimes \Psi_f$ be the cusp form in the space of $\pi^{3,3}$ such that $\Psi_\infty$ is a highest weight vector of the minimal $K_\infty$-type of $\pi_{\infty, 1}^{3,3} \subseteq \pi_\infty^{3,3} |_{\Sp_6(\R)}$. For any $\mu$ as above, recall that we denote $\Psi^{[\lambda,\mu]} = A^{[\lambda,\mu]} \cdot \Psi_\infty \otimes \Psi_f$,  where $A^{[\lambda,\mu]}$ is the operator that appeared in Proposition \ref{period}.

\begin{theorem} \label{TheoGS} Assume that the integral $\mathcal{I}_S(\Phi, \Psi^{[\lambda,\mu]} , 1)$ is non-zero for some Schwartz-Bruhat function $\Phi$. Then the class $ \mathcal{Z}_{\H, \text{\'et}}^{\sigma}$ generates the trivial sub-representation $\mathbf{1}$ of $M_\ell(\pi_f)$.
\end{theorem}

\begin{proof}
By the comparison theorem between \'etale and Betti cohomology \cite[Expos\'e XI, Theorem 4.4 (iii)]{sga}, Proposition \ref{equivalences} and Corollary \ref{CorollaryNonVanishingCycle}, we know that the projection   ${\rm pr}_{\pi^\vee} \mathcal{Z}_{\H, \text{\'et}}^{[\lambda, \mu]}$ to $M_\ell(\pi_f) \otimes (\pi_f^U)^\vee$ generates a one dimensional subspace, which is trivial for the action of the Galois group. As we have explained above, the image of $\rho_\sigma$ is either dense in $G_2(\overline{\Q}_\ell)$ or in $\SL_2(\overline{\Q}_\ell) \to  {\rm PGL}_2(\overline{\Q}_\ell) \hookrightarrow G_2(\overline{\Q}_\ell)$. In either case, the representation $M_\ell(\sigma_f)$ is irreducible and the trivial factor $\mathbf{1}$ in $M_\ell(\pi_f)$ is hence generated by the image of $ \mathcal{Z}_{\H, \text{\'et}}^{\sigma}$.  
\end{proof}

\bibliographystyle{amsalpha}
\bibliography{bibliography}

 \end{document}